\documentclass{article}
\usepackage{amssymb}
\usepackage{amsfonts}
\usepackage{url}
\usepackage{amsmath}
\usepackage{graphicx}

\newtheorem{theorem}{Theorem}

\newtheorem{corollary}{Corollary}

\newtheorem{definition}{Definition}

\newtheorem{lemma}{Lemma}

\newtheorem{proposition}{Proposition}
\newtheorem{remark}{Remark}

\newenvironment{proof}[1][Proof]{\textbf{#1.} }{\ \rule{0.5em}{0.5em}}

\begin{document}

\title {Compactness property of the linearized Boltzmann collision
operator for a mixture of monatomic and polyatomic species }

\author{Niclas Bernhoff \footnote{niclas.bernhoff@kau.se}\\Department of Mathematics and Computer Science\\ Karlstad University,  65188 Karlstad, Sweden }

\maketitle

\abstract{The linearized Boltzmann collision operator has
a central role in many important applications of the Boltzmann equation.
Recently some important classical properties of the linearized collision
operator for monatomic single species were extended to multicomponent
monatomic gases and polyatomic single species. For multicomponent polyatomic
gases, the case where the polyatomicity is modelled by a discrete internal
energy variable was considered lately. Here we considers the corresponding
case for a continuous internal energy variable. Compactness results, saying
that the linearized operator can be decomposed into a sum of a positive
multiplication operator, the collision frequency, and a compact operator,
bringing e.g., self-adjointness, is extended from the classical result for
monatomic single species, under reasonable assumptions on the collision
kernel. With a probabilistic formulation of the collision operator as a
starting point, the compactness property is shown by a decomposition, such
that the terms are, or at least are uniform limits of, Hilbert-Schmidt
integral operators and therefore are compact operators. Moreover, bounds on
- including coercivity of - the collision frequency are obtained for a hard
sphere like, as well as hard potentials with cutoff like, models, from which
Fredholmness of the linearized collision operator follows, as well as its
domain.}



\section{Introduction\label{S1}}

The Boltzmann equation is a fundamental equation of kinetic theory of gases,
e.g., for computations of the flow around a space shuttle in the upper
atmosphere during reentry \cite{BBBD-18}. Studies of the main properties of
the linearized collision operator are of great importance in gaining
increased knowledge about related problems, see, e.g., \cite{Cercignani-88}
and references therein. The linearized collision operator, obtained by
considering deviations of an equilibrium, or Maxwellian, distribution, can
in a natural way be written as a sum of a positive multiplication operator -
the collision frequency - and an integral operator $-K$. Compact properties
of the integral operator $K$ (for angular cut-off kernels) are extensively
studied for monatomic single species, see, e.g., \cite{Gr-63, Dr-75,
Cercignani-88, LS-10}, and more recently for monatomic multi-component
mixtures \cite{BGPS-13,Be-23a}. Extensions to polyatomic single species,
where the polyatomicity is modeled by either a discrete, or, a continuous
internal energy variable \cite{Be-23a,Be-23b} and polyatomic multicomponent
mixtures \cite{Be-23c}, where the polyatomicity is modeled by discrete
internal energy variables, have also been conducted. For models, assuming a
continuous internal energy variable, see also \cite{BBS-23} for the case of
molecules undergoing resonant collisions (for which internal energy and
kinetic energy, respectively, are conserved under collisions), and \cite%
{BST-22, BST-23} for diatomic and polyatomic gases, respectively - with more
restrictive assumptions on the collision kernels than in \cite{Be-23b}, but
also a more direct approach. The integral operator can be written as the sum
of a Hilbert-Schmidt integral operator and an approximately Hilbert-Schmidt
integral operator -\ which is a uniform limit of Hilbert-Schmidt integral
operators (cf. Lemma $\ref{LGD}$ in Section $\ref{PT1}$) \cite{Glassey}, and
so compactness of the integral operator $K$ can be obtained. In this work,
we extend the results of \cite{Be-23a,Be-23b} for monatomic multicomponent
mixtures and polyatomic single species, where the polyatomicity is modeled
by a continuous internal energy variable \cite{BDLP-94,GP-23}, to the case
of multicomponent mixtures of monatomic and/or polyatomic gases, where the
polyatomicity is modeled by a continuous internal energy variable \cite%
{DMS-05,BBBD-18}. To consider mixtures of monatomic and polyatomic molecules
are of highest relevance in, e.g., the upper atmosphere \cite{BBBD-18}.

Following the lines of \cite{Be-23a,Be-23b,Be-23c}, motivated by an approach
by Kogan in \cite[Sect. 2.8]{Kogan} for the monatomic single species case, a
probabilistic formulation of the collision operator is considered as the
starting point. With this approach, it is shown that the integral operator $%
K $ can be written as a sum of Hilbert-Schmidt integral operators and
operators, which are uniform limits of Hilbert-Schmidt integral operators -
and so compactness of the integral operator $K$ follows. The operator $K$ is
self-adjoint, as well as the collision frequency. Thus the linearized
collision operator as the sum of two self-adjoint operators whereof (at
least) one is bounded, is also self-adjoint.

For models corresponding to hard sphere models, as well as hard potentials
with cut off models, in the monatomic case, bounds on the collision
frequency are obtained. Here we also want to point out reference \cite{DL-23}%
, where the corresponding upper bound in \cite{Be-23b} for the single
species\ case was improved. Then the collision frequency is coercive and
becomes a Fredholm operator. The set of Fredholm operators is closed under
addition with compact operators. Therefore, also the linearized collision
operator becomes a Fredholm operator by the compactness of the integral
operator $K$. The Fredholm property is vital in the Chapman-Enskog process,
and the Fredholmness of the linearized operator supply with the Fredholm
property taken for granted in \cite{BBBD-18}, for those models. Note that
for monatomic species, the linearized operator is not Fredholm for soft
potential models, unlike for hard potential models. The domain of collision
frequency - and, hence, of the linearized collision operator as well -
follows directly by the obtained bounds

For hard sphere like models the linearized collision operator satisfies all
the properties of the general linear operator in the abstract half-space
problem considered in \cite{Be-23d}, and, hence, the existence results in 
\cite{Be-23d} apply.

The rest of the paper is organized as follows. In Section $\ref{S2}$, the
model considered is presented. The probabilistic formulation of the
collision operators considered and its relations to more classical
formulations \cite{DMS-05,BBBD-18} are accounted for in Section $\ref{S2.1}$%
.\ Some classical results for the collision operators in Section $\ref{S2.2}$
and the linearized collision operator in Section $\ref{S2.3}$ are reviewed.
Section $\ref{S3}$ is devoted to the main results of this paper, while the
main proofs are addressed in Sections $\ref{PT1}-\ref{PT2}$; a proof of
compactness of the integral operators $K$ is presented in Section $\ref{PT1}$%
, while a proof of the bounds on the collision frequency appears in Section $%
\ref{PT2}$.

\section{Model\label{S2}}

This section concerns the model considered. Probabilistic formulations of
the collision operators are considered, whose relations to more classical
formulations are accounted for. Known properties of the models and
corresponding linearized collision operators are also reviewed .

Consider a multicomponent mixture of $s$ species $a_{1},...,a_{s}$, with $%
s_{0}$ monatomic and $s_{1}:=s-s_{0}$ polyatomic species, and masses $%
m_{1},...,m_{s}$, respectively. Here the polyatomicity is modeled by a
continuous internal energy variable $I\in $ $\mathbb{R}_{+}$ \cite{BL-75}.

The distribution functions are of the form $f=\left( f_{1},...,f_{s}\right) $%
, where for $\alpha \in \left\{
1,...,s\right\}$ the component $f_{\alpha }=f_{\alpha }\left( t,\mathbf{x},\mathbf{Z}\right) $,
with 
\begin{equation}
\mathbf{Z=Z}_{\alpha }:=\left\{ 
\begin{array}{l}
\text{ \ \ }\boldsymbol{\xi }\text{ \ \ \ \ \ for }\alpha \in \left\{
1,...,s_{0}\right\} \text{ \ \ } \\ 
\left( \boldsymbol{\xi },I\right) \text{ \ \ for }\alpha \in \left\{
s_{0}+1,...,s\right\}%
\end{array}%
\right. \text{,}  \label{z1}
\end{equation}%
$\left\{ t,I\right\} \subset \mathbb{R}_{+}$, $\mathbf{x}=\left(
x,y,z\right) \in \mathbb{R}^{3}$, and $\boldsymbol{\xi }=\left( \xi _{x},\xi
_{y},\xi _{z}\right) \in \mathbb{R}^{3}$, is the distribution function for
species $a_{\alpha }$.

Moreover, consider the real Hilbert space 
\begin{equation*}
\mathcal{\mathfrak{h}}:=\left( L^{2}\left( d\boldsymbol{\xi \,}\right)
\right) ^{s_{0}}\times \left( L^{2}\left( d\boldsymbol{\xi \,}dI\right)
\right) ^{s_{1}},
\end{equation*}%
with inner product%
\begin{equation*}
\left( f,g\right) =\sum_{\alpha =1}^{s_{0}}\int_{\mathbb{R}^{3}}f_{\alpha
}g_{\alpha }\,d\boldsymbol{\xi \,}+\sum_{\alpha =s_{0}+1}^{s}\int_{\mathbb{R}%
^{3}\times \mathbb{R}_{+}}f_{\alpha }g_{\alpha }\,d\boldsymbol{\xi \,}dI%
\text{, }f,g\in \mathcal{\mathfrak{h}}\text{.}
\end{equation*}

The evolution of the distribution functions is (in the absence of external
forces) described by the (vector) Boltzmann equation%
\begin{equation}
\frac{\partial f}{\partial t}+\left( \boldsymbol{\xi }\cdot \nabla _{\mathbf{%
x}}\right) f=Q\left( f,f\right) \text{,}  \label{BE1}
\end{equation}%
where the (vector) collision operator $Q=\left( Q_{1},...,Q_{s}\right) $ is
a quadratic bilinear operator that accounts for the change of velocities and
internal energies of particles due to binary collisions (assuming that the
gas is rarefied, such that other collisions are negligible). Here the
component $Q_{\alpha }$ is the collision operator for species $a_{\alpha }$
for $\alpha \in \left\{ 1,...,s\right\} $.

A collision betweeng two particles of species $a_{\alpha }$ and $a_{\beta
} $, respectively, where $\left\{ \alpha ,\beta \right\} \subset \left\{ 1,...,s\right\} $, can be represented by two pre-collisional elements, each element
consisting of a microscopic velocity and possibly also an internal energy, $%
\mathbf{Z}$ and $\mathbf{Z}_{\ast }$, and two corresponding post-collisional
elements, $\mathbf{Z}^{\prime }$ and $\mathbf{Z}_{\ast }^{\prime }$. The
notation for pre- and post-collisional pairs may be interchanged as well.
Due to momentum and total energy conservation, the following relations have
to be satisfied by the elements%
\begin{equation}
m_{\alpha }\boldsymbol{\xi }+m_{\beta }\boldsymbol{\xi }_{\ast }=m_{\alpha }%
\boldsymbol{\xi }^{\prime }+m_{\beta }\boldsymbol{\xi }_{\ast }^{\prime }
\label{CIM}
\end{equation}%
and if $\left\{ \alpha ,\beta \right\} \subseteq \left\{ 1,...,s_{0}\right\} 
$%
\begin{equation}
m_{\alpha }\left\vert \boldsymbol{\xi }\right\vert ^{2}+m_{\beta }\left\vert 
\boldsymbol{\xi }_{\ast }\right\vert ^{2}=m_{\alpha }\left\vert \boldsymbol{%
\xi }^{\prime }\right\vert ^{2}+m_{\beta }\left\vert \boldsymbol{\xi }_{\ast
}^{\prime }\right\vert ^{2}\text{,}  \label{CIEa}
\end{equation}%
if $\alpha \in \left\{ 1,...,s_{0}\right\} $ and $\beta \in \left\{
s_{0}+1,...,s\right\} $%
\begin{equation}
m_{\alpha }\left\vert \boldsymbol{\xi }\right\vert ^{2}+m_{\beta }\left\vert 
\boldsymbol{\xi }_{\ast }\right\vert ^{2}+I_{\ast }=m_{\alpha }\left\vert 
\boldsymbol{\xi }^{\prime }\right\vert ^{2}+m_{\beta }\left\vert \boldsymbol{%
\xi }_{\ast }^{\prime }\right\vert ^{2}+I_{\ast }^{\prime }\text{,}
\label{CIEb}
\end{equation}%
if $\alpha \in \left\{ s_{0}+1,...,s\right\} $ and $\beta \in \left\{
1,...,s_{0}\right\} $%
\begin{equation}
m_{\alpha }\left\vert \boldsymbol{\xi }\right\vert ^{2}+m_{\beta }\left\vert 
\boldsymbol{\xi }_{\ast }\right\vert ^{2}+I=m_{\alpha }\left\vert 
\boldsymbol{\xi }^{\prime }\right\vert ^{2}+m_{\beta }\left\vert \boldsymbol{%
\xi }_{\ast }^{\prime }\right\vert ^{2}+I^{\prime }\text{,}  \label{CIEc}
\end{equation}%
while if $\left\{ \alpha ,\beta \right\} \subseteq \left\{
s_{0}+1,...,s\right\} $ 
\begin{equation}
m_{\alpha }\left\vert \boldsymbol{\xi }\right\vert ^{2}+m_{\beta }\left\vert 
\boldsymbol{\xi }_{\ast }\right\vert ^{2}+I+I_{\ast }=m_{\alpha }\left\vert 
\boldsymbol{\xi }^{\prime }\right\vert ^{2}+m_{\beta }\left\vert \boldsymbol{%
\xi }_{\ast }^{\prime }\right\vert ^{2}+I^{\prime }+I_{\ast }^{\prime }\text{%
.}  \label{CIEd}
\end{equation}

\begin{remark}
\label{R1}In an attempt to apply a unified approach, dependence on internal
energies will still, in general, be indicated also for monatomic species,
but in those cases the internal energies are constant (by technical reasons
assumed to be one), and any integration with respect to an internal energy
variable $I$ etc. for a monatomic species, conventionally, is assumed to
include a Dirac's delta function $\delta _{1}\left( I-1\right) $ etc. .
\end{remark}

\subsection{Collision operator\label{S2.1}}

The (vector) collision operator $Q=\left( Q_{1},...,Q_{s}\right) $ has
components that can be written in the following form 
\begin{eqnarray*}
&&Q_{\alpha }(f,f)\\
&=&\sum_{\beta =1}^{s}Q_{\alpha \beta }(f,f)=\sum_{\beta
=1}^{s}\int_{\left( \mathbb{R}^{3}\times \mathbb{R}_{+}\right)
^{3}}W_{\alpha \beta }\Lambda _{\alpha \beta }(f)\,d\boldsymbol{\xi }_{\ast
}d\boldsymbol{\xi }^{\prime }d\boldsymbol{\xi }_{\ast }^{\prime }dI_{\ast
}dI^{\prime }dI_{\ast }^{\prime }\text{,} \\
&&\text{where }\Lambda _{\alpha \beta }(f):=\frac{f_{\alpha }^{\prime
}f_{\beta \ast }^{\prime }}{\left( I^{\prime }\right) ^{\delta ^{\left(
\alpha \right) }/2-1}\left( I_{\ast }^{\prime }\right) ^{\delta ^{\left(
\beta \right) }/2-1}}-\frac{f_{\alpha }f_{\beta \ast }}{I^{\delta ^{\left(
\alpha \right) }/2-1}I_{\ast }^{\delta ^{\left( \beta \right) }/2-1}}\text{.}
\end{eqnarray*}%
Here and below the abbreviations%
\begin{equation}
f_{\alpha \ast }=f_{\alpha }\left( t,\mathbf{x},\mathbf{Z}_{\ast }\right) 
\text{, }f_{\alpha }^{\prime }=f_{\alpha }\left( t,\mathbf{x},\mathbf{Z}%
^{\prime }\right) \text{, and }f_{\alpha \ast }^{\prime }=f_{\alpha }\left(
t,\mathbf{x},\mathbf{Z}_{\ast }^{\prime }\right)  \label{a1}
\end{equation}%
where $\mathbf{Z}_{\ast }$, $\mathbf{Z}^{\prime }$, and $\mathbf{Z}_{\ast
}^{\prime }$, are defined as the natural extension of definition $\left( \ref%
{z1}\right) $, i.e. $\mathbf{Z_{\ast }=Z_{\ast \alpha }}=\left\{ 
\begin{array}{l}
\text{ \ \ }\boldsymbol{\xi }_{\ast }\text{ \ \ \ \ \ \ \ for }\alpha \in
\left\{ 1,...,s_{0}\right\} \text{ \ \ } \\ 
\left( \boldsymbol{\xi }_{\ast },I_{\ast }\right) \text{ \ \ for }\alpha \in
\left\{ s_{0}+1,...,s\right\}%
\end{array}%
\right. $ etc., \ are used for $\alpha \in \left\{ 1,...,s\right\} $.
Moreover, $\delta ^{\left( 1\right) }=...=\delta ^{\left( s_{0}\right) }=2$,
while $\delta ^{\left( \alpha \right) }$, with $\delta ^{\left( \alpha
\right) }\geq 2$, denote the number of internal degrees of freedom of the
species for $\alpha \in \left\{ s_{0}+1,...,s\right\} $.

The transition probabilities $W_{\alpha \beta }$ are of the form, cf. \cite%
{Be-23a,Be-23b},%
\begin{eqnarray}
&&W_{\alpha \beta }=W_{\alpha \beta }(\boldsymbol{\xi },\boldsymbol{\xi }%
_{\ast },I,I_{\ast }\left\vert \boldsymbol{\xi }^{\prime },\boldsymbol{\xi }%
_{\ast }^{\prime },I^{\prime },I_{\ast }^{\prime }\right. )  \notag \\
&=&\left( m_{\alpha }+m_{\beta }\right) ^{2}m_{\alpha }m_{\beta }\left(
I^{\prime }\right) ^{\delta ^{\left( \alpha \right) }/2-1}\left( I_{\ast
}^{\prime }\right) ^{\delta ^{\left( \beta \right) }/2-1}\sigma _{\alpha
\beta }^{\prime }\frac{\left\vert \mathbf{g}^{\prime }\right\vert }{%
\left\vert \mathbf{g}\right\vert }\widehat{\delta }_{1}\widehat{\delta }_{3}
\notag \\
&=&\left( m_{\alpha }+m_{\beta }\right) ^{2}m_{\alpha }m_{\beta }I^{\delta
^{\left( \alpha \right) }/2-1}I_{\ast }^{\delta ^{\left( \beta \right)
}/2-1}\sigma _{\alpha \beta }\frac{\left\vert \mathbf{g}\right\vert }{%
\left\vert \mathbf{g}^{\prime }\right\vert }\widehat{\delta }_{1}\widehat{%
\delta }_{3}\text{,}  \notag \\
&&\text{where }\sigma _{\alpha \beta }=\sigma _{\alpha \beta }\left(
\left\vert \mathbf{g}\right\vert ,\cos \theta ,I,I_{\ast },I^{\prime
},I_{\ast }^{\prime }\right) >0\text{ and }  \notag \\
&&\sigma _{\alpha \beta }^{\prime }=\sigma _{\alpha \beta }\left( \left\vert 
\mathbf{g}^{\prime }\right\vert ,\left\vert \cos \theta \right\vert
,I^{\prime },I_{\ast }^{\prime },I,I_{\ast }\right) >0\text{ a.e., with} 
\notag \\
&&\widehat{\delta }_{1}=\delta _{1}\left( \frac{1}{2}\left( m_{\alpha
}\left\vert \boldsymbol{\xi }\right\vert ^{2}+m_{\beta }\left\vert 
\boldsymbol{\xi }_{\ast }\right\vert ^{2}-m_{\alpha }\left\vert \boldsymbol{%
\xi }^{\prime }\right\vert ^{2}-m_{\beta }\left\vert \boldsymbol{\xi }_{\ast
}^{\prime }\right\vert ^{2}\right) -\Delta I\right) \text{,}  \notag \\
&&\widehat{\delta }_{3}=\delta _{3}\left( m_{\alpha }\boldsymbol{\xi }%
+m_{\beta }\boldsymbol{\xi }_{\ast }-m_{\alpha }\boldsymbol{\xi }^{\prime
}-m_{\beta }\boldsymbol{\xi }_{\ast }^{\prime }\right) \text{, }\cos \theta =%
\frac{\mathbf{g}\cdot \mathbf{g}^{\prime }}{\left\vert \mathbf{g}\right\vert
\left\vert \mathbf{g}^{\prime }\right\vert }\text{,}  \notag \\
&&\mathbf{g}=\boldsymbol{\xi }-\boldsymbol{\xi }_{\ast }\text{, }\mathbf{g}%
^{\prime }=\boldsymbol{\xi }^{\prime }-\boldsymbol{\xi }_{\ast }^{\prime }%
\text{, and }\Delta I=I^{\prime }+I_{\ast }^{\prime }-I-I_{\ast }\text{.}
\label{tp}
\end{eqnarray}%
Here $\delta _{3}$ and $\delta _{1}$ denote the Dirac's delta function in $%
\mathbb{R}^{3}$ and $\mathbb{R}$, respectively - $\widehat{\delta }_{1}$ and 
$\widehat{\delta }_{3}$ taking the conservation of momentum and total energy
into account. Note, cf. Remark $\ref{R1}$, that for $\alpha \in \left\{
1,...,s_{0}\right\} $ the scattering cross sections $\sigma _{\alpha \beta }$
are independent of $I$ and $I^{\prime }$, while correspondingly, for $\beta
\in \left\{ 1,...,s_{0}\right\} $ the scattering cross sections $\sigma
_{\alpha \beta }$ are independent of $I_{\ast }$ and $I_{\ast }^{\prime }$.

Furthermore, it is assumed that the scattering cross sections $\sigma
_{\alpha \beta }$, $\left\{ \alpha ,\beta \right\} \subseteq \left\{
1,...,s\right\} $, satisfy the microreversibility conditions%
\begin{eqnarray}
&&I^{\delta ^{\left( \alpha \right) }/2-1}I_{\ast }^{\delta ^{\left( \beta
\right) }/2-1}\left\vert \mathbf{g}\right\vert ^{2}\sigma _{\alpha \beta
}\left( \left\vert \mathbf{g}\right\vert ,\left\vert \cos \theta \right\vert
,I,I_{\ast },I^{\prime },I_{\ast }^{\prime }\right)  \notag \\
&=&\left( I^{\prime }\right) ^{\delta ^{\left( \alpha \right) }/2-1}\left(
I_{\ast }^{\prime }\right) ^{\delta ^{\left( \beta \right) }/2-1}\left\vert 
\mathbf{g}^{\prime }\right\vert ^{2}\sigma _{\alpha \beta }\left( \left\vert 
\mathbf{g}^{\prime }\right\vert ,\left\vert \cos \theta \right\vert
,I^{\prime },I_{\ast }^{\prime },I,I_{\ast }\right) \text{.}  \label{mr}
\end{eqnarray}%
Furthermore, to obtain invariance of change of particles in a collision, it
is assumed that the scattering cross sections $\sigma _{\alpha \beta }$, $%
\left\{ \alpha ,\beta \right\} \subseteq \left\{ 1,...,s\right\} $, satisfy
the symmetry relations%
\begin{equation}
\sigma _{\alpha \beta }\left( \left\vert \mathbf{g}\right\vert ,\cos \theta
,I,I_{\ast },I^{\prime },I_{\ast }^{\prime }\right) =\sigma _{\beta \alpha
}\left( \left\vert \mathbf{g}\right\vert ,\cos \theta ,I_{\ast },I,I_{\ast
}^{\prime },I^{\prime }\right) \text{,}  \label{sr}
\end{equation}%
while 
\begin{eqnarray}
\sigma _{\alpha \alpha } &=&\sigma _{\alpha \alpha }\left( \left\vert 
\mathbf{g}\right\vert ,\left\vert \cos \theta \right\vert ,I,I_{\ast
},I^{\prime },I_{\ast }^{\prime }\right) =\sigma _{\alpha \alpha }\left(
\left\vert \mathbf{g}\right\vert ,\left\vert \cos \theta \right\vert
,I_{\ast },I,I^{\prime },I_{\ast }^{\prime }\right)  \notag \\
&=&\sigma _{\alpha \alpha }\left( \left\vert \mathbf{g}\right\vert
,\left\vert \cos \theta \right\vert ,I_{\ast },I,I_{\ast }^{\prime
},I^{\prime }\right)  \label{sr1}
\end{eqnarray}%
The invariance under change of particles in a collision, which follows
directly by the definition of the transition probability $\left( \ref{tp}%
\right) $ and the symmetry relations $\left( \ref{sr}\right) ,\left( \ref%
{sr1}\right) $ for the collision frequency, and the microreversibility of
the collisions $\left( \ref{mr}\right) $, implies that the transition
probabilities $\left( \ref{tp}\right) $ satisfy the relations

\begin{eqnarray}
W_{\alpha \beta }(\boldsymbol{\xi },\boldsymbol{\xi }_{\ast },I,I_{\ast
}\left\vert \boldsymbol{\xi }^{\prime },\boldsymbol{\xi }_{\ast }^{\prime
},I^{\prime },I_{\ast }^{\prime }\right. ) &=&W_{\beta \alpha }(\boldsymbol{%
\xi }_{\ast },\boldsymbol{\xi },I_{\ast },I\left\vert \boldsymbol{\xi }%
_{\ast }^{\prime },\boldsymbol{\xi }^{\prime },I_{\ast }^{\prime },I^{\prime
}\right. )  \notag \\
W_{\alpha \beta }(\boldsymbol{\xi },\boldsymbol{\xi }_{\ast },I,I_{\ast
}\left\vert \boldsymbol{\xi }^{\prime },\boldsymbol{\xi }_{\ast }^{\prime
},I^{\prime },I_{\ast }^{\prime }\right. ) &=&W_{\alpha \beta }(\boldsymbol{%
\xi }^{\prime },\boldsymbol{\xi }_{\ast }^{\prime },I^{\prime },I_{\ast
}^{\prime }\left\vert \boldsymbol{\xi },\boldsymbol{\xi }_{\ast },I,I_{\ast
}\right. )  \notag \\
W_{\alpha \alpha }(\boldsymbol{\xi },\boldsymbol{\xi }_{\ast },I,I_{\ast
}\left\vert \boldsymbol{\xi }^{\prime },\boldsymbol{\xi }_{\ast }^{\prime
},I^{\prime },I_{\ast }^{\prime }\right. ) &=&W_{\alpha \alpha }(\boldsymbol{%
\xi },\boldsymbol{\xi }_{\ast },I,I_{\ast }\left\vert \boldsymbol{\xi }%
_{\ast }^{\prime },\boldsymbol{\xi }^{\prime },I_{\ast }^{\prime },I^{\prime
}\right. )\text{.}  \label{rel1}
\end{eqnarray}

Applying known properties of Dirac's delta function, the transition
probabilities may be transformed to 
\begin{eqnarray*}
&&W_{\alpha \beta }=W_{\alpha \beta }(\boldsymbol{\xi },\boldsymbol{\xi }%
_{\ast },I,I_{\ast }\left\vert \boldsymbol{\xi }^{\prime },\boldsymbol{\xi }%
_{\ast }^{\prime },I^{\prime },I_{\ast }^{\prime }\right. ) \\
&=&\left( m_{\alpha }+m_{\beta }\right) ^{2}m_{\alpha }m_{\beta }\left(
I^{\prime }\right) ^{\delta ^{\left( \alpha \right) }/2-1}\left( I_{\ast
}^{\prime }\right) ^{\delta ^{\left( \beta \right) }/2-1}\sigma _{\alpha
\beta }^{\prime }\frac{\left\vert \mathbf{g}^{\prime }\right\vert }{%
\left\vert \mathbf{g}\right\vert } \\
&&\times \delta _{3}\left( \left( m_{\alpha }+m_{\beta }\right) \left( 
\mathbf{G}_{\alpha \beta }-\mathbf{G}_{\alpha \beta }^{\prime }\right)
\right) \delta _{1}\left( \dfrac{\mu_{\alpha \beta}}{2}\left( \left\vert \mathbf{g}\right\vert ^{2}-\left\vert 
\mathbf{g}^{\prime }\right\vert ^{2}\right) -\Delta I\right) \\
&=&2\left( I^{\prime }\right) ^{\delta ^{\left( \alpha \right) }/2-1}\left(
I_{\ast }^{\prime }\right) ^{\delta ^{\left( \beta \right) }/2-1}\sigma
_{\alpha \beta }^{\prime }\frac{\left\vert \mathbf{g}^{\prime }\right\vert }{%
\left\vert \mathbf{g}\right\vert } \\
&&\times \delta _{3}\left( \mathbf{G}_{\alpha \beta }-\mathbf{G}_{\alpha
\beta }^{\prime }\right) \\
&=&\left( I^{\prime }\right) ^{\delta ^{\left( \alpha \right) }/2-1}\left(
I_{\ast }^{\prime }\right) ^{\delta ^{\left( \beta \right) }/2-1}\sigma
_{\alpha \beta }^{\prime }\frac{1}{\left\vert \mathbf{g}\right\vert }\mathbf{%
1}_{\mu_{\alpha \beta}\left\vert \mathbf{g}\right\vert ^{2}>2\Delta I}\delta _{3}\left( \mathbf{G}_{\alpha \beta }-\mathbf{G}_{\alpha
\beta }^{\prime }\right)\\
&&\times \delta
_{1}\left( \sqrt{\left\vert \mathbf{g}\right\vert ^{2}-\frac{2\Delta I}{%
\mu_{\alpha \beta}}}-\left\vert \mathbf{g}^{\prime }\right\vert \right)  \\
&=&I^{\delta ^{\left( \alpha \right) }/2-1}I_{\ast }^{\delta ^{\left( \beta
\right) }/2-1}\sigma _{\alpha \beta }\frac{\left\vert \mathbf{g}\right\vert 
}{\left\vert \mathbf{g}^{\prime }\right\vert ^{2}}\mathbf{1}_{\mu_{\alpha \beta}\left\vert 
\mathbf{g}\right\vert ^{2}>2\Delta I}\delta _{1}\left( \sqrt{%
\left\vert \mathbf{g}\right\vert ^{2}-\frac{2\Delta I}{%
\mu_{\alpha \beta}}}-\left\vert 
\mathbf{g}^{\prime }\right\vert \right) \\
&&\times \delta _{3}\left( \mathbf{G}_{\alpha \beta }-\mathbf{G}_{\alpha
\beta }^{\prime }\right) \text{, with } \\
\text{ } &&\mathbf{G}_{\alpha \beta }=\frac{m_{\alpha }\boldsymbol{\xi }%
+m_{\beta }\boldsymbol{\xi }_{\ast }}{m_{\alpha }+m_{\beta }}\text{, }%
\mathbf{G}_{\alpha \beta }^{\prime }=\frac{m_{\alpha }\boldsymbol{\xi }%
^{\prime }+m_{\beta }\boldsymbol{\xi }_{\ast }^{\prime }}{m_{\alpha
}+m_{\beta }}\text{, and } \mu_{\alpha \beta}=\dfrac{m_{\alpha }m_{\beta }}{\left( m_{\alpha
}+m_{\beta }\right) }\text{.}
\end{eqnarray*}

\begin{remark}
Note that%
\begin{equation*}
\delta _{1}\left( \dfrac{\mu_{\alpha \beta}}{2}\left( \left\vert \mathbf{g}\right\vert ^{2}-\left\vert \mathbf{g}%
^{\prime }\right\vert ^{2}\right) -\Delta I\right) =\delta _{1}\left(
E_{\alpha \beta }-E_{\alpha \beta }^{\prime }\right) ,
\end{equation*}%
where $E_{\alpha \beta }=\dfrac{\mu_{\alpha \beta}}{2}\left\vert \mathbf{g}\right\vert ^{2}+I+I_{\ast }\ $and 
$E_{\alpha \beta }^{\prime }=\dfrac{\mu_{\alpha \beta}}{2}\left\vert \mathbf{g}^{\prime }\right\vert
^{2}+I^{\prime }+I_{\ast }^{\prime }$.
\end{remark}

Observe that, by a series of change of variables:\newline
$\left\{ \boldsymbol{\xi }^{\prime },\boldsymbol{\xi }_{\ast }^{\prime
}\right\} \rightarrow \!\left\{ \mathbf{g}^{\prime }=\boldsymbol{\xi }%
^{\prime }-\boldsymbol{\xi }_{\ast }^{\prime }\text{,}\mathbf{G}_{\alpha
\beta }^{\prime }=\dfrac{m_{\alpha }\boldsymbol{\xi }^{\prime }+m_{\beta }%
\boldsymbol{\xi }_{\ast }^{\prime }}{m_{\alpha }+m_{\beta }}\!\right\} $,
followed by a change to spherical\ coordinates $\left\{ \mathbf{g}^{\prime
}\right\} \rightarrow \left\{ \left\vert \mathbf{g}^{\prime }\right\vert ,%
\boldsymbol{\omega \,}=\dfrac{\mathbf{g}^{\prime }}{\left\vert \mathbf{g}%
^{\prime }\right\vert }\right\} $, then, if $\beta \in \left\{
s_{0}+1,...,s\right\} $,  the change 
$\left\{ \left\vert \mathbf{g}^{\prime }\right\vert ,I_{\ast }^{\prime
}\right\} \rightarrow \left\{ R=\dfrac{\mu_{\alpha \beta}}{2}\dfrac{\left\vert \mathbf{g}^{\prime
}\right\vert ^{2}}{E_{\alpha \beta }},E_{\alpha \beta }^{\prime }=\dfrac{%
\mu_{\alpha \beta}}{2}\left\vert 
\mathbf{g}^{\prime }\right\vert ^{2}+I^{\prime }+I_{\ast }^{\prime }\right\} 
$, and finally, if also $\alpha \in \left\{ s_{0}+1,...,s\right\} $, the change $\left\{
I^{\prime }\right\} \rightarrow \left\{ r=\dfrac{I^{\prime }}{\left(
1-R\right) E_{\alpha \beta }}\right\} $%
\begin{eqnarray}
&&d\boldsymbol{\xi }^{\prime }d\boldsymbol{\xi }_{\ast }^{\prime }dI^{\prime
}dI_{\ast }^{\prime }=d\mathbf{G}_{\alpha \beta }^{\prime }d\mathbf{g}%
^{\prime }dI^{\prime }dI_{\ast }^{\prime }  \notag \\
&=&\left\vert \mathbf{g}^{\prime }\right\vert ^{2}d\mathbf{G}_{\alpha \beta
}^{\prime }d\left\vert \mathbf{g}^{\prime }\right\vert d\boldsymbol{\omega \,%
}dI^{\prime }dI_{\ast }^{\prime }  \notag \\
&=&\sqrt{2}\left( \frac{E_{\alpha \beta }}{\mu_{\alpha \beta }}%
\right) ^{3/2}R^{1/2}dRd\boldsymbol{\omega \,}d%
\mathbf{G}_{\alpha \beta }^{\prime }dI^{\prime }dE_{\alpha \beta }^{\prime }
\notag \\
&=& \frac{\sqrt{2}}{\mu_{\alpha \beta}^{3/2}}E_{\alpha \beta }^{5/2} (1-R)R^{1/2}drdRd\boldsymbol{\omega \,}d%
\mathbf{G}_{\alpha \beta }^{\prime }dE_{\alpha \beta }^{\prime }\text{.}
\label{df1}
\end{eqnarray}%
Then for two monatomic species, with $\left\{ \alpha ,\beta \right\} $ $%
\subset \left\{ 1,...,s_{0}\right\} $, (mono/mono-case) 
\begin{eqnarray*}
&&Q_{\alpha \beta }(f,f) \\
&=&\int_{\left( \mathbb{R}^{3}\right) ^{2}\times \mathbb{R}_{+}\mathbb{%
\times S}^{2}}W_{\alpha \beta }\left\vert \mathbf{g}^{\prime }\right\vert
^{2}\left( f_{\alpha }^{\prime }f_{\beta \ast }^{\prime }-f_{\alpha
}f_{\beta \ast }\right) \,d\boldsymbol{\xi }_{\ast }d\mathbf{G}_{\alpha
\beta }^{\prime }d\left\vert \mathbf{g}^{\prime }\right\vert d\boldsymbol{%
\omega } \\
&=&\int_{\mathbb{R}^{3}\mathbb{\times S}^{2}}B_{0\alpha \beta }\left(
f_{\alpha }^{\prime }f_{\beta \ast }^{\prime }-f_{\alpha }f_{\beta \ast
}\right) \,d\boldsymbol{\xi }_{\ast }d\boldsymbol{\omega }\text{, with }%
B_{0\alpha \beta }=\sigma _{\alpha \beta }\left\vert \mathbf{g}\right\vert 
\text{,}
\end{eqnarray*}%
or, for a monatomic species, with $\alpha \in \left\{ 1,...,s_{0}\right\} $,
and a polyatomic species, with $\beta \in \left\{ s_{0}+1,...,s\right\} $,
(mono/poly-case),%
\begin{eqnarray*}
&&Q_{\alpha \beta }(f,f) \\
&=&\int_{\left( \mathbb{R}^{3}\right) ^{2}\times \left( \mathbb{R}%
_{+}\right) ^{2}\times \lbrack 0,1]\mathbb{\times S}^{2}}W_{\alpha \beta }%
\sqrt{2}\left( \frac{E_{\alpha \beta }}{\mu_{\alpha \beta }}%
\right) ^{3/2}\left( \frac{f_{\alpha }^{\prime }f_{\beta \ast }^{\prime }}{\left(
I_{\ast }^{\prime }\right) ^{\delta ^{\left( \beta \right) }/2-1}}-\frac{%
f_{\alpha }f_{\beta \ast }}{I_{\ast }^{\delta ^{\left( \beta \right) }/2-1}}%
\right) \\
&&\times R^{1/2}d\boldsymbol{\xi }_{\ast }d\mathbf{G}%
_{\alpha \beta }^{\prime }dRd\boldsymbol{\omega \,}dI_{\ast }dE_{\alpha
\beta }^{\prime } \\
&=&\int_{\mathbb{R}^{3}\times \mathbb{R}_{+}\times \lbrack 0,1]\mathbb{%
\times S}^{2}}B_{1\alpha \beta }\left( \frac{f_{\alpha }^{\prime }f_{\beta
\ast }^{\prime }}{\left( I_{\ast }^{\prime }\right) ^{\delta ^{\left( \beta
\right) }/2-1}}-\frac{f_{\alpha }f_{\beta \ast }}{I_{\ast }^{\delta ^{\left(
\beta \right) }/2-1}}\right) \left( 1-R\right) ^{\delta ^{\left( \beta
\right) }/2-1} \\
&&\times R^{1/2}I_{\ast }^{\delta ^{\left( \beta \right) }/2-1}d\boldsymbol{\xi }%
_{\ast }dRd\boldsymbol{\omega \,}dI_{\ast }\text{, with} \\
&&B_{1\alpha \beta }=\frac{\sigma _{\alpha \beta }\sqrt{2/\mu_{\alpha \beta }}\left\vert 
\mathbf{g}\right\vert }{\left( I_{\ast }^{\prime }\right) ^{\delta ^{\left(
\beta \right) }/2-1}\sqrt{\left\vert \mathbf{g}\right\vert ^{2}-\dfrac{2\Delta I}{%
\mu_{\alpha \beta}}}}E_{\alpha \beta }^{\left( \delta ^{\left( \beta \right)
}+1\right) /2}=\frac{\sigma _{\alpha \beta }\left\vert \mathbf{g}\right\vert
E_{\alpha \beta }}{R^{1/2}(1-R)^{\delta ^{\left( \beta \right) }/2-1}}
\end{eqnarray*}%
or, for a polyatomic species, with $\alpha \in \left\{ s_{0}+1,...,s\right\} 
$, and a monatomic species, with $\beta \in \left\{ 1,...,s_{0}\right\} $,
(poly/mono-case),%
\begin{eqnarray*}
&&Q_{\alpha \beta }(f,f) \\
&=&\int_{\left( \mathbb{R}^{3}\right) ^{2}\times \mathbb{R}_{+}\times
\lbrack 0,1]\mathbb{\times S}^{2}}W_{\alpha \beta }\sqrt{2}\left( \frac{E_{\alpha \beta }}{\mu_{\alpha \beta }}%
\right) ^{3/2}\left( \frac{%
f_{\alpha }^{\prime }f_{\beta \ast }^{\prime }}{\left( I^{\prime }\right)
^{\delta ^{\left( \alpha \right) }/2-1}}-\frac{f_{\alpha }f_{\beta \ast }}{%
I^{\delta ^{\left( \alpha \right) }/2-1}}\right) \\
&&\times R^{1/2}\,d\boldsymbol{\xi }_{\ast }d\mathbf{G%
}_{\alpha \beta }^{\prime }dRd\boldsymbol{\omega \,}dE_{\alpha \beta
}^{\prime } \\
&=&\int_{\mathbb{R}^{3}\times \lbrack 0,1]\mathbb{\times S}^{2}}B_{1\beta
\alpha }\left( \frac{f_{\alpha }^{\prime }f_{\beta \ast }^{\prime }}{\left(
I^{\prime }\right) ^{\delta ^{\left( \alpha \right) }/2-1}}-\frac{f_{\alpha
}f_{\beta \ast }}{I^{\delta ^{\left( \alpha \right) }/2-1}}\right) \left(
1-R\right) ^{\delta ^{\left( \alpha \right) }/2-1}R^{1/2} \\
&&\times I^{\delta ^{\left( \alpha \right) }/2-1}d\boldsymbol{\xi }_{\ast
}dRd\boldsymbol{\omega }
\end{eqnarray*}%
and, finally, for two polyatomic species, with $\left\{ \alpha ,\beta
\right\} $ $\subset \left\{ s_{0}+1,...,s\right\} $, (poly/poly-case) 
\begin{eqnarray*}
&&Q_{\alpha \beta }(f,f) \\
&=&\int\limits_{\left( \mathbb{R}^{3}\right) ^{2}\times \left( \mathbb{R}%
_{+}\right) ^{2}\times \lbrack 0,1]^{2}\mathbb{\times S}^{2}} \left( \frac{f_{\alpha }^{\prime }f_{\beta \ast }^{\prime }}{\left(
I^{\prime }\right) ^{\delta ^{\left( \alpha \right) }/2-1}\left( I_{\ast
}^{\prime }\right) ^{\delta ^{\left( \beta \right) }/2-1}}-\frac{f_{\alpha
}f_{\beta \ast }}{I^{\delta ^{\left( \alpha \right) }/2-1}I_{\ast }^{\delta
^{\left( \beta \right) }/2-1}}\right) \\
&&\times W_{\alpha \beta
}\frac{\sqrt{2}}{\mu_{\alpha \beta}^{3/2}}E_{\alpha \beta }^{5/2}(1-R)R^{1/2}d\boldsymbol{\xi }_{\ast
}drdRd\boldsymbol{\omega \,}d\mathbf{G}_{\alpha \beta }^{\prime }dE_{\alpha
\beta }^{\prime }dI_{\ast } \\
&=&\int\limits_{\mathbb{R}^{3}\times \mathbb{R}_{+}\times \lbrack 0,1]^{2}\mathbb{%
\times S}^{2}}\left( \frac{f_{\alpha }^{\prime }f_{\beta
\ast }^{\prime }}{\left( I^{\prime }\right) ^{\delta ^{\left( \alpha \right)
}/2-1}\left( I_{\ast }^{\prime }\right) ^{\delta ^{\left( \beta \right)
}/2-1}}-\frac{f_{\alpha }f_{\beta \ast }}{I^{\delta ^{\left( \alpha \right)
}/2-1}I_{\ast }^{\delta ^{\left( \beta \right) }/2-1}}\right) \\
&&\times B_{2\alpha \beta }r^{\delta ^{\left( \alpha \right) }/2-1}\left( 1-r\right) ^{\delta
^{\left( \beta \right) }/2-1}(1-R)^{\left( \delta ^{\left( \alpha \right)
}+\delta ^{\left( \beta \right) }\right) /2-1}R^{1/2} \\
&&\times I^{\delta ^{\left( \alpha \right) }/2-1}I_{\ast }^{\delta ^{\left(
\beta \right) }/2-1}d\boldsymbol{\xi }_{\ast }drdRd\boldsymbol{\omega \,}%
dI_{\ast }\text{, } \\
&&\text{with }B_{2\alpha \beta }=\sigma _{\alpha \beta }\frac{\sqrt{2/\mu_{\alpha \beta }}\left\vert 
\mathbf{g}\right\vert E_{\alpha \beta }^{\left( \delta ^{\left( \alpha
\right) }+\delta ^{\left( \beta \right) }+1\right) /2}}{\left( I^{\prime
}\right) ^{\delta ^{\left( \alpha \right) }/2-1}\left( I_{\ast }^{\prime
}\right) ^{\delta ^{\left( \beta \right) }/2-1}\sqrt{\left\vert \mathbf{g}%
\right\vert ^{2}-\dfrac{2\Delta I}{\mu_{\alpha \beta}}}} \\
&=&\frac{\sigma _{\alpha \beta }\left\vert \mathbf{g}\right\vert E_{\alpha
\beta }^{2}}{r^{\delta ^{\left( \alpha \right) }/2-1}\left( 1-r\right)
^{\delta ^{\left( \beta \right) }/2-1}(1-R)^{\left( \delta ^{\left( \alpha
\right) }+\delta ^{\left( \beta \right) }\right) /2-2}R^{1/2}}
\end{eqnarray*}%
where%
\begin{equation*}
\left\{ 
\begin{array}{l}
\boldsymbol{\xi }^{\prime }=\mathbf{G}_{\alpha \beta }+\omega \dfrac{%
m_{\beta }}{m_{\alpha }+m_{\beta }}\sqrt{\left\vert \mathbf{g}\right\vert
^{2}-\dfrac{2\Delta I}{\mu_{\alpha \beta}}}\medskip
\\ 
\boldsymbol{\xi }_{\ast }^{\prime }=\mathbf{G}_{\alpha \beta }-\omega \dfrac{%
m_{\alpha }}{m_{\alpha }+m_{\beta }}\sqrt{\left\vert \mathbf{g}\right\vert
^{2}-\dfrac{2\Delta I}{\mu_{\alpha \beta}}}%
\end{array}%
\right. \text{, }\omega \in S^{2}\text{,}
\end{equation*}%
resulting in more familiar forms of the Boltzmann collision operators for
mixtures of monatomic and/or polyatomic molecules modeled with a continuous
energy variable, cf., e.g., \cite{DMS-05,BBBD-18}. Here and below the
internal energy gaps are given by $\Delta I=0$ in the mono/mono-case, $%
\Delta I=I_{\ast }^{\prime }-I_{\ast }$ in the mono/poly-case, $\Delta
I=I^{\prime }-I$ in the poly/mono-case, while in the poly/poly-case $\Delta I=I^{\prime }+I_{\ast
}^{\prime }-I-I_{\ast }$.

\subsection{Collision invariants and Maxwellian distributions\label{S2.2}}

The following lemma follows directly by the relations $\left( \ref{rel1}%
\right) $.

\begin{lemma}
\label{L0}The measures 
\begin{equation*}
dA_{\alpha \beta }=W_{\alpha \beta }(\boldsymbol{\xi },\boldsymbol{\xi }%
_{\ast },I,I_{\ast }\left\vert \boldsymbol{\xi }^{\prime },\boldsymbol{\xi }%
_{\ast }^{\prime },I^{\prime },I_{\ast }^{\prime }\right. )d\boldsymbol{\xi }%
\,d\boldsymbol{\xi }_{\ast }d\boldsymbol{\xi }^{\prime }d\boldsymbol{\xi }%
_{\ast }^{\prime }dIdI_{\ast }dI^{\prime }dI_{\ast }^{\prime }
\end{equation*}%
are invariant under the (ordered) interchange%
\begin{equation}
\left( \boldsymbol{\xi },\boldsymbol{\xi }_{\ast },I,I_{\ast }\right)
\leftrightarrow \left( \boldsymbol{\xi }^{\prime },\boldsymbol{\xi }_{\ast
}^{\prime },I^{\prime },I_{\ast }^{\prime }\right) \text{,}  \label{tr}
\end{equation}%
of variables for $\left( \alpha ,\beta \right) \subseteq \left\{
1,...,s\right\} $, while%
\begin{equation*}
dA_{\alpha \beta }+dA_{\beta \alpha }\text{, }\left\{ \alpha ,\beta \right\}
\subseteq \left\{ 1,...,s\right\} \text{,}
\end{equation*}%
are invariant under the (ordered) interchange of variables%
\begin{equation}
\left( \boldsymbol{\xi },\boldsymbol{\xi }^{\prime },I,I^{\prime }\right)
\leftrightarrow \left( \boldsymbol{\xi }_{\ast },\boldsymbol{\xi }_{\ast
}^{\prime },I_{\ast },I_{\ast }^{\prime }\right) \text{.}  \label{tr1}
\end{equation}
\end{lemma}

The weak form of the collision operator $Q(f,f)$ reads%
\begin{eqnarray*}
\left( Q(f,f),g\right) &=&\sum_{\alpha ,\beta =1}^{s}\int_{\left( \mathbb{R}%
^{3}\times \mathbb{R}_{+}\right) ^{4}}\Lambda _{\alpha \beta }(f)g_{\alpha
}\,dA_{\alpha \beta } \\
&=&\sum_{\alpha ,\beta =1}^{s}\int_{\left( \mathbb{R}^{3}\times \mathbb{R}%
_{+}\right) ^{4}}\Lambda _{\alpha \beta }(f)g_{_{\beta \ast }}\,dA_{\alpha
\beta } \\
&=&-\sum_{\alpha ,\beta =1}^{s}\int_{\left( \mathbb{R}^{3}\times \mathbb{R}%
_{+}\right) ^{4}}\Lambda _{\alpha \beta }(f)g_{\alpha }^{\prime
}\,dA_{\alpha \beta } \\
&=&-\sum_{\alpha ,\beta =1}^{s}\int_{\left( \mathbb{R}^{3}\times \mathbb{R}%
_{+}\right) ^{4}}\Lambda _{\alpha \beta }(f)g_{\beta \ast }^{\prime
}\,dA_{\alpha \beta }
\end{eqnarray*}%
for any function $g=\left( g_{1},...,g_{s}\right) $, with $g_{\alpha
}=g_{\alpha }(\boldsymbol{\xi },I)$, such that the first integrals are
defined for all $\left\{ \alpha ,\beta \right\} \subseteq \left\{
1,...,s\right\} $, while the following equalities are obtained by applying
Lemma $\ref{L0}$.

Denote for any function $g=\left( g_{1},...,g_{s}\right) $, with $g_{\alpha
}=g_{\alpha }(\boldsymbol{\xi },I)$, 
\begin{equation*}
\Delta _{\alpha \beta }\left( g\right) :=g_{\alpha }+g_{_{\beta \ast
}}-g_{\alpha }^{\prime }-g_{\beta \ast }^{\prime }\text{.}
\end{equation*}

We have the following proposition.

\begin{proposition}
\label{P1}Let $g=\left( g_{1},...,g_{s}\right) $, with $g_{\alpha
}=g_{\alpha }(\mathbf{Z})$, be such that for all $\left\{ \alpha ,\beta
\right\} \subseteq \left\{ 1,...,s\right\} $%
\begin{equation*}
\int_{\left( \mathbb{R}^{3}\times \mathbb{R}_{+}\right) ^{4}}\Lambda
_{\alpha \beta }(f)g_{\alpha }\,dA_{\alpha \beta }\text{, }
\end{equation*}%
where $\Lambda _{\alpha
\beta }(f)=\dfrac{f_{\alpha }^{\prime }f_{\beta \ast }^{\prime }}{\left(
I^{\prime }\right) ^{\delta ^{\left( \alpha \right) }/2-1}\left( I_{\ast
}^{\prime }\right) ^{\delta ^{\left( \beta \right) }/2-1}}-\dfrac{f_{\alpha
}f_{\beta \ast }}{I^{\delta ^{\left( \alpha \right) }/2-1}I_{\ast }^{\delta
^{\left( \beta \right) }/2-1}}$, is defined. Then%
\begin{equation*}
\left( Q(f,f),g\right) =\frac{1}{4}\sum_{\alpha ,\beta =1}^{s}\int_{\left( 
\mathbb{R}^{3}\times \mathbb{R}_{+}\right) ^{4}}\Lambda _{\alpha \beta
}(f)\Delta _{\alpha \beta }\left( g\right) \,dA_{\alpha \beta }.
\end{equation*}
\end{proposition}

\begin{definition}
A function $g=\left( g_{1},...,g_{s}\right) $, with $g_{\alpha }=g_{\alpha }(%
\mathbf{Z})$,$\ $is a collision invariant if 
\begin{equation*}
\Delta _{\alpha \beta }\left( g\right) \,W_{\alpha \beta }(\boldsymbol{\xi },%
\boldsymbol{\xi }_{\ast },I,I_{\ast }\left\vert \boldsymbol{\xi }^{\prime },%
\boldsymbol{\xi }_{\ast }^{\prime },I^{\prime },I_{\ast }^{\prime }\right.
)=0\text{ a.e.}
\end{equation*}%
for all $\left\{ \alpha ,\beta \right\} \subseteq \left\{ 1,...,s\right\} $.
\end{definition}

It is clear that $e_{1},$ $...,e_{s},$ $m\xi _{x},$ $m\xi _{y},$ $m\xi _{z},$
and $m\left\vert \boldsymbol{\xi }\right\vert ^{2}+2\mathbb{I}$, denoting $%
m=\left( m_{1},...,m_{s}\right) $, $\mathbb{I}=(\underset{s_{0}}{\underbrace{%
0,...,0}},\underset{s_{1}=s-s_{0}}{\underbrace{I,...,I}})$, and $\left\{
e_{1},...,e_{s}\right\} $ for the standard basis of $\mathbb{R}^{s}$, are
collision invariants - corresponding to conservation of mass(es), momentum,
and total energy.

In fact, we have the following proposition, cf. \cite{DMS-05, BBBD-18}.

\begin{proposition}
\label{P2}Let $m=\left( m_{1},...,m_{s}\right) $, $\mathbb{I}=(\underset{%
s_{0}}{\underbrace{0,...,0}},\underset{s_{1}}{\underbrace{I,...,I}})$, and $%
\left\{ e_{1},...,e_{s}\right\} $ be the standard basis of $\mathbb{R}^{s}$.
Then the vector space of collision invariants is generated by 
\begin{equation*}
\left\{ e_{1},...,e_{s},m\xi _{x},m\xi _{y},m\xi _{z},\mathbf{m}\left\vert 
\boldsymbol{\xi }\right\vert ^{2}+2\mathbb{I}\right\} .
\end{equation*}
\end{proposition}

Define%
\begin{equation*}
\mathcal{W}\left[ f\right] :=\left( Q(f,f),\log \left( \varphi ^{-1}f\right)
\right) ,
\end{equation*}%
where $\varphi =\mathrm{diag}\left( I^{\delta ^{\left( 1\right)
}/2-1},...,I^{\delta ^{\left( s\right) }/2-1}\right) $. It follows by
Proposition $\ref{P1}$ that%
\begin{eqnarray*}
\mathcal{W}\left[ f\right] &=&-\frac{1}{4}\sum_{\alpha ,\beta
=1}^{s}\int\limits_{\left( \mathbb{R}^{3}\times \mathbb{R}_{+}\right)
^{4}}\left( \frac{I^{\delta ^{\left( \alpha \right) }/2-1}I_{\ast }^{\delta
^{\left( \beta \right) }/2-1}f_{\alpha }^{\prime }f_{\beta \ast }^{\prime }}{%
f_{\alpha }f_{\beta \ast }\left( I^{\prime }\right) ^{\delta ^{\left( \alpha
\right) }/2-1}\left( I_{\ast }^{\prime }\right) ^{\delta ^{\left( \beta
\right) }/2-1}}-1\right) \\
&&\times \log \left( \frac{I^{\delta ^{\left( \alpha \right) }/2-1}I_{\ast
}^{\delta ^{\left( \beta \right) }/2-1}f_{\alpha }^{\prime }f_{\beta \ast
}^{\prime }}{f_{\alpha }f_{\beta \ast }\left( I^{\prime }\right) ^{\delta
^{\left( \alpha \right) }/2-1}\left( I_{\ast }^{\prime }\right) ^{\delta
^{\left( \beta \right) }/2-1}}\right) \frac{f_{\alpha }f_{\beta \ast }}{%
I^{\delta ^{\left( \alpha \right) }/2-1}I_{\ast }^{\delta ^{\left( \beta
\right) }/2-1}}\,dA_{\alpha \beta }\text{.}
\end{eqnarray*}%
Since $\left( x-1\right) \mathrm{log}\left( x\right) \geq 0$ for $x>0$, with
equality if and only if $x=1$,%
\begin{equation*}
\mathcal{W}\left[ f\right] \leq 0\text{,}
\end{equation*}%
with equality if and only if for all $\left\{ \alpha ,\beta \right\}
\subseteq \left\{ 1,...,s\right\} $ 
\begin{equation}
\Lambda _{\alpha \beta }(f)W_{\alpha \beta }=0\text{ a.e.,}  \label{m1}
\end{equation}%
or, equivalently, if and only if%
\begin{equation*}
Q(f,f)\equiv 0\text{.}
\end{equation*}

For any equilibrium, or, Maxwellian, distribution $M=(M_{1},...,M_{s})$, it
follows by equation $\left( \ref{m1}\right) $, since $Q(M,M)\equiv 0$, that%
\begin{equation*}
\left( \log \frac{M_{\alpha }}{I^{\delta ^{\left( \alpha \right) }/2-1}}%
+\log \frac{M_{\beta \ast }}{I_{\ast }^{\delta ^{\left( \beta \right) }/2-1}}%
-\log \frac{M_{\alpha }^{\prime }}{\left( I^{\prime }\right) ^{\delta
^{\left( \alpha \right) }/2-1}}-\log \frac{M_{\beta \ast }^{\prime }}{\left(
I_{\ast }^{\prime }\right) ^{\delta ^{\left( \beta \right) }/2-1}}\right)
W_{\alpha \beta }=0\text{ a.e. .}
\end{equation*}%
Hence, $\log \left( \varphi ^{-1}M\right) =\left( \log \dfrac{M_{1}}{%
I^{\delta ^{\left( 1\right) }/2-1}},...,\log \dfrac{M_{s}}{I^{\delta
^{\left( s\right) }/2-1}}\right) $ is a collision invariant, and the
components of the Maxwellian distributions $M=(M_{1},...,M_{s})$ are of the
form 
\begin{equation*}
M_{\alpha }=\left\{ 
\begin{array}{l}
\dfrac{n_{\alpha }m_{\alpha }^{3/2}}{\left( 2\pi k_{B}T\right) ^{3/2}}%
e^{-m_{\alpha }\left\vert \boldsymbol{\xi }-\mathbf{u}\right\vert
^{2}/\left( 2k_{B}T\right) }\text{, }\alpha \in \left\{ 1,...,s_{0}\right\} 
\text{ \ \ } \\ 
\dfrac{n_{\alpha }I^{\delta ^{\left( \alpha \right) }/2-1}m_{\alpha }^{3/2}e^{-\left( m_{\alpha }\left\vert \boldsymbol{\xi }-%
\mathbf{u}\right\vert ^{2}+2I\right) /\left( 2k_{B}T\right) }}{%
\left( 2\pi \right) ^{3/2}\left( k_{B}T\right) ^{\left( \delta ^{\left(
\alpha \right) }+3\right) /2}\Gamma \left(\delta ^{\left( \alpha
\right) }/2\right) }\text{, }\alpha
\in \left\{ s_{0}+1,...,s\right\} \text{,}%
\end{array}%
\right.
\end{equation*}%
where $n_{\alpha }=\left( M,e_{_{\alpha }}\right) $, $\mathbf{u}=\dfrac{1}{%
\rho }\left( M,m\boldsymbol{\xi }\right) $, and $T=\dfrac{1}{3nk_{B}}\left(
M,m\left\vert \boldsymbol{\xi }-\mathbf{u}\right\vert ^{2}\right) $, with $%
m=(m_{1},...,m_{s})$, $n=\sum\limits_{\alpha =1}^{s}n_{\alpha }$, and $\rho
=\sum\limits_{\alpha =1}^{s}m_{\alpha }n_{\alpha }$, while $\Gamma =\Gamma
(n)$ and $k_{B}$ denote the Gamma function $\Gamma (n)=\int_{0}^{\infty
}x^{n-1}e^{-x}\,dx$ and the Boltzmann constant, respectively.

Note that, by equation $\left( \ref{m1}\right) $, any Maxwellian
distribution $M=(M_{1},...,M_{s})$ satisfies the relations 
\begin{equation}
\Lambda _{\alpha \beta }(M)W_{\alpha \beta }=0\text{ a.e.}  \label{M1}
\end{equation}
for any $\left\{ \alpha ,\beta \right\} \subseteq \left\{ 1,...,s\right\} $.

\begin{remark}
Introducing the $\mathcal{H}$-functional%
\begin{equation*}
\mathcal{H}\left[ f\right] =\left( f,\log \left( I^{1-\delta /2}f\right)
\right) \text{,}
\end{equation*}%
an $\mathcal{H}$-theorem can be obtained, cf. \cite{DMS-05, BBBD-18}.
\end{remark}

\subsection{Linearized collision operator\label{S2.3}}

Consider a deviation of a Maxwellian
distribution $M=(M_{1},...,M_{s})$, where $M_{\alpha }=\left\{ 
\begin{array}{l}
\dfrac{n_{\alpha }m_{\alpha }^{3/2}}{\left( 2\pi \right) ^{3/2}}%
e^{-m_{\alpha }\left\vert \boldsymbol{\xi }\right\vert ^{2}/2}\text{, }%
\alpha \in \left\{ 1,...,s_{0}\right\} \text{ \ \ } \\ 
\dfrac{n_{\alpha }I^{\delta ^{\left( \alpha \right) }/2-1}m_{\alpha }^{3/2}}{%
\left( 2\pi \right) ^{3/2}\Gamma \left( \delta ^{\left( \alpha \right)
}/2\right) }e^{-m_{\alpha }\left\vert \boldsymbol{\xi }\right\vert
^{2}/2}e^{-I}\text{, }\alpha \in \left\{ s_{0}+1,...,s\right\} \text{,}%
\end{array}%
\right. $, of the form%
\begin{equation}
f=M+M^{1/2}h\text{.}  \label{s1}
\end{equation}%
Insertion in the Boltzmann equation $\left( \ref{BE1}\right) $ results in
the system%
\begin{equation}
\frac{\partial h}{\partial t}+\left( \boldsymbol{\xi }\cdot \nabla _{\mathbf{%
x}}\right) h+\mathcal{L}h=\Gamma \left( h,h\right) \text{,}  \label{LBE}
\end{equation}%
where the components of the linearized collision operator $\mathcal{L}%
=\left( \mathcal{L}_{1},...,\mathcal{L}_{s}\right) $ are given by%
\begin{eqnarray}
\mathcal{L}_{\alpha }h &=&-M_{\alpha }^{-1/2}\left( Q_{\alpha
}(M,M^{1/2}h)+Q_{\alpha }(M^{1/2}h,M)\right)  \notag \\
&=&\sum\limits_{\beta =1}^{s}\int_{\left( \mathbb{R}^{3}\times \mathbb{R}%
_{+}\right) ^{3}}\frac{\left( M_{\beta \ast }M_{\alpha }^{\prime }M_{\beta
\ast }^{\prime }\right) ^{1/2}}{\left( II^{\prime }\right) ^{\delta ^{\left(
\alpha \right) }/4-1/2}\left( I_{\ast }I_{\ast }^{\prime }\right) ^{\delta
^{\left( \beta \right) }/4-1/2}}\Delta _{\alpha \beta }\left( \frac{h}{%
M^{1/2}}\right)  \notag \\
&&\times W_{\alpha \beta }d\boldsymbol{\xi }_{\ast }d\boldsymbol{\xi }%
^{\prime }d\boldsymbol{\xi }_{\ast }^{\prime }dI_{\ast }dI^{\prime }dI_{\ast
}^{\prime }  \notag \\
&=&\nu _{\alpha }h_{\alpha }-K_{\alpha }\left( h\right) \text{,}
\label{dec2}
\end{eqnarray}%
\ with%
\begin{eqnarray}
\nu _{\alpha } &=&\sum\limits_{\beta =1}^{s}\int_{\left( \mathbb{R}%
^{3}\times \mathbb{R}_{+}\right) ^{3}}\frac{M_{\beta \ast }}{I^{\delta
^{\left( \alpha \right) }/2-1}I_{\ast }^{\delta ^{\left( \beta \right) }/2-1}%
}W_{\alpha \beta }d\boldsymbol{\xi }_{\ast }d\boldsymbol{\xi }^{\prime }d%
\boldsymbol{\xi }_{\ast }^{\prime }dI_{\ast }dI^{\prime }dI_{\ast }^{\prime }%
\text{,}  \notag \\
K_{\alpha } &=&\sum\limits_{\beta =1}^{s}\int_{\left( \mathbb{R}^{3}\times 
\mathbb{R}_{+}\right) ^{3}}\left( \frac{h_{\alpha }^{\prime }}{\left(
M_{\alpha }^{\prime }\right) ^{1/2}}+\frac{h_{\beta \ast }^{\prime }}{\left(
M_{\beta \ast }^{\prime }\right) ^{1/2}}-\frac{h_{\beta \ast }}{M_{\beta
\ast }^{1/2}}\right) W_{\alpha \beta }  \notag \\
&&\times \frac{\left( M_{\beta \ast }M_{\alpha }^{\prime }M_{\beta \ast
}^{\prime }\right) ^{1/2}}{\left( II^{\prime }\right) ^{\delta ^{\left(
\alpha \right) }/4-1/2}\left( I_{\ast }I_{\ast }^{\prime }\right) ^{\delta
^{\left( \beta \right) }/4-1/2}}d\boldsymbol{\xi }_{\ast }d\boldsymbol{\xi }%
^{\prime }d\boldsymbol{\xi }_{\ast }^{\prime }dI_{\ast }dI^{\prime }dI_{\ast
}^{\prime }\text{,}  \label{dec1}
\end{eqnarray}%
while the components of the quadratic term $S=\left( S_{1},...,S_{s}\right) $
are given by%
\begin{equation}
S_{\alpha }\left( h,h\right) =M_{\alpha }^{-1/2}Q_{\alpha
}(M^{1/2}h,M^{1/2}h)\text{.}  \label{nl1}
\end{equation}%
for $\alpha \in \left\{ 1,...,s\right\} $. The multiplication operator $%
\Lambda $ defined by 
\begin{equation*}
\Lambda (f)=\nu f\text{, where }\nu =\mathrm{diag}\left( \nu _{1},...,\nu
_{s}\right) \text{,}
\end{equation*}%
is a closed, densely defined, self-adjoint operator on $\mathcal{\mathfrak{h}%
}$. It is Fredholm, as well, if and only if $\Lambda $ is coercive.

The following lemma follows immediately by Lemma $\ref{L0}$.

\begin{lemma}
\label{L1}For any $\left\{ \alpha ,\beta \right\} \subseteq \left\{
1,...,s\right\} $ the measure 
\begin{equation*}
d\widetilde{A}_{\alpha \beta }=\frac{\left( M_{\alpha }M_{\beta \ast
}M_{\alpha }^{\prime }M_{\beta \ast }^{\prime }\right) ^{1/2}}{\left(
II^{\prime }\right) ^{\delta ^{\left( \alpha \right) }/4-1/2}\left( I_{\ast
}I_{\ast }^{\prime }\right) ^{\delta ^{\left( \beta \right) }/4-1/2}}%
dA_{\alpha \beta }
\end{equation*}%
is invariant under the (ordered) interchange $\left( \ref{tr}\right) $ of
variables, while%
\begin{equation*}
d\widetilde{A}_{\alpha \beta }+d\widetilde{A}_{\beta \alpha }\text{, with }%
\left\{ \alpha ,\beta \right\} \subseteq \left\{ 1,...,s\right\} \text{,}
\end{equation*}%
is invariant under the (ordered) interchange $\left( \ref{tr1}\right) $ of
variables
\end{lemma}

The weak form of the linearized collision operator $\mathcal{L}$ reads%
\begin{eqnarray*}
\left( \mathcal{L}h,g\right) &=&\sum_{\alpha ,\beta =1}^{s}\int_{\left( 
\mathbb{R}^{3}\times \mathbb{R}_{+}\right) ^{4}}\Delta _{\alpha \beta
}\left( \frac{h}{M^{1/2}}\right) \frac{g_{\alpha }}{M_{\alpha }^{1/2}}\,d%
\widetilde{A}_{\alpha \beta } \\
&=&\sum_{\alpha ,\beta =1}^{s}\int_{\left( \mathbb{R}^{3}\times \mathbb{R}%
_{+}\right) ^{4}}\Delta _{\alpha \beta }\left( \frac{h}{M^{1/2}}\right) 
\frac{g_{\beta \ast }}{M_{\beta \ast }^{1/2}}\,d\widetilde{A}_{\alpha \beta }
\\
&=&-\sum_{\alpha ,\beta =1}^{s}\int_{\left( \mathbb{R}^{3}\times \mathbb{R}%
_{+}\right) ^{4}}\Delta _{\alpha \beta }\left( \frac{h}{M^{1/2}}\right) 
\frac{g_{\alpha }^{\prime }}{\left( M_{\alpha }^{\prime }\right) ^{1/2}}\,d%
\widetilde{A}_{\alpha \beta } \\
&=&-\sum_{\alpha ,\beta =1}^{s}\int_{\left( \mathbb{R}^{3}\times \mathbb{R}%
_{+}\right) ^{4}}\Delta _{\alpha \beta }\left( \frac{h}{M^{1/2}}\right) 
\frac{g_{\beta \ast }^{\prime }}{\left( M_{\beta \ast }^{\prime }\right)
^{1/2}}\,d\widetilde{A}_{\alpha \beta }\text{,}
\end{eqnarray*}%
for any function $g=\left( g_{1},...,g_{s}\right) $, with $g_{\alpha
}=g_{\alpha }(\boldsymbol{\xi },I)$, such that the first integrals are
defined for all $\left\{ \alpha ,\beta \right\} \subseteq \left\{
1,...,s\right\} $, while the following equalities are obtained by applying
Lemma $\ref{L1}$.

We have the following lemma.

\begin{lemma}
\label{L2}Let $g=\left( g_{1},...,g_{s}\right) $, with $g_{\alpha
}=g_{\alpha }(\boldsymbol{\xi },I)$, be such that%
\begin{equation*}
\int_{\left( \mathbb{R}^{3}\times \mathbb{R}_{+}\right) ^{4}}\Delta _{\alpha
\beta }\left( \frac{h}{M^{1/2}}\right) \frac{g_{\alpha }}{M_{\alpha }^{1/2}}%
\,d\widetilde{A}_{\alpha \beta }\text{,}
\end{equation*}%
is defined for any $\left\{ \alpha ,\beta \right\} \subseteq \left\{
1,...,s\right\} $. Then%
\begin{equation*}
\left( \mathcal{L}h,g\right) =\frac{1}{4}\sum_{\alpha ,\beta
=1}^{s}\int_{\left( \mathbb{R}^{3}\times \mathbb{R}_{+}\right) ^{4}}\Delta
_{\alpha \beta }\left( \frac{h}{M^{1/2}}\right) \,\Delta _{\alpha \beta
}\left( \frac{g}{M^{1/2}}\right) \,d\widetilde{A}_{\alpha \beta }.
\end{equation*}
\end{lemma}

\begin{proposition}
\label{P3}The linearized collision operator is symmetric and nonnegative,%
\begin{equation*}
\left( \mathcal{L}h,g\right) =\left( h,\mathcal{L}g\right) \text{ and }%
\left( \mathcal{L}h,h\right) \geq 0\text{,}
\end{equation*}%
and\ the kernel of $\mathcal{L}$, $\ker \mathcal{L}$, is generated by%
\begin{equation*}
\left\{ \mathcal{M}^{1/2}e_{1},...,\mathcal{M}^{1/2}e_{s},\mathcal{M}%
^{1/2}m\xi _{x},\mathcal{M}^{1/2}m\xi _{y},\mathcal{M}^{1/2}m\xi _{z},%
\mathcal{M}^{1/2}\left( m\left\vert \boldsymbol{\xi }\right\vert ^{2}+2%
\mathbb{I}\right) \right\} \text{,}
\end{equation*}%
where $m=\left( m_{1},...,m_{s}\right) $, $\mathbb{I}=(\underset{s_{0}}{%
\underbrace{0,...,0}},\underset{s_{1}}{\underbrace{I,...,I}})$, and $%
\mathcal{M}=\mathrm{diag}\left( M_{1},...,M_{s}\right) $.
\end{proposition}

\begin{proof}
By Lemma $\ref{L2}$, it is immediate that $\left( \mathcal{L}h,g\right)
=\left( h,\mathcal{L}g\right) $, and%
\begin{equation*}
\left( \mathcal{L}h,h\right) =\frac{1}{4}\sum_{\alpha ,\beta
=1}^{s}\int_{\left( \mathbb{R}^{3}\times \mathbb{R}_{+}\right) ^{4}}\left(
\Delta _{\alpha \beta }\left( \frac{h}{M^{1/2}}\right) \right) ^{2}d%
\widetilde{A}_{\alpha \beta }.
\end{equation*}%
Furthermore, $h\in \ker \mathcal{L}$ if and only if $\left( \mathcal{L}%
h,h\right) =0$, which will be fulfilled\ if and only if for all $\left\{
\alpha ,\beta \right\} \subseteq \left\{ 1,...,s\right\} $%
\begin{equation*}
\Delta _{\alpha \beta }\left( \frac{h}{M^{1/2}}\right) \,W_{\alpha \beta }(%
\boldsymbol{\xi },\boldsymbol{\xi }_{\ast },I,I_{\ast }\left\vert 
\boldsymbol{\xi }^{\prime },\boldsymbol{\xi }_{\ast }^{\prime },I^{\prime
},I_{\ast }^{\prime }\right. )=0\text{ a.e. ,}
\end{equation*}%
i.e., if and only if $\mathcal{M}^{-1/2}h$ is a collision invariant. The
last part of the lemma now follows by Proposition $\ref{P2}.$
\end{proof}

\begin{remark}
Note also that the quadratic term is orthogonal to the kernel of $\mathcal{L}
$, i.e. $S\left( h,h\right) \in \left( \ker \mathcal{L}\right) ^{\perp _{%
\mathcal{\mathfrak{h}}}}$.
\end{remark}

\section{Main Results\label{S3}}

This section is devoted to the main results, concerning compact properties
in Theorems \ref{Thm1} and bounds of collision frequencies in Theorems \ref%
{Thm2}.

Assume that for some positive number $\gamma $, such that $0<\gamma <1$,
there is for all $\left\{ \alpha ,\beta \right\} \subseteq \left\{
1,...,s\right\} $ a bound 
\begin{eqnarray}
0\leq \sigma _{\alpha \beta }\left( \left\vert \mathbf{g}\right\vert ,\cos
\theta ,I,I_{\ast },I^{\prime },I_{\ast }^{\prime }\right) \leq C\frac{\Psi
_{\alpha \beta }+\left( \Psi _{\alpha \beta }\right) ^{\gamma /2}}{%
\left\vert \mathbf{g}\right\vert ^{2}}\Upsilon _{\alpha \beta }\text{, where}
\notag \\
\Upsilon _{\alpha \beta }=\frac{\left( I^{\prime }\right) ^{\delta
^{\left( \alpha \right) }/2-1}\left( I_{\ast }^{\prime }\right) ^{\delta
^{\left( \beta \right) }/2-1}}{\mathcal{E}_{\alpha \beta }^{\delta ^{\left(
\alpha \right) }/2}\left( \mathcal{E}_{\alpha \beta }^{\ast }\right)
^{\delta ^{\left( \beta \right) }/2}}\text{ and }\Psi _{\alpha \beta
}=\left\vert \mathbf{g}\right\vert \sqrt{\left\vert \mathbf{g}\right\vert
^{2}-\frac{2\Delta I}{\mu_{\alpha \beta }}}\text{,}
\label{est1}
\end{eqnarray}%
for $\mu_{\alpha \beta }\left\vert \mathbf{g}\right\vert ^{2}>2\Delta I$, on the scattering cross sections, or,
equivalently, the bound 
\begin{equation*}
0\leq B_{i\alpha \beta }\left( \left\vert \mathbf{g}\right\vert ,\cos \theta
,I,I_{\ast },I^{\prime },I_{\ast }^{\prime }\right) \leq CE_{\alpha \beta
}^{1/2}\left( 1+\frac{1}{\Psi _{\alpha \beta }^{1-\gamma /2}}\right)
\end{equation*}%
for $\mu_{\alpha \beta }\left\vert \mathbf{g}\right\vert ^{2}>2\Delta I$, on the collision kernels.

Here and below $\mathcal{E}_{\alpha \beta }=\mathcal{E}_{\beta \alpha
}^{\ast }=1$ if $\alpha \in \left\{ 1,...,s_{0}\right\} $, if $\alpha \in \left\{ s_{0}+1,...,s\right\} $%
 and $\beta \in \left\{ 1,...,s_{0}\right\} $, then $\mathcal{E}%
_{\beta \alpha }^{\ast }=\mu_{\alpha \beta }\left\vert \mathbf{g}\right\vert ^{2}/2+I_{\ast }=\mu_{\alpha \beta }\left\vert 
\mathbf{g}^{\prime }\right\vert ^{2}/2+I_{\ast }^{\prime }$ and $\mathcal{E}%
_{\alpha \beta }=\mu_{\alpha \beta }\left\vert \mathbf{g}\right\vert ^{2}/2+I=\mu_{\alpha \beta }\left\vert \mathbf{g}^{\prime
}\right\vert ^{2}/2+I^{\prime }$, and  if $%
\left\{ \alpha ,\beta \right\} \subset \left\{ s_{0}+1,...,s\right\} $, then $\mathcal{E}%
_{\alpha \beta }=\mathcal{E}_{\alpha \beta }^{\ast }=\mu_{\alpha \beta }\left\vert \mathbf{g}%
\right\vert ^{2}/2+I+I_{\ast }=\mu_{\alpha \beta }\left\vert \mathbf{g}^{\prime }\right\vert
^{2}/2+I^{\prime }+I_{\ast }^{\prime }$.

Then the following result may be obtained.

\begin{theorem}
\label{Thm1}Assume that for all $\left\{ \alpha ,\beta \right\} \subseteq
\left\{ 1,...,s\right\} $ the scattering cross sections $\sigma _{\alpha
\beta }$ satisfy the bound $\left( \ref{est1}\right) $ for some positive
number $\gamma $, $0<\gamma <1$. 

Then the operator $K=\left(
K_{1},...,K_{s}\right) $, with the components $K_{\alpha }$ given by $\left( %
\ref{dec1}\right) $ is a self-adjoint compact operator on $\mathcal{%
\mathfrak{h}}=\left( L^{2}\left( d\boldsymbol{\xi \,}\right) \right)
^{s_{0}}\times \left( L^{2}\left( d\boldsymbol{\xi \,}dI\right) \right)
^{s_{1}}$.
\end{theorem}

The proof of Theorem \ref{Thm1} will be addressed in Section $\ref{PT1}$.

\begin{corollary}
\label{Cor1}The linearized collision operator $\mathcal{L}$, with scattering
cross sections satisfying $\left( \ref{est1}\right) $, is a closed, densely
defined, self-adjoint operator on $\mathcal{\mathfrak{h}}$.
\end{corollary}

\begin{proof}
By Theorem \ref{Thm1}, the linear operator $\mathcal{L}=\Lambda -K$ is
closed as the sum of a closed and a bounded operator, and densely defined,
since the domains of the linear operators $\mathcal{L}$ and $\Lambda $ are
equal; $D(\mathcal{L})=D(\Lambda )$. Furthermore, it is a self-adjoint
operator, since the set of self-adjoint operators is closed under addition
of bounded self-adjoint operators, see Theorem 4.3 of Chapter V in \cite%
{Kato}.
\end{proof}

Now consider the scattering cross sections%
\begin{equation}
\sigma _{\alpha \beta }=C_{\alpha \beta }\dfrac{\sqrt{\left\vert \mathbf{g}%
\right\vert ^{2}-\dfrac{2\Delta I}{\mu_{\alpha \beta }}}}{\left\vert \mathbf{g%
}\right\vert E_{\alpha \beta }^{\eta /2}}\frac{\left( I^{\prime }\right)
^{\delta ^{\left( \alpha \right) }/2-1}\left( I_{\ast }^{\prime }\right)
^{\delta ^{\left( \beta \right) }/2-1}}{\mathcal{E}_{\alpha \beta }^{\delta
^{\left( \alpha \right) }/2}\left( \mathcal{E}_{\alpha \beta }^{\ast
}\right) ^{\delta ^{\left( \beta \right) }/2}}\text{, if }\mu_{\alpha \beta }\left\vert \mathbf{%
g}\right\vert ^{2}>2\Delta I\text{,}  \label{e1}
\end{equation}%
for some positive constants $C_{\alpha
\beta }>0$ for all $\left\{ \alpha ,\beta \right\} \subseteq \left\{
1,...,s\right\} $ and nonnegative number $\eta $ less than $1$, $0\leq \eta
<1,$ - cf. hard sphere models for $\eta =0$.

In fact, it would be enough with the bounds (for $\mu_{\alpha \beta }\left\vert \mathbf{%
g}\right\vert ^{2}>2\Delta I$)
\begin{eqnarray}
C_{-}\dfrac{\sqrt{\left\vert \mathbf{g}\right\vert ^{2}-\dfrac{2\Delta I}{\mu_{\alpha \beta }}}}{\left\vert \mathbf{g}\right\vert E_{\alpha \beta
}^{\eta /2}}\Upsilon _{\alpha \beta }\leq \sigma _{\alpha \beta }\leq C_{+}%
\dfrac{\sqrt{\left\vert \mathbf{g}\right\vert ^{2}-\dfrac{2\Delta I}{\mu_{\alpha \beta }}}}{\left\vert \mathbf{g}\right\vert E_{\alpha \beta }^{\eta
/2}}\Upsilon _{\alpha \beta }\text{,}  \notag \\
\text{with }\Upsilon _{\alpha \beta }=\frac{\left( I^{\prime }\right)
^{\delta ^{\left( \alpha \right) }/2-1}\left( I_{\ast }^{\prime }\right)
^{\delta ^{\left( \beta \right) }/2-1}}{\mathcal{E}_{\alpha \beta }^{\delta
^{\left( \alpha \right) }/2}\left( \mathcal{E}_{\alpha \beta }^{\ast
}\right) ^{\delta ^{\left( \beta \right) }/2}}\text{ and }\widetilde{\Delta }%
_{\alpha \beta }I=\frac{m_{\alpha }+m_{\beta }}{m_{\alpha }m_{\beta }}\Delta
I\text{,}  \label{ie1}
\end{eqnarray}%
for some nonnegative number $\eta $ less than $1$, $0\leq \eta <1$, and some
positive constants $C_{\pm }>0$, on the scattering cross sections - cf. hard
potential with cut-off models.

The following bounds restricted to single species\ were obtained in \cite%
{Be-23b,DL-23}. In \cite{DL-23} the improved - compared to the one in \cite%
{Be-23b} - upper bound below was shown for single species.

\begin{theorem}
\label{Thm2} The linearized collision operator $\mathcal{L}$, with
scattering cross section $\left( \ref{e1}\right) $ (or $\left( \ref{ie1}%
\right) $), can be split into a positive multiplication operator $\Lambda $,
where $\Lambda \left( f\right) =\nu f$, with $\nu =\nu (\left\vert 
\boldsymbol{\xi }\right\vert )$, minus a compact operator $K$ on $\mathcal{%
\mathfrak{h}}$, such that there exist positive numbers $\nu _{-}$ and $\nu
_{+}$, $0<\nu _{-}<\nu _{+}$, such that for any $\alpha \in \left\{
1,...,s\right\} $ 
\begin{equation}
\nu _{-}\left( 1+\left\vert \boldsymbol{\xi }\right\vert +\sqrt{I}\right)
^{1-\eta }\leq \nu _{\alpha }\leq \nu _{+}\left( 1+\left\vert \boldsymbol{%
\xi }\right\vert +\sqrt{I}\right) ^{1-\eta }\text{.}  \label{ine1}
\end{equation}
\end{theorem}

The decomposition follows by decomposition $\left( \ref{dec2}\right) ,\left( %
\ref{dec1}\right) $ and Theorem $\ref{Thm1}$, while the bounds $\left( \ref%
{ine1}\right) $ on the collision frequency will be proven in Section $\ref%
{PT2}$.

\begin{corollary}
\label{Cor2}The linearized collision operator $\mathcal{L}$, with scattering
cross section $\left( \ref{e1}\right) $ (or $\left( \ref{ie1}\right) $), is
a Fredholm operator with domain%
\begin{equation*}
D(\mathcal{L})=\left( L^{2}\left( \left( 1+\left\vert \boldsymbol{\xi }%
\right\vert \right) ^{1-\eta }d\boldsymbol{\xi \,}\right) \right)
^{s_{0}}\times \left( L^{2}\left( \left( 1+\left\vert \boldsymbol{\xi }%
\right\vert +\sqrt{I}\right) ^{1-\eta }d\boldsymbol{\xi \,}dI\right) \right)
^{s_{1}}\text{.}
\end{equation*}
\end{corollary}

\begin{proof}
By Theorem \ref{Thm2} the multiplication operator $\Lambda $ is coercive,
and thus it is a Fredholm operator. Furthermore, the set of Fredholm
operators is closed under addition of compact operators, see Theorem 5.26 of
Chapter IV in \cite{Kato} and its proof, so, by Theorem \ref{Thm2}, $%
\mathcal{L}$ is a Fredholm operator.
\end{proof}

We stress that Corollary $\ref{Cor2}$ finally yields the Fredholmness of the
linearized operator assumed in \cite{BBBD-18} for kernels of the form $%
\left( \ref{e1}\right) $ or $\left( \ref{ie1}\right) $.

For hard sphere like models we obtain the following result.

\begin{corollary}
\label{Cor3}For the linearized collision operator $\mathcal{L}$, with
scattering cross section $\left( \ref{e1}\right) $ (or $\left( \ref{ie1}%
\right) $) where $\eta =0$, there exists a positive number $\lambda $, $%
0<\lambda <1$, such that 
\begin{equation*}
\left( h,\mathcal{L}h\right) \geq \lambda \left( h,\nu (\left\vert 
\boldsymbol{\xi }\right\vert ,I)h\right) \geq \lambda \nu _{-}\left(
h,\left( 1+\left\vert \boldsymbol{\xi }\right\vert \right) h\right)
\end{equation*}%
for all $h\in D\left( \mathcal{L}\right) \cap \mathrm{Im}\mathcal{L}$.
\end{corollary}

\begin{proof}
Let $h\in D\left( \mathcal{L}\right) \cap \left( \mathrm{ker}\mathcal{L}%
\right) ^{\perp }=D\left( \mathcal{L}\right) \cap \mathrm{Im}\mathcal{L}$.
As a Fredholm operator, $\mathcal{L}$ is closed with a closed range, and as
a compact operator, $K$ is bounded, and so there are positive constants $\nu
_{0}>0$ and $c_{K}>0$, such that%
\begin{equation*}
(h,\mathcal{L}h)\geq \nu _{0}(h,h)\text{ and }(h,Kh)\leq c_{K}(h,h).
\end{equation*}%
Let $\lambda =\dfrac{\nu _{0}}{\nu _{0}+c_{K}}$. Then the corollary follows,
since 
\begin{eqnarray*}
(h,\mathcal{L}h) &=&(1-\lambda )(h,\mathcal{L}h)+\lambda (h,(\nu (|%
\boldsymbol{\xi }|,I)-K)h) \\
&\geq &(1-\lambda )\nu _{0}(h,h)+\lambda (h,\nu (|\boldsymbol{\xi }%
|,I)h)-\lambda c_{K}(h,h) \\
&=&(\nu _{0}-\lambda (\nu _{0}+c_{K}))(h,h)+\lambda (h,\nu (|\boldsymbol{\xi 
}|,I)h) \\
&=&\lambda (h,\nu (|\boldsymbol{\xi }|,I)h)\text{.}
\end{eqnarray*}
\end{proof}

\begin{remark}
By Proposition $\ref{P3}$ and Corollary $\ref{Cor1}-\ref{Cor3}$ the
linearized operator $\mathcal{L}$ fulfills the properties assumed on the
linear operators in \cite{Be-23d}, and hence, the results therein can be
applied to hard sphere like models.
\end{remark}

\section{Compactness \label{PT1}}

This section concerns the proof of Theorem \ref{Thm1}. Note that in the
proof the kernels are rewritten in such a way that $\mathbf{Z}_{\ast }$ -
and not $\mathbf{Z}^{\prime }$ or $\mathbf{Z}_{\ast }^{\prime }$ - always
will be arguments of the distribution functions. As for single species,
either $\mathbf{Z}_{\ast }$ is an argument in the loss term (like $\mathbf{Z}
$) or in the gain term (unlike $\mathbf{Z}$) of the collision operator.
However,\ in the latter case, unlike for single species, for mixtures one
have to differ between two different cases; either $\mathbf{Z}_{\ast }$ are
associated to the same species as $\mathbf{Z}$, or not. The kernels of the
terms from the loss part of the collision operator will be shown\ to be
Hilbert-Schmidt in a quite direct way. The kernels of - some of - the terms
- for which $\mathbf{Z}_{\ast }$ is associated to the same species as $%
\mathbf{Z}$ - from the gain parts of the collision operators will be shown
to be uniform limits of Hilbert-Schmidt integral operators, i.e.,
approximately Hilbert-Schmidt in the sense of Lemma \ref{LGD}. Furthermore,
it will be shown that the kernels of the remaining terms - i.e. for which $%
\mathbf{Z}_{\ast }$ are associated to the opposite species to $\mathbf{Z}$ -
from the gain parts of the collision operators, are Hilbert-Schmidt.

To show the compactness properties we will apply the following result.
Denote, for any (nonzero) natural number $N$,%
\begin{align*}
\mathfrak{h}_{N}& :=\left\{ (\mathbf{Z},\mathbf{Z}_{\ast })\in \mathbb{%
Y\times Y}_{\ast }:\left\vert \boldsymbol{\xi }-\boldsymbol{\xi }_{\ast
}\right\vert \geq \frac{1}{N}\text{; }\left\vert \boldsymbol{\xi }%
\right\vert \leq N\right\} \text{, and} \\
b^{(N)}& =b^{(N)}(\mathbf{Z},\mathbf{Z}_{\ast }):=b(\mathbf{Z},\mathbf{Z}%
_{\ast })\mathbf{1}_{\mathfrak{h}_{N}}\text{.}
\end{align*}%
Here, either $\mathbf{Z}=\boldsymbol{\xi }$ and $\mathbb{Y=R}^{3}$, or, $%
\mathbf{Z}=\left( \boldsymbol{\xi },I\right) $ and $\mathbb{Y=R}^{3}\times 
\mathbb{R}_{+}$, and correspondingly, either $\mathbf{Z}=\boldsymbol{\xi }%
_{\ast }$ and $\mathbb{Y_{\ast }=R}^{3}$, or, $\mathbf{Z}_{\ast }=\left( 
\boldsymbol{\xi }_{\ast },I_{\ast }\right) $ and $\mathbb{Y_{\ast }=R}%
^{3}\times \mathbb{R}_{+}$. Then we have the following lemma, cf Glassey 
\cite[Lemma 3.5.1]{Glassey} and Drange \cite{Dr-75}.

\begin{lemma}
\label{LGD} Assume that $Tf\left( \mathbf{Z}\right) =\int_{\mathbb{Y}_{\ast
}}b(\mathbf{Z},\mathbf{Z}_{\ast })f\left( \mathbf{Z}_{\ast }\right) \,d%
\mathbf{Z}_{\ast }$, with $b(\mathbf{Z},\mathbf{Z}_{\ast })\geq 0$. Then $T$
is compact on $L^{2}\left( d\mathbf{Z}\right) $ if

(i) $\int_{\mathbb{Y}}b(\mathbf{Z},\mathbf{Z}_{\ast })\,d\mathbf{Z}$ is
bounded in $\mathbf{Z}_{\ast }$;

(ii) $b^{(N)}\in L^{2}\left( d\mathbf{Z}\boldsymbol{\,}d\mathbf{Z}_{\ast
}\right) $ for any (nonzero) natural number $N$;

(iii) $\underset{\mathbf{Z}\in \mathbb{Y}}{\sup }\int_{\mathbb{Y}_{\ast }}b(%
\mathbf{Z},\mathbf{Z}_{\ast })-b^{(N)}(\mathbf{Z},\mathbf{Z}_{\ast })\,d%
\mathbf{Z}_{\ast }\rightarrow 0$ as $N\rightarrow \infty $.
\end{lemma}

Then $T$ is the uniform limit of Hilbert-Schmidt integral operators \cite[%
Lemma 3.5.1]{Glassey} and we say that the kernel $b(\mathbf{Z},\mathbf{Z}%
_{\ast })$ is approximately Hilbert-Schmidt, while $T$ is an approximately
Hilbert-Schmidt integral operator. The reader is referred to Glassey \cite[%
Lemma 3.5.1]{Glassey} for a proof of Lemma \ref{LGD}.

Now we turn to the proof of Theorem \ref{Thm1}. Note that throughout the
proof $C$ will denote a generic positive constant.

\begin{proof}
For $\alpha \in \left\{ 1,...,s\right\} $ rewrite expression $\left( \ref%
{dec1}\right) $ as%
\begin{eqnarray*}
K_{\alpha } &=&\left( M_{\alpha }\right) ^{-1/2}\sum\limits_{\beta
=1}^{s}\int_{\left( \mathbb{R}^{3}\times \mathbb{R}_{+}\right)
^{3}}w_{\alpha \beta }(\boldsymbol{\xi },\boldsymbol{\xi }_{\ast },I,I_{\ast
}\left\vert \boldsymbol{\xi }^{\prime },\boldsymbol{\xi }_{\ast }^{\prime
},I^{\prime },I_{\ast }^{\prime }\right. ) \\
&&\times \left( \frac{h_{\alpha }^{\prime }}{\left( M_{\alpha }^{\prime
}\right) ^{1/2}}+\frac{h_{\beta \ast }^{\prime }}{\left( M_{\beta \ast
}^{\prime }\right) ^{1/2}}-\frac{h_{\beta \ast }}{M_{\beta \ast }^{1/2}}%
\right) \,d\boldsymbol{\xi }_{\ast }d\boldsymbol{\xi }^{\prime }d\boldsymbol{%
\xi }_{\ast }^{\prime }dI_{\ast }dI^{\prime }dI_{\ast }^{\prime }\text{,}
\end{eqnarray*}%
with 
\begin{eqnarray*}
&&w_{\alpha \beta }(\boldsymbol{\xi },\boldsymbol{\xi }_{\ast },I,I_{\ast
}\left\vert \boldsymbol{\xi }^{\prime },\boldsymbol{\xi }_{\ast }^{\prime
},I^{\prime },I_{\ast }^{\prime }\right. ) \\
&=&\frac{\left( M_{\alpha }M_{\beta \ast }M_{\alpha }^{\prime }M_{\beta \ast
}^{\prime }\right) ^{1/2}}{\left( II^{\prime }\right) ^{\delta ^{\left(
\alpha \right) }/4-1/2}\left( I_{\ast }I_{\ast }^{\prime }\right) ^{\delta
^{\left( \beta \right) }/4-1/2}}W_{\alpha \beta }(\boldsymbol{\xi },%
\boldsymbol{\xi }_{\ast },I,I_{\ast }\left\vert \boldsymbol{\xi }^{\prime },%
\boldsymbol{\xi }_{\ast }^{\prime },I^{\prime },I_{\ast }^{\prime }\right. )%
\text{.}
\end{eqnarray*}%
Due to relations $\left( \ref{rel1}\right) $, the relations%
\begin{eqnarray}
w_{\alpha \beta }(\boldsymbol{\xi },\boldsymbol{\xi }_{\ast },I,I_{\ast
}\left\vert \boldsymbol{\xi }^{\prime },\boldsymbol{\xi }_{\ast }^{\prime
},I^{\prime },I_{\ast }^{\prime }\right. ) &=&w_{\beta \alpha }(\boldsymbol{%
\xi }_{\ast },\boldsymbol{\xi },I_{\ast },I\left\vert \boldsymbol{\xi }%
_{\ast }^{\prime },\boldsymbol{\xi }^{\prime },I_{\ast }^{\prime },I^{\prime
}\right. )  \notag \\
w_{\alpha \beta }(\boldsymbol{\xi },\boldsymbol{\xi }_{\ast },I,I_{\ast
}\left\vert \boldsymbol{\xi }^{\prime },\boldsymbol{\xi }_{\ast }^{\prime
},I^{\prime },I_{\ast }^{\prime }\right. ) &=&w_{\alpha \beta }(\boldsymbol{%
\xi }^{\prime },\boldsymbol{\xi }_{\ast }^{\prime },I^{\prime },I_{\ast
}^{\prime }\left\vert \boldsymbol{\xi },\boldsymbol{\xi }_{\ast },I,I_{\ast
}\right. )  \notag \\
w_{\alpha \alpha }(\boldsymbol{\xi },\boldsymbol{\xi }_{\ast },I,I_{\ast
}\left\vert \boldsymbol{\xi }^{\prime },\boldsymbol{\xi }_{\ast }^{\prime
},I^{\prime },I_{\ast }^{\prime }\right. ) &=&w_{\alpha \alpha }(\boldsymbol{%
\xi },\boldsymbol{\xi }_{\ast },I,I_{\ast }\left\vert \boldsymbol{\xi }%
_{\ast }^{\prime },\boldsymbol{\xi }^{\prime },I_{\ast }^{\prime },I^{\prime
}\right. )  \label{rel2}
\end{eqnarray}%
are satisfied.

By renaming $\left\{ \boldsymbol{\xi }_{\ast },I_{\ast }\right\}
\leftrightarrows \left\{ \boldsymbol{\xi }_{\ast }^{\prime },I_{\ast
}^{\prime }\right\} $,%
\begin{eqnarray*}
&&\int_{\left( \mathbb{R}^{3}\times \mathbb{R}_{+}\right) ^{3}}w_{\alpha
\beta }(\boldsymbol{\xi },\boldsymbol{\xi }_{\ast },I,I_{\ast }\left\vert 
\boldsymbol{\xi }^{\prime },\boldsymbol{\xi }_{\ast }^{\prime },I^{\prime
},I_{\ast }^{\prime }\right. )\,\frac{h_{\beta \ast }^{\prime }}{\left(
M_{\beta \ast }^{\prime }\right) ^{1/2}}\,d\boldsymbol{\xi }_{\ast }d%
\boldsymbol{\xi }^{\prime }d\boldsymbol{\xi }_{\ast }^{\prime }dI_{\ast
}dI^{\prime }dI_{\ast }^{\prime } \\
&=&\int_{\left( \mathbb{R}^{3}\times \mathbb{R}_{+}\right) ^{3}}w_{\alpha
\beta }(\boldsymbol{\xi },\boldsymbol{\xi }_{\ast }^{\prime },I,I_{\ast
}^{\prime }\left\vert \boldsymbol{\xi }^{\prime },\boldsymbol{\xi }_{\ast
},I^{\prime },I_{\ast }\right. )\,\frac{h_{\beta \ast }}{M_{\beta \ast
}^{1/2}}\,\,d\boldsymbol{\xi }_{\ast }d\boldsymbol{\xi }^{\prime }d%
\boldsymbol{\xi }_{\ast }^{\prime }dI_{\ast }dI^{\prime }dI_{\ast }^{\prime }%
\text{.}
\end{eqnarray*}%
Moreover, by renaming $\left\{ \boldsymbol{\xi }_{\ast }\right\}
\leftrightarrows \left\{ \boldsymbol{\xi }^{\prime }\right\} $, 
\begin{eqnarray*}
&&\int_{\left( \mathbb{R}^{3}\times \mathbb{R}_{+}\right) ^{3}}w_{\alpha
\beta }(\boldsymbol{\xi },\boldsymbol{\xi }_{\ast },I,I_{\ast }\left\vert 
\boldsymbol{\xi }^{\prime },\boldsymbol{\xi }_{\ast }^{\prime },I^{\prime
},I_{\ast }^{\prime }\right. )\,\frac{h_{\alpha ,k}^{\prime }}{\left(
M_{\alpha }^{\prime }\right) ^{1/2}}\,d\boldsymbol{\xi }_{\ast }d\boldsymbol{%
\xi }^{\prime }d\boldsymbol{\xi }_{\ast }^{\prime }dI_{\ast }dI^{\prime
}dI_{\ast }^{\prime } \\
&=&\int_{\left( \mathbb{R}^{3}\times \mathbb{R}_{+}\right) ^{3}}w_{\alpha
\beta }(\boldsymbol{\xi },\boldsymbol{\xi }^{\prime },I,I^{\prime
}\left\vert \boldsymbol{\xi }_{\ast },\boldsymbol{\xi }_{\ast }^{\prime
},I_{\ast },I_{\ast }^{\prime }\right. )\,\frac{h_{\alpha \ast }}{M_{\alpha
\ast }^{1/2}}\,d\boldsymbol{\xi }_{\ast }d\boldsymbol{\xi }^{\prime }d%
\boldsymbol{\xi }_{\ast }^{\prime }dI_{\ast }dI^{\prime }dI_{\ast }^{\prime }%
\text{.}
\end{eqnarray*}%
It follows that%
\begin{eqnarray}
K_{\alpha }\left( h\right) &=&\int_{\mathbb{R}^{3}\times \mathbb{R}%
_{+}}k_{\alpha \beta }\left( \boldsymbol{\xi },\boldsymbol{\xi }_{\ast
},I,I_{\ast }\right) \,h_{\ast }\,d\boldsymbol{\xi }_{\ast }\text{, where } 
\notag \\
k_{\alpha \beta }h_{\ast } &=&k_{\alpha \beta }^{\left( \alpha \right)
}h_{\alpha \ast }+k_{\alpha \beta }^{\left( \beta \right) }h_{\beta \ast
}=k_{\alpha \beta }^{\left( \alpha \right) }h_{\alpha \ast }+\left(
k_{\alpha \beta 2}^{\left( \beta \right) }-k_{\alpha \beta 1}^{\left( \beta
\right) }\right) h_{\beta \ast }\text{, with}  \notag \\
k_{\alpha \beta }^{\left( \alpha \right) }(\boldsymbol{\xi },\boldsymbol{\xi 
}_{\ast },I,I_{\ast }) &=&\int_{\left( \mathbb{R}^{3}\times \mathbb{R}%
_{+}\right) ^{2}}\frac{w_{\alpha \beta }(\boldsymbol{\xi },\boldsymbol{\xi }%
^{\prime },I,I^{\prime }\left\vert \boldsymbol{\xi }_{\ast },\boldsymbol{\xi 
}_{\ast }^{\prime },I_{\ast },I_{\ast }^{\prime }\right. )}{\left( M_{\alpha
}M_{\alpha \ast }\right) ^{1/2}}\,d\boldsymbol{\xi }^{\prime }d\boldsymbol{%
\xi }_{\ast }^{\prime }dI^{\prime }dI_{\ast }^{\prime }\text{,}  \notag \\
k_{\alpha \beta 1}^{\left( \beta \right) }(\boldsymbol{\xi },\boldsymbol{\xi 
}_{\ast },I,I_{\ast }) &=&\int_{\left( \mathbb{R}^{3}\times \mathbb{R}%
_{+}\right) ^{2}}\frac{w_{\alpha \beta }(\boldsymbol{\xi },\boldsymbol{\xi }%
_{\ast },I,I_{\ast }\left\vert \boldsymbol{\xi }^{\prime },\boldsymbol{\xi }%
_{\ast }^{\prime },I^{\prime },I_{\ast }^{\prime }\right. )}{\left(
M_{\alpha }M_{\beta \ast }\right) ^{1/2}}d\boldsymbol{\xi }^{\prime }d%
\boldsymbol{\xi }_{\ast }^{\prime }dI^{\prime }dI_{\ast }^{\prime }\text{,
and}  \notag \\
k_{\alpha \beta 2}^{\left( \beta \right) }(\boldsymbol{\xi },\boldsymbol{\xi 
}_{\ast },I,I_{\ast }) &=&\int_{\left( \mathbb{R}^{3}\times \mathbb{R}%
_{+}\right) ^{2}}\frac{w_{\alpha \beta }(\boldsymbol{\xi },\boldsymbol{\xi }%
_{\ast }^{\prime },I,I_{\ast }^{\prime }\left\vert \boldsymbol{\xi }^{\prime
},\boldsymbol{\xi }_{\ast },I^{\prime },I_{\ast }\right. )}{\left( M_{\alpha
}M_{\beta \ast }\right) ^{1/2}}\,d\boldsymbol{\xi }^{\prime }d\boldsymbol{%
\xi }_{\ast }^{\prime }dI^{\prime }dI_{\ast }^{\prime }\text{.}  \label{k1}
\end{eqnarray}

By applying second relation of $\left( \ref{rel2}\right) $ and renaming $%
\left\{ \boldsymbol{\xi }^{\prime },I^{\prime }\right\} \leftrightarrows
\left\{ \boldsymbol{\xi }_{\ast }^{\prime },I_{\ast }^{\prime }\right\} $,%
\begin{eqnarray}
k_{\alpha \beta }^{\left( \alpha \right) }(\boldsymbol{\xi },\boldsymbol{\xi 
}_{\ast },I,I_{\ast }) &=&\int_{\left( \mathbb{R}^{3}\times \mathbb{R}%
_{+}\right) ^{2}}\frac{w_{\alpha \beta }(\boldsymbol{\xi },\boldsymbol{\xi }%
^{\prime },I,I^{\prime }\left\vert \boldsymbol{\xi }_{\ast },\boldsymbol{\xi 
}_{\ast }^{\prime },I_{\ast },I_{\ast }^{\prime }\right. )}{\left( M_{\alpha
}M_{\alpha \ast }\right) ^{1/2}}\,d\boldsymbol{\xi }^{\prime }d\boldsymbol{%
\xi }_{\ast }^{\prime }dI^{\prime }dI_{\ast }^{\prime }  \notag \\
&=&\int_{\left( \mathbb{R}^{3}\times \mathbb{R}_{+}\right) ^{2}}\frac{%
w_{\alpha \beta }(\boldsymbol{\xi }_{\ast },\boldsymbol{\xi }^{\prime
},I_{\ast },I^{\prime }\left\vert \boldsymbol{\xi },\boldsymbol{\xi }_{\ast
}^{\prime },I,I_{\ast }^{\prime }\right. )}{\left( M_{\alpha }M_{\alpha \ast
}\right) ^{1/2}}\,d\boldsymbol{\xi }^{\prime }d\boldsymbol{\xi }_{\ast
}^{\prime }dI^{\prime }dI_{\ast }^{\prime }  \notag \\
&=&k_{\alpha \beta }^{\left( \alpha \right) }(\boldsymbol{\xi }_{\ast },%
\boldsymbol{\xi })\text{.}  \label{sa1}
\end{eqnarray}%
Moreover,%
\begin{equation}
k_{\alpha \beta }^{\left( \beta \right) }(\boldsymbol{\xi },\boldsymbol{\xi }%
_{\ast },I,I_{\ast })=k_{\beta \alpha 1}^{\left( \alpha \right) }(%
\boldsymbol{\xi }_{\ast },\boldsymbol{\xi },I_{\ast },I)-k_{\beta \alpha
2}^{\left( \alpha \right) }(\boldsymbol{\xi }_{\ast },\boldsymbol{\xi }%
,I_{\ast },I)=k_{\beta \alpha }^{\left( \alpha \right) }(\boldsymbol{\xi }%
_{\ast },\boldsymbol{\xi },I_{\ast },I),  \label{sa2}
\end{equation}%
since, by applying first relation of $\left( \ref{rel2}\right) $ and
renaming $\left\{ \boldsymbol{\xi }^{\prime },I^{\prime }\right\}
\leftrightarrows \left\{ \boldsymbol{\xi }_{\ast }^{\prime },I_{\ast
}^{\prime }\right\} $, 
\begin{eqnarray*}
k_{\alpha \beta 1}^{\left( \beta \right) }(\boldsymbol{\xi },\boldsymbol{\xi 
}_{\ast },I,I_{\ast }) &=&\int_{\left( \mathbb{R}^{3}\times \mathbb{R}%
_{+}\right) ^{2}}\frac{w_{\beta \alpha }(\boldsymbol{\xi }_{\ast },%
\boldsymbol{\xi },I_{\ast },I\left\vert \boldsymbol{\xi }_{\ast }^{\prime },%
\boldsymbol{\xi }^{\prime },I_{\ast }^{\prime },I^{\prime }\right. )}{\left(
M_{\alpha }M_{\beta \ast }\right) ^{1/2}}\,d\boldsymbol{\xi }^{\prime }d%
\boldsymbol{\xi }_{\ast }^{\prime }dI^{\prime }dI_{\ast }^{\prime } \\
&=&\int_{\left( \mathbb{R}^{3}\times \mathbb{R}_{+}\right) ^{2}}\frac{%
w_{\beta \alpha }(\boldsymbol{\xi }_{\ast },\boldsymbol{\xi },I_{\ast
},I\left\vert \boldsymbol{\xi }^{\prime },\boldsymbol{\xi }_{\ast }^{\prime
},I^{\prime },I_{\ast }^{\prime }\right. )}{\left( M_{\alpha }M_{\beta \ast
}\right) ^{1/2}}\,d\boldsymbol{\xi }^{\prime }d\boldsymbol{\xi }_{\ast
}^{\prime }dI^{\prime }dI_{\ast }^{\prime } \\
&=&k_{\beta \alpha 1}^{\left( \alpha \right) }(\boldsymbol{\xi }_{\ast },%
\boldsymbol{\xi },I_{\ast },I)\text{,}
\end{eqnarray*}%
while, by applying the first two relations of $\left( \ref{rel2}\right) $ and then
renaming $\left\{ \boldsymbol{\xi }^{\prime },I^{\prime }\right\}
\leftrightarrows \left\{ \boldsymbol{\xi }_{\ast }^{\prime },I_{\ast
}^{\prime }\right\} $,%
\begin{eqnarray*}
k_{\alpha \beta 2}^{\left( \beta \right) }(\boldsymbol{\xi },\boldsymbol{\xi 
}_{\ast },I,I_{\ast }) &=&\int_{\left( \mathbb{R}^{3}\times \mathbb{R}%
_{+}\right) ^{2}}\frac{w_{\beta \alpha }(\boldsymbol{\xi }_{\ast }^{\prime },%
\boldsymbol{\xi },I_{\ast }^{\prime },I\left\vert \boldsymbol{\xi }_{\ast },%
\boldsymbol{\xi }^{\prime },I_{\ast },I^{\prime }\right. )}{\left( M_{\alpha
}M_{\beta \ast }\right) ^{1/2}}\,d\boldsymbol{\xi }^{\prime }d\boldsymbol{%
\xi }_{\ast }^{\prime }dI^{\prime }dI_{\ast }^{\prime } \\
&=&\int_{\left( \mathbb{R}^{3}\times \mathbb{R}_{+}\right) ^{2}}\frac{%
w_{\beta \alpha }(\boldsymbol{\xi }_{\ast },\boldsymbol{\xi }^{\prime
},I_{\ast },I^{\prime }\left\vert \boldsymbol{\xi }_{\ast }^{\prime },%
\boldsymbol{\xi },I_{\ast }^{\prime },I\right. )}{\left( M_{\alpha }M_{\beta
\ast }\right) ^{1/2}}\,d\boldsymbol{\xi }^{\prime }d\boldsymbol{\xi }_{\ast
}^{\prime }dI^{\prime }dI_{\ast }^{\prime } \\
&=&\int_{\left( \mathbb{R}^{3}\times \mathbb{R}_{+}\right) ^{2}}\frac{%
w_{\beta \alpha }(\boldsymbol{\xi }_{\ast },\boldsymbol{\xi }_{\ast
}^{\prime },I_{\ast },I_{\ast }^{\prime }\left\vert \boldsymbol{\xi }%
^{\prime },\boldsymbol{\xi },I^{\prime },I\right. )}{\left( M_{\alpha
}M_{\beta \ast }\right) ^{1/2}}\,d\boldsymbol{\xi }^{\prime }d\boldsymbol{%
\xi }_{\ast }^{\prime }dI^{\prime }dI_{\ast }^{\prime } \\
&=&k_{\beta \alpha 2}^{\left( \alpha \right) }(\boldsymbol{\xi }_{\ast },%
\boldsymbol{\xi },I_{\ast },I)\text{.}
\end{eqnarray*}


\begin{figure}[h]
\centering
\includegraphics[width=0.6\textwidth]{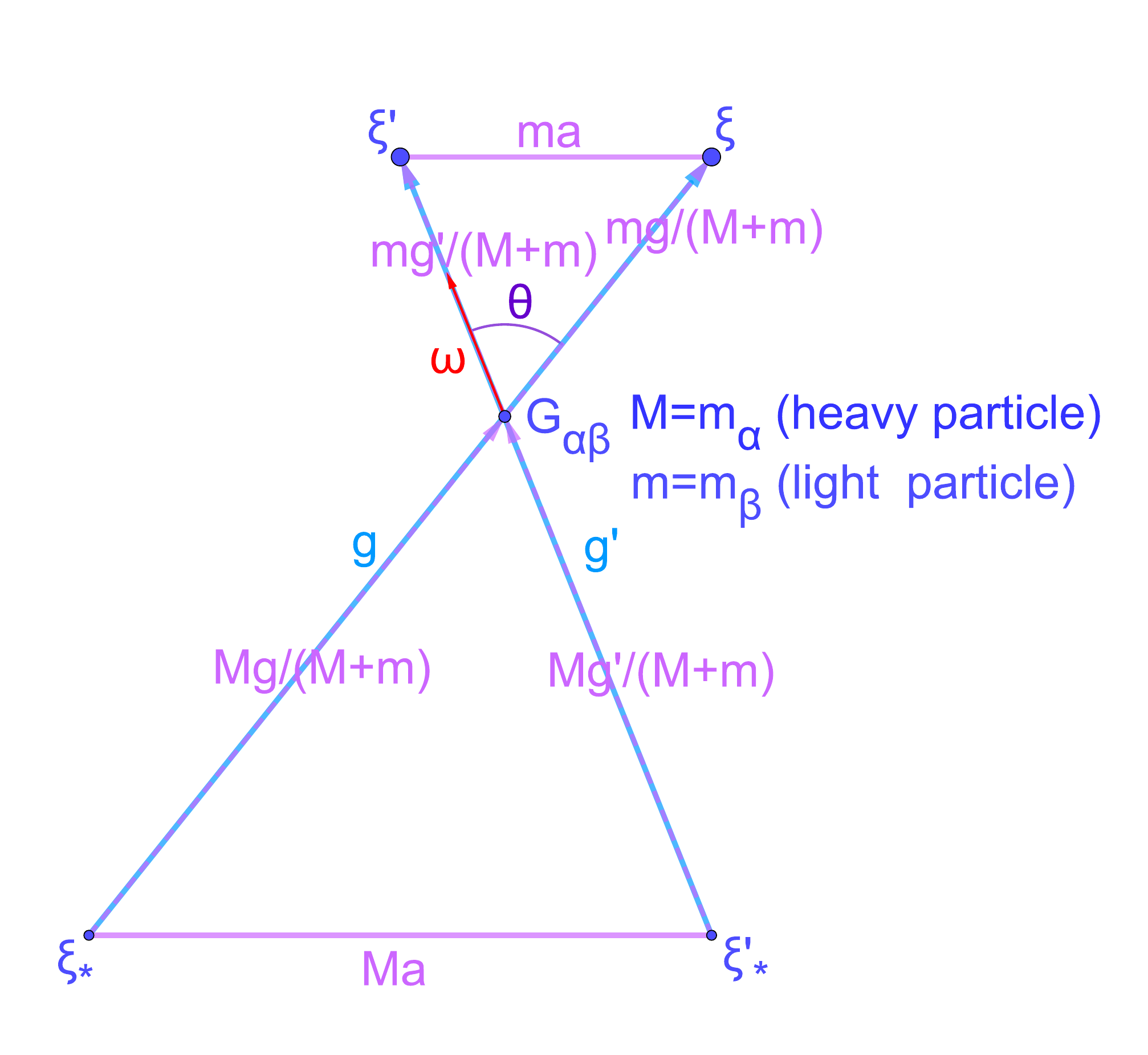}
\caption{Classical representation of an inelastic collision between particles of different species.}
\label{fig1}
\end{figure}

We now continue by proving the compactness for the three different types of
collision kernel separately. Note that, by applying the last relation in $%
\left( \ref{rel2}\right) $, $k_{\alpha \beta 2}^{\left( \beta \right) }(%
\boldsymbol{\xi },\boldsymbol{\xi }_{\ast },I,I_{\ast })=k_{\alpha \beta
2}^{\left( \alpha \right) }(\boldsymbol{\xi },\boldsymbol{\xi }_{\ast
},I,I_{\ast })$ if $\alpha =\beta $, and we will remain with only two cases
- the first two below. Even if $m_{\alpha }=m_{\beta }$, the kernels $%
k_{\alpha \beta }^{\left( \alpha \right) }(\boldsymbol{\xi },\boldsymbol{\xi 
}_{\ast },I,I_{\ast })$ and $k_{\alpha \beta 2}^{\left( \beta \right) }(%
\boldsymbol{\xi },\boldsymbol{\xi }_{\ast },I,I_{\ast })$ are structurally
equal, and we (in principle) remain with (first) two cases (the second one
twice).

\textbf{I. Compactness of }$K_{\alpha \beta 1}=\int_{\mathbb{R}^{3}\times 
\mathbb{R}_{+}}k_{\alpha \beta 1}^{\left( \beta \right) }(\boldsymbol{\xi },%
\boldsymbol{\xi }_{\ast },I,I_{\ast })\,h_{\beta \ast }\,d\boldsymbol{\xi }%
_{\ast }dI_{\ast }$ for $\left\{ \alpha ,\beta \right\} \subseteq \left\{
1,...,s\right\} $.

By a change of variables 
\begin{equation*}
\left\{ \boldsymbol{\xi }^{\prime },\boldsymbol{\xi }_{\ast }^{\prime
}\right\} \rightarrow \left\{ \left\vert \mathbf{g}^{\prime }\right\vert
=\left\vert \boldsymbol{\xi }^{\prime }-\boldsymbol{\xi }_{\ast }^{\prime
}\right\vert ,\boldsymbol{\omega }=\dfrac{\mathbf{g}^{\prime }}{\left\vert 
\mathbf{g}^{\prime }\right\vert },\mathbf{G}_{\alpha \beta }^{\prime }=%
\dfrac{m_{\alpha }\boldsymbol{\xi }^{\prime }+m_{\beta }\boldsymbol{\xi }%
_{\ast }^{\prime }}{m_{\alpha }+m_{\beta }}\right\} ,
\end{equation*}%
cf. Figure $\ref{fig1}$, noting that $\left( \ref{df1}\right) $, and using
relation $\left( \ref{M1}\right) $, expression $\left( \ref{k1}\right) $ of $%
k_{\alpha \beta 1}^{\left( \beta \right) }$ may be transformed to%
\begin{eqnarray*}
&&k_{\alpha \beta 1}^{\left( \beta \right) }(\boldsymbol{\xi },\boldsymbol{%
\xi }_{\ast },I,I_{\ast }) \\
&=&\int_{\left( \mathbb{R}^{3}\times \mathbb{R}_{+}\right) ^{2}}\frac{\left(
M_{\alpha }^{\prime }M_{\beta \ast }^{\prime }\right) ^{1/2}W_{\alpha \beta }%
}{\left( II^{\prime }\right) ^{\delta ^{\left( \alpha \right) }/4-1/2}\left(
I_{\ast }I_{\ast }^{\prime }\right) ^{\delta ^{\left( \beta \right) }/4-1/2}}%
\,d\boldsymbol{\xi }^{\prime }d\boldsymbol{\xi }_{\ast }^{\prime }dI^{\prime
}dI_{\ast }^{\prime } \\
&=&\int_{\mathbb{R}^{3}\times \mathbb{R}_{+}^{3}\times \mathbb{S}^{2}}\frac{%
\left( M_{\alpha }^{\prime }M_{\beta \ast }^{\prime }\right) ^{1/2}W_{\alpha
\beta }\,}{\left( II^{\prime }\right) ^{\delta ^{\left( \alpha \right)
}/4-1/2}\left( I_{\ast }I_{\ast }^{\prime }\right) ^{\delta ^{\left( \beta
\right) }/4-1/2}}\left\vert \mathbf{g}^{\prime }\right\vert ^{2}d\mathbf{G}%
_{\alpha \beta }^{\prime }d\left\vert \mathbf{g}^{\prime }\right\vert d%
\boldsymbol{\omega \,}dI^{\prime }dI_{\ast }^{\prime } \\
&=&\frac{\left( M_{\alpha }M_{\beta \ast }\right) ^{1/2}}{I^{\delta ^{\left(
\alpha \right) }/2-1}I_{\ast }^{\delta ^{\left( \beta \right) }/2-1}}%
\left\vert \mathbf{g}\right\vert \int_{\mathbb{S}^{2}\times \mathbb{R}%
_{+}^{2}}\mathbf{1}_{\left\vert \mathbf{g}\right\vert ^{2}>2\widetilde{%
\Delta }I}\sigma _{\alpha \beta }\,d\boldsymbol{\omega }\,dI^{\prime
}dI_{\ast }^{\prime }\text{.}
\end{eqnarray*}%
Since $E_{\alpha \beta }\geq \Psi _{\alpha \beta }$, it follows, by
assumption $\left( \ref{est1}\right) $, that 
\begin{eqnarray*}
&&\left( k_{\alpha \beta 1}^{\left( \beta \right) }(\boldsymbol{\xi },%
\boldsymbol{\xi }_{\ast },I,I_{\ast })\right) ^{2} \\
&\leq &CM_{\alpha }M_{\beta \ast }\frac{I^{\delta ^{\left( \alpha \right)
}/2-1}I_{\ast }^{\delta ^{\left( \beta \right) }/2-1}}{\left\vert \mathbf{g}%
\right\vert ^{2}} \\
&&\times \left( \int_{\mathbb{S}^{2}\times \mathbb{R}_{+}^{2}}\Upsilon
_{\alpha \beta }\left( \Psi _{\alpha \beta }+\Psi _{\alpha \beta }^{\gamma
/2}\right) \mathbf{1}_{I^{\prime }\leq \mathcal{E}_{\alpha \beta }}\mathbf{1}%
_{I_{\ast }^{\prime }\leq \mathcal{E}_{\alpha \beta }^{\ast }}\,d\boldsymbol{%
\omega }\,dI^{\prime }dI_{\ast }^{\prime }\right) ^{2} \\
&\leq &C\frac{M_{\alpha }M_{\beta \ast }}{I^{\delta ^{\left( \alpha \right)
}/2-1}I_{\ast }^{\delta ^{\left( \beta \right) }/2-1}}\frac{I^{\delta
^{\left( \alpha \right) }-2}I_{\ast }^{\delta ^{\left( \beta \right) }-2}}{%
\left\vert \mathbf{g}\right\vert ^{2}}\left( E_{\alpha \beta }+E_{\alpha
\beta }^{\gamma /2}\right) ^{2} \\
&&\times \left( \int_{0}^{\mathcal{E}_{\alpha \beta }}\frac{\left( I^{\prime
}\right) ^{\delta ^{\left( \alpha \right) }/2-1}}{\mathcal{E}_{\alpha \beta
}^{\delta ^{\left( \alpha \right) }/2}}\,dI^{\prime }\right) ^{2}\left(
\int_{0}^{\mathcal{E}_{\alpha \beta }^{\ast }}\frac{\left( I_{\ast }^{\prime
}\right) ^{\delta ^{\left( \beta \right) }/2-1}}{\left( \mathcal{E}_{\alpha
\beta }^{\ast }\right) ^{\delta ^{\left( \beta \right) }/2}}\,dI_{\ast
}^{\prime }\right) ^{2} \\
&=&C\frac{M_{\alpha }M_{\beta \ast }}{I^{\delta ^{\left( \alpha \right)
}/2-1}I_{\ast }^{\delta ^{\left( \beta \right) }/2-1}}\frac{I^{\delta
^{\left( \alpha \right) }-2}I_{\ast }^{\delta ^{\left( \beta \right) }-2}}{%
\left\vert \mathbf{g}\right\vert ^{2}}\left( E_{\alpha \beta }+E_{\alpha
\beta }^{\gamma /2}\right) ^{2}.
\end{eqnarray*}%
Then, by assumption $\left( \ref{est1}\right) $ and\textbf{\ } 
\begin{equation*}
m_{\alpha }\frac{\left\vert \boldsymbol{\xi }\right\vert ^{2}}{2}+m_{\beta
}\frac{\left\vert \boldsymbol{\xi }_{\ast }\right\vert ^{2}}{2}+I+I_{\ast }=%
\frac{m_{\alpha }+m_{\beta }}{2}\left\vert \mathbf{G}_{\alpha \beta
}\right\vert ^{2}+E_{\alpha \beta }\text{,}
\end{equation*}%
where $E_{\alpha \beta }=\mu_{\alpha \beta }\left\vert \mathbf{g}\right\vert ^{2}/2+I+I_{\ast }$, the bound%
\begin{eqnarray}
&&\left( k_{\alpha \beta 1}^{\left( \beta \right) }(\boldsymbol{\xi },%
\boldsymbol{\xi }_{\ast },I,I_{\ast })\right) ^{2}  \notag \\
&\leq &Ce^{-\left( m_{\alpha }+m_{\beta }\right) \left\vert \mathbf{G}%
_{\alpha \beta }\right\vert ^{2}/2-E_{\alpha \beta }}\frac{I^{\delta
^{\left( \alpha \right) }/2-1}I_{\ast }^{\delta ^{\left( \beta \right) }/2-1}%
}{\left\vert \mathbf{g}\right\vert ^{2}}\left( 1+E_{\alpha \beta }^{2}\right)
 \notag \\
&\leq &Ce^{-\left( m_{\alpha }+m_{\beta }\right) \left\vert \mathbf{G}%
_{\alpha \beta }\right\vert ^{2}/2-E_{\alpha \beta }}\frac{\left(
1+\left\vert \mathbf{g}\right\vert ^{2}\right) ^{2}}{\left\vert \mathbf{g}%
\right\vert ^{2}}\notag \\
&&\times \left( 1+I\right) ^{2}I^{\delta ^{\left( \alpha \right)
}/2-1}\left( 1+I_{\ast }\right) ^{2}I_{\ast }^{\delta ^{\left( \beta \right)
}/2-1} \label{b1}
\end{eqnarray}%
may be obtained. Hence, by applying the bound $\left( \ref{b1}\right) $ and
first changing variables of integration $\left\{ \boldsymbol{\xi },%
\boldsymbol{\xi }_{\ast }\right\} \rightarrow \left\{ \mathbf{g},\mathbf{G}%
_{\alpha \beta }\right\} $, with unitary Jacobian, and then to spherical
coordinates,%
\begin{eqnarray*}
&&\int_{\left( \mathbb{R}^{3}\times \mathbb{R}_{+}\right) ^{2}}\left(
k_{\alpha \beta 1}^{\left( \beta \right) }(\boldsymbol{\xi },\boldsymbol{\xi 
}_{\ast },I,I_{\ast })\right) ^{2}d\boldsymbol{\xi \,}d\boldsymbol{\xi }%
_{\ast }dI\boldsymbol{\,}dI_{\ast } \\
&\leq &C\int_{\left( \mathbb{R}^{3}\right) ^{2}}e^{-\left( m_{\alpha
}+m_{\beta }\right) \left\vert \mathbf{G}_{\alpha \beta }\right\vert
^{2}/2}e^{-\mu_{\alpha \beta }\left\vert \mathbf{g}\right\vert ^{2}/
2 }\frac{1+\left\vert \mathbf{g}%
\right\vert ^{4}}{\left\vert \mathbf{g}\right\vert ^{2}}d\mathbf{g}%
\boldsymbol{\,}d\boldsymbol{\mathbf{G}_{\alpha \beta }\,} \\
&&\times \int_{0}^{\infty }\left( 1+I\right) ^{2}e^{-I}I^{\delta ^{\left(
\alpha \right) }-2}dI\int_{0}^{\infty }\left( 1+I_{\ast }\right)
^{2}e^{-I_{\ast }}I_{\ast }^{\delta ^{\left( \beta \right) }-2}dI_{\ast } \\
&=&C\int_{0}^{\infty }e^{-\mu_{\alpha \beta }r^{2}/4}\left( 1+r^{4}\right)
dr\int_{0}^{\infty }R^{2}e^{-R^{2}}dR=C
\end{eqnarray*}%
Therefore,%
\begin{equation*}
K_{\alpha \beta 1}=\int_{\mathbb{R}^{3}\times \mathbb{R}_{+}}k_{\alpha \beta
1}^{\left( \beta \right) }(\boldsymbol{\xi },\boldsymbol{\xi }_{\ast
},I,I_{\ast })\,h_{\beta \ast }\,d\boldsymbol{\xi }_{\ast }dI_{\ast }
\end{equation*}%
are Hilbert-Schmidt integral operators and as such compact on $L^{2}\left( d%
\mathbf{Z}_{\alpha }\right) $, see, e.g., Theorem 7.83 in \cite%
{RenardyRogers}, for $\left\{ \alpha ,\beta \right\} \subseteq \left\{
1,...,s\right\} $.


\begin{figure}[h]
\centering
\includegraphics[width=0.5\textwidth]{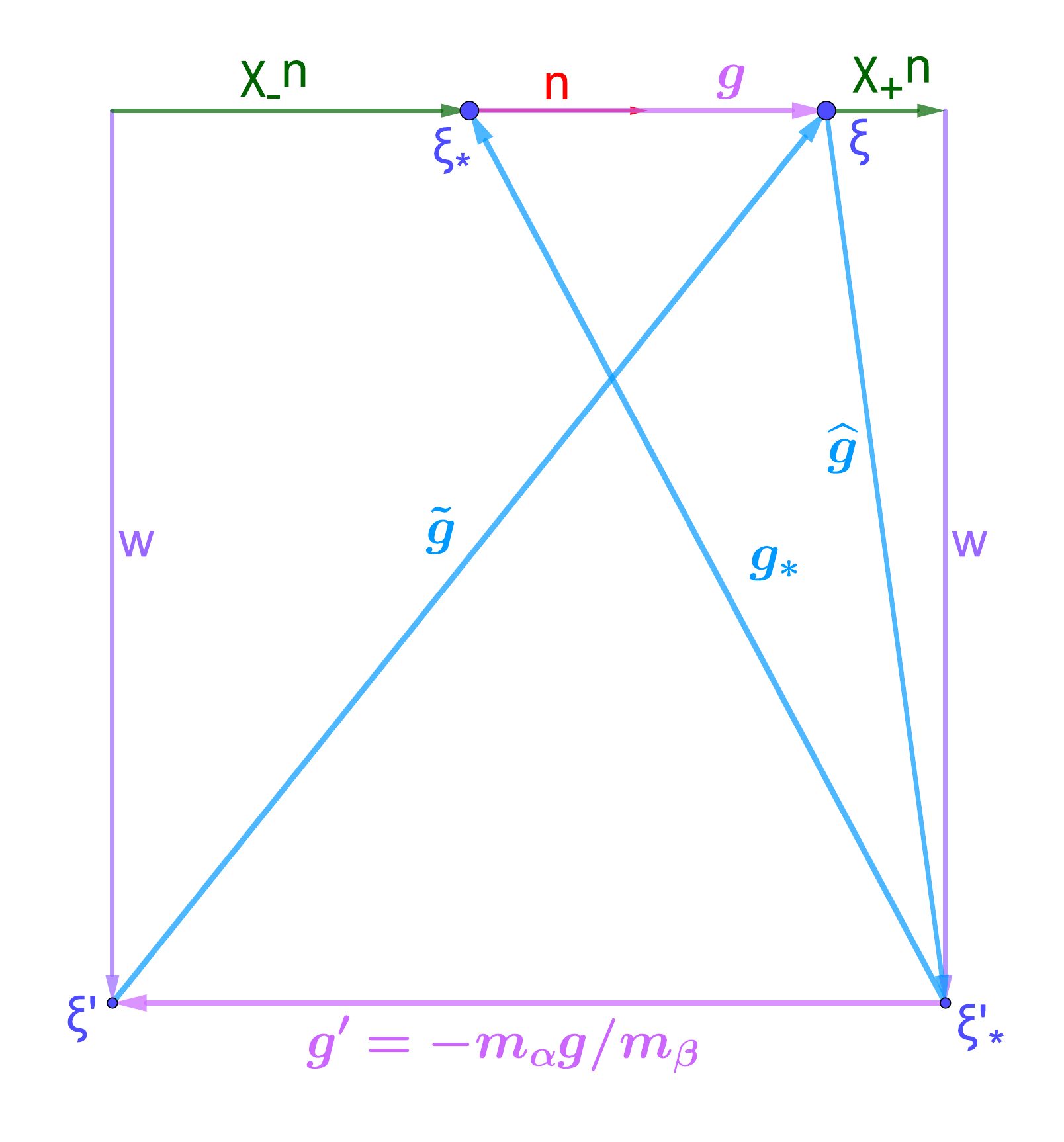}
\caption{Typical collision of $K_{\protect\alpha \protect\beta 3}$.}
\label{fig2}
\end{figure}

\textbf{II. Compactness of }$K_{\alpha \beta 3}=\int_{\mathbb{R}^{3}\times 
\mathbb{R}_{+}}k_{\alpha \beta }^{\left( \alpha \right) }(\boldsymbol{\xi },%
\boldsymbol{\xi }_{\ast },I,I_{\ast })\,h_{\alpha \ast }\,d\boldsymbol{\xi }%
_{\ast }dI_{\ast }$ for $\left\{ \alpha ,\beta \right\} \subseteq \left\{
1,...,s\right\} $.

Note that, cf. Figure \ref{fig2},%
\begin{eqnarray*}
&&W_{\alpha \beta }(\boldsymbol{\xi },\boldsymbol{\xi }^{\prime
},I,I^{\prime }\left\vert \boldsymbol{\xi }_{\ast },\boldsymbol{\xi }_{\ast
}^{\prime },I_{\ast },I_{\ast }^{\prime }\right. ) \\
&=&\left( m_{\alpha }+m_{\beta }\right) ^{2}m_{\alpha }m_{\beta }I^{\delta
^{\left( \alpha \right) }/2-1}\left( I^{\prime }\right) ^{\delta ^{\left(
\beta \right) }/2-1}\widetilde{\sigma }_{\alpha \beta }\delta _{3}\left(
m_{\alpha }\mathbf{g}+m_{\beta }\mathbf{g}^{\prime }\right) \\
&&\times \frac{\left\vert \widetilde{\mathbf{g}}\right\vert }{\left\vert 
\mathbf{g}_{\ast }\right\vert }\delta _{1}\left( \left\vert \mathbf{g}%
\right\vert m_{\alpha }\left( \chi -\frac{m_{\alpha }-m_{\beta }}{2m_{\beta }%
}\left\vert \mathbf{g}\right\vert \right) -\Delta I_{\ast }\right)
\end{eqnarray*}%
\begin{eqnarray*}
&=&\frac{\left\vert \widetilde{\mathbf{g}}\right\vert \!\left( m_{\alpha
}+m_{\beta }\right) ^{2}}{\left\vert \mathbf{g}_{\ast }\right\vert
\left\vert \mathbf{g}\right\vert m_{\beta }^{2}}I^{\delta ^{\left( \alpha
\right) }/2-1}\left( I^{\prime }\right) ^{\delta ^{\left( \beta \right)
}/2-1}\widetilde{\sigma }_{\alpha \beta }\delta _{3}\left( \frac{m_{\alpha }%
}{m_{\beta }}\mathbf{g}+\mathbf{g}^{\prime }\right) \\
&&\times \delta _{1}\left( \chi -\frac{m_{\alpha }-m_{\beta }}{2m_{\beta }}%
\left\vert \mathbf{g}\right\vert -\frac{\Delta I_{\ast }}{m_{\alpha
}\left\vert \mathbf{g}\right\vert }\right) \text{,}
\end{eqnarray*}%
where $\widetilde{\sigma }_{\alpha \beta }=\sigma _{\alpha \beta }\left(
\left\vert \widetilde{\mathbf{g}}\right\vert ,\dfrac{\widetilde{\mathbf{g}}%
\cdot \mathbf{g}_{\ast }}{\left\vert \widetilde{\mathbf{g}}\right\vert
\left\vert \mathbf{g}_{\ast }\right\vert },I,I^{\prime },I_{\ast },I_{\ast
}^{\prime }\right) $, $\mathbf{g}=\boldsymbol{\xi }-\boldsymbol{\xi }_{\ast
} $, $\mathbf{g}^{\prime }=\boldsymbol{\xi }^{\prime }-\boldsymbol{\xi }%
_{\ast }^{\prime }$, $\widetilde{\mathbf{g}}=\boldsymbol{\xi }-\boldsymbol{%
\xi }^{\prime }$, $\mathbf{g}_{\ast }=\boldsymbol{\xi }_{\ast }-\boldsymbol{%
\xi }_{\ast }^{\prime }$, $\Delta I_{\ast }=I_{\ast }+I_{\ast }^{\prime
}-I-I^{\prime }$, and $\chi =\left( \boldsymbol{\xi }_{\ast }^{\prime }-%
\boldsymbol{\xi }\right) \cdot \mathbf{n}$, with $\mathbf{n}=\dfrac{\mathbf{g%
}}{\left\vert \mathbf{g}\right\vert }$. Then by performing the following change of
variables $\left\{ \boldsymbol{\xi }^{\prime },\boldsymbol{\xi }_{\ast
}^{\prime }\right\} \rightarrow \left\{ \mathbf{g}^{\prime }=\boldsymbol{\xi 
}^{\prime }-\boldsymbol{\xi }_{\ast }^{\prime },~\widehat{\mathbf{g}}=%
\boldsymbol{\xi }_{\ast }^{\prime }-\boldsymbol{\xi }\right\} $, with%
\begin{equation*}
d\boldsymbol{\xi }^{\prime }d\boldsymbol{\xi }_{\ast }^{\prime }=d\mathbf{g}%
^{\prime }d\widehat{\mathbf{g}}=d\mathbf{g}^{\prime }d\chi d\mathbf{w}\text{%
, with }\mathbf{w}=\boldsymbol{\xi }_{\ast }^{\prime }-\boldsymbol{\xi }%
-\chi \mathbf{n}\text{,}
\end{equation*}%
the expression $\left( \ref{k1}\right) $ of $k_{\alpha \beta }^{\left(
\alpha \right) }$ may be rewritten in the following way%
\begin{eqnarray*}
&&k_{\alpha \beta }^{\left( \alpha \right) }(\boldsymbol{\xi },\boldsymbol{%
\xi }_{\ast },I,I_{\ast }) \\
&=&\int_{\left( \mathbb{R}^{3}\times \mathbb{R}_{+}\right) ^{2}}\left(
M_{\beta }^{\prime }M_{\beta \ast }^{\prime }\right) ^{1/2}\frac{W_{\alpha
\beta }(\boldsymbol{\xi },\boldsymbol{\xi }^{\prime },I,I^{\prime
}\left\vert \boldsymbol{\xi }_{\ast },\boldsymbol{\xi }_{\ast }^{\prime
},I_{\ast },I_{\ast }^{\prime }\right. )}{\left( II_{\ast }\right) ^{\delta
^{\left( \alpha \right) }/4-1/2}\left( I^{\prime }I_{\ast }^{\prime }\right)
^{\delta ^{\left( \beta \right) }/4-1/2}}d\mathbf{g}^{\prime }d\widehat{%
\mathbf{g}}dI^{\prime }dI_{\ast }^{\prime } \\
&=&\frac{I^{\delta ^{\left( \alpha \right) }/4-1/2}}{I_{\ast }^{\delta
^{\left( \alpha \right) }/4-1/2}}\int_{\left( \mathbb{R}^{3}\right) ^{\perp
_{\mathbf{n}}}\times \mathbb{R}_{+}^{2}}\frac{\left( m_{\alpha }+m_{\beta
}\right) ^{2}}{m_{\beta }^{2}}\frac{\left\vert \widetilde{\mathbf{g}}%
\right\vert \left( M_{\beta }^{\prime }M_{\beta \ast }^{\prime }\right)
^{1/2}}{\left\vert \mathbf{g}_{\ast }\right\vert \left\vert \mathbf{g}%
\right\vert }\frac{\left( I^{\prime }\right) ^{\delta ^{\left( \beta \right)
}/4-1/2}}{\left( I^{\prime }\right) _{\ast }^{\delta ^{\left( \beta \right)
}/4-1/2}} \\
&&\times \mathbf{1}_{\mu_{\alpha \beta }\left\vert \widetilde{\mathbf{g}}\right\vert ^{2}>2%
\Delta I_{\ast }}\widetilde{\sigma }_{\alpha \beta }d\mathbf{w}%
dI^{\prime }dI_{\ast }^{\prime }\text{.}
\end{eqnarray*}%
Here, see Figure $\ref{fig2}$,%
\begin{equation*}
\left\{ 
\begin{array}{l}
\boldsymbol{\xi }^{\prime }=\boldsymbol{\xi }_{\ast }+\mathbf{w}-\chi _{-}%
\mathbf{n} \\ 
\boldsymbol{\xi }_{\ast }^{\prime }=\boldsymbol{\xi }+\mathbf{w}-\chi _{+}%
\mathbf{n}%
\end{array}%
\right. \text{, with }\chi _{\pm }=\frac{\Delta I_{\ast }}{m_{\alpha
}\left\vert \mathbf{g}\right\vert }\pm \frac{m_{\alpha }-m_{\beta }}{%
2m_{\beta }}\left\vert \mathbf{g}\right\vert \text{,}
\end{equation*}%
implying that%
\begin{eqnarray}
&&\frac{\left\vert \boldsymbol{\xi }^{\prime }\right\vert ^{2}}{2}+\frac{%
\left\vert \boldsymbol{\xi }_{\ast }^{\prime }\right\vert ^{2}}{2}  \notag \\
&=&\left\vert \frac{\boldsymbol{\xi +\xi }_{\ast }}{2}-\frac{\Delta I_{\ast }%
}{m_{\alpha }\left\vert \mathbf{g}\right\vert }\mathbf{n}+\mathbf{w}%
\right\vert ^{2}+\frac{m_{\alpha }^{2}}{4m_{\beta }^{2}}\left\vert 
\boldsymbol{\xi }-\boldsymbol{\xi }_{\ast }\right\vert ^{2}  \notag \\
&=&\left\vert \frac{\left( \boldsymbol{\xi +\xi }_{\ast }\right) _{\perp _{%
\boldsymbol{n}}}}{2}+\mathbf{w}\right\vert ^{2}+\left( \frac{\left( 
\boldsymbol{\xi +\xi }_{\ast }\right) _{\mathbf{n}}}{2}-\frac{\Delta I_{\ast
}}{m_{\alpha }\left\vert \mathbf{g}\right\vert }\right) ^{2}+\frac{m_{\alpha
}^{2}}{4m_{\beta }^{2}}\left\vert \mathbf{g}\right\vert ^{2}   \notag \\
&=&\left\vert \frac{\left( \boldsymbol{\xi +\xi }_{\ast }\right) _{\perp _{%
\boldsymbol{n}}}}{2}+\mathbf{w}\right\vert ^{2}+\frac{\left( m_{\alpha
}\left( \left\vert \boldsymbol{\xi }_{\ast }\right\vert ^{2}-\left\vert 
\boldsymbol{\xi }\right\vert ^{2}\right) +2\Delta I_{\ast }\right) ^{2}}{%
4m_{\alpha }^{2}\left\vert \mathbf{g}\right\vert ^{2}}+\frac{m_{\alpha }^{2}%
}{4m_{\beta }^{2}}\left\vert \mathbf{g}\right\vert ^{2},  \notag \\ \label{eq1}
\end{eqnarray}%
where$\!\!$%
\begin{equation*}
\left( \boldsymbol{\xi +\xi }_{\ast }\right) _{\mathbf{n}}=\left( 
\boldsymbol{\xi +\xi }_{\ast }\right) \cdot \mathbf{n}=\frac{\left\vert 
\boldsymbol{\xi }\right\vert ^{2}-\left\vert \boldsymbol{\xi }_{\ast
}\right\vert ^{2}}{\left\vert \boldsymbol{\xi }-\boldsymbol{\xi }_{\ast
}\right\vert },\ \left( \boldsymbol{\xi +\xi }_{\ast }\right) _{\perp _{%
\boldsymbol{n}}}=\boldsymbol{\xi +\xi }_{\ast }-\left( \boldsymbol{\xi +\xi }%
_{\ast }\right) _{\mathbf{n}}\mathbf{n}.
\end{equation*}

Let $0\leq \varpi \leq 1$, with $\varpi \equiv 0$ if the species $a_{\alpha
} $, $\alpha \in \left\{ 1,...,s_{0}\right\} $, is monatomic. Note that%
\begin{equation}
\frac{\left( II_{\ast }\right) ^{\delta ^{\left( \alpha \right) }/4-1/2}}{%
\widetilde{\mathcal{E}}_{\alpha \beta }^{\delta ^{\left( \alpha \right) }/2}}%
\leq \frac{C}{\widetilde{\mathcal{E}}_{\alpha \beta }^{\varpi }I^{\kappa
}I_{\ast }^{1-\kappa -\varpi }}\leq \frac{C}{\left\vert \mathbf{g}%
\right\vert ^{2\varpi }I^{\kappa }I_{\ast }^{1-\kappa -\varpi }}\text{ for }%
0\leq \kappa \leq 1-\varpi \text{.}  \label{b5}
\end{equation}%
Here and below $\widetilde{\mathcal{E}}_{\alpha \beta }=\widetilde{\mathcal{E}}_{\beta
\alpha }^{\ast }=1$ if $\alpha \in \left\{ 1,...,s_{0}\right\} $, while if $\alpha \in \left\{ s_{0}+1,...,s\right\} $ and $\beta \in \left\{
1,...,s_{0}\right\} $, then $%
\widetilde{\mathcal{E}}_{\beta \alpha }^{\ast }=\mu_{\alpha \beta }\left\vert \widetilde{\mathbf{g}}%
\right\vert ^{2}/2+I^{\prime }=\mu_{\alpha \beta }\left\vert \mathbf{g}_{\ast }\right\vert ^{2}/2+I_{\ast
}^{\prime }$ and $\widetilde{\mathcal{E}}_{\alpha \beta }=\mu_{\alpha \beta }\left\vert \widetilde{%
\mathbf{g}}\right\vert ^{2}/2+I=\mu_{\alpha \beta }\left\vert \mathbf{g}_{\ast }\right\vert ^{2}/2+I_{\ast }$, and if $%
\left\{ \alpha ,\beta \right\} \subset \left\{ s_{0}+1,...,s\right\} $, then $\widetilde{\mathcal{E}}_{\alpha \beta }=%
\widetilde{\mathcal{E}}_{\alpha \beta }^{\ast }=\mu_{\alpha \beta }\left\vert \widetilde{\mathbf{g}}%
\right\vert ^{2}/2+I+I^{\prime }=\mu_{\alpha \beta }\left\vert \mathbf{g}_{\ast }\right\vert
^{2}/2+I_{\ast }+I_{\ast }^{\prime }$.

Denoting $\left( \mathbb{R}^{3}\right) ^{\perp _{\mathbf{n}}}=\left\{ 
\mathbf{w}\in \mathbb{R}^{3};\text{ }\mathbf{w}\perp \mathbf{n}\right\} $,
we obtain the bound%
\begin{eqnarray}
&&\int_{\left( \mathbb{R}^{3}\right) ^{\perp _{\mathbf{n}}}}\left( 1+\frac{%
\mathbf{1}_{\mu_{\alpha \beta }\left\vert \widetilde{\mathbf{g}}\right\vert ^{2}>2\Delta I_{\ast }}}{\widetilde{\Psi }_{\alpha \beta }^{1-\gamma /2}}\right)
\exp \left( -\frac{m_{\beta }}{2}\left\vert \frac{\left( \boldsymbol{\xi
+\xi }_{\ast }\right) _{\perp _{\boldsymbol{n}}}}{2}+\mathbf{w}\right\vert
^{2}\right) d\mathbf{w}  \notag \\
&\leq &C\int_{\left( \mathbb{R}^{3}\right) ^{\perp _{\mathbf{n}}}}\left(
1+\left\vert \mathbf{w}\right\vert ^{\gamma -2}\right) \exp \left( -\frac{%
m_{\beta }}{2}\left\vert \frac{\left( \boldsymbol{\xi +\xi }_{\ast }\right)
_{\perp _{\boldsymbol{n}}}}{2}+\mathbf{w}\right\vert ^{2}\right) \,d\mathbf{%
w\,}  \notag \\
&\leq &C\left( \int\limits_{\left\vert \mathbf{w}\right\vert \leq 1}1+\left\vert 
\mathbf{w}\right\vert ^{\gamma -2}d\mathbf{w}+2\int\limits_{\left\vert \mathbf{w}%
\right\vert \geq 1}\exp \left( -\frac{m_{\beta }}{2}\left\vert \frac{\left( 
\boldsymbol{\xi +\xi }_{\ast }\right) _{\perp _{\boldsymbol{n}}}}{2}+\mathbf{%
w}\right\vert ^{2}\right) d\mathbf{w}\right)  \notag \\
&\leq &C\left( \int_{\left\vert \mathbf{w}\right\vert \leq 1}1+\left\vert 
\mathbf{w}\right\vert ^{\gamma -2}\,d\mathbf{w}\,+\int_{\left( \mathbb{R}%
^{3}\right) ^{\perp _{\mathbf{n}}}}e^{-\left\vert \widetilde{\mathbf{w}}%
\right\vert ^{2}}\,d\widetilde{\mathbf{w}}\right)  \notag \\
&=&C\left( \int_{0}^{1}1+r^{\gamma -1}\,dr+\int_{0}^{\infty
}re^{-r^{2}}\,dr\right) =C\text{, where }\widetilde{\Psi }_{\alpha \beta
}=\left\vert \widetilde{\mathbf{g}}\right\vert \left\vert \mathbf{g}_{\ast
}\right\vert \text{.}  \label{b5a}
\end{eqnarray}%
By the bounds $\left( \ref{b5}\right) $ and $\left( \ref{b5a}\right) $, and
assumption $\left( \ref{est1}\right) $, for any number $\kappa $, $0\leq
\kappa \leq 1-\varpi $,%
\begin{eqnarray}
&&k_{\alpha \beta }^{\left( \alpha \right) }(\boldsymbol{\xi },\boldsymbol{%
\xi }_{\ast },I,I_{\ast })  \notag \\
&\leq &\frac{C}{\left\vert \mathbf{g}\right\vert ^{1+2\varpi }}\int\limits_{\mathbb{%
R}_{+}^{2}}\exp \left( -m_{\beta }\frac{\left( m_{\alpha }\left( \left\vert 
\boldsymbol{\xi }_{\ast }\right\vert ^{2}-\left\vert \boldsymbol{\xi }%
\right\vert ^{2}\right) +2\Delta I_{\ast }\right) ^{2}}{8m_{\alpha
}^{2}\left\vert \mathbf{g}\right\vert ^{2}}-\frac{m_{\alpha }^{2}}{8m_{\beta
}}\left\vert \mathbf{g}\right\vert ^{2}\right)  \notag \\
&&\times \int\limits_{\left( \mathbb{R}^{3}\right) ^{\perp _{\mathbf{n}}}}\left( 1+%
\frac{\mathbf{1}_{\mu_{\alpha \beta }\left\vert \widetilde{\mathbf{g}}\right\vert ^{2}>2\Delta I_{\ast }}}{\widetilde{\Psi }_{\alpha \beta }^{1-\gamma
/2}}\right) \exp \left( -\frac{m_{\beta }}{2}\left\vert \frac{\left( 
\boldsymbol{\xi +\xi }_{\ast }\right) _{\perp _{\boldsymbol{n}}}}{2}+\mathbf{%
w}\right\vert ^{2}\right) d\mathbf{w}  \notag \\
&&\,\times \mathbf{\,}e^{-\left( I^{\prime }+I_{\ast }^{\prime }\right) /2}%
\frac{\left( I^{\prime }I_{\ast }^{\prime }\right) ^{\delta ^{\left( \beta
\right) }/2-1}}{\widetilde{\mathcal{E}}_{\alpha \beta }^{\delta ^{\left(
\beta \right) }/2}}dI^{\prime }dI_{\ast }^{\prime }\frac{1}{I^{\kappa
}I_{\ast }^{1-\kappa -\varpi }}  \notag \\
&\leq &\frac{C}{\left\vert \mathbf{g}\right\vert ^{1+2\varpi }}\frac{1}{%
I^{\kappa }I_{\ast }^{1-\kappa -\varpi }}\int\limits_{\mathbb{R}_{+}^{2}}\exp
\left( -\frac{m_{\beta }}{8}\left( \left\vert \mathbf{g}\right\vert
-2\left\vert \boldsymbol{\xi }\right\vert \cos \varphi +2\chi _{\alpha
}\right) ^{2}-\frac{m_{\alpha }^{2}}{8m_{\beta }}\left\vert \mathbf{g}%
\right\vert ^{2}\right)  \notag \\
&&\times \frac{e^{-\left( I^{\prime }+I_{\ast }^{\prime }\right) /2}}{\left(
I^{\prime }I_{\ast }^{\prime }\right) ^{1-\delta ^{\left( \beta \right) }/4}}%
dI^{\prime }dI_{\ast }^{\prime }\text{, where }\chi _{\alpha }=\frac{\Delta
I_{\ast }}{m_{\alpha }\left\vert \mathbf{g}\right\vert }\text{ and }\cos
\varphi =\mathbf{n}\cdot \frac{\boldsymbol{\xi }}{\left\vert \boldsymbol{\xi 
}\right\vert }\text{.}  \label{b2a}
\end{eqnarray}%
For $0\leq \kappa \leq 1-\varpi $, by the bound $\left( \ref{b2a}\right) $
and the Cauchy-Schwarz inequality, 
\begin{eqnarray}
&&\left( k_{\alpha \beta }^{\left( \alpha \right) }(\boldsymbol{\xi },%
\boldsymbol{\xi }_{\ast },I,I_{\ast })\right) ^{2}  \notag \\
&\leq &\frac{C}{\left\vert \mathbf{g}\right\vert ^{2+4\varpi }}\frac{1}{%
I^{2\kappa }I_{\ast }^{2-2\kappa -2\varpi }}\int_{\mathbb{R}_{+}^{2}}\mathbf{%
\,}\frac{e^{-\left( I^{\prime }+I_{\ast }^{\prime }\right) /2}}{\left(
I^{\prime }I_{\ast }^{\prime }\right) ^{1-\delta ^{\left( \beta \right) }/4}}%
dI^{\prime }dI_{\ast }^{\prime }\int_{\mathbb{R}_{+}^{2}}\frac{e^{-\left(
I^{\prime }+I_{\ast }^{\prime }\right) /2}}{\left( I^{\prime }I_{\ast
}^{\prime }\right) ^{1-\delta ^{\left( \beta \right) }/4}}  \notag \\
&&\times \exp \left( -\frac{m_{\beta }}{4}\left( \left\vert \mathbf{g}%
\right\vert -2\left\vert \boldsymbol{\xi }\right\vert \cos \varphi +2\chi
_{\alpha }\right) ^{2}-\frac{m_{\alpha }^{2}}{4m_{\beta }}\left\vert \mathbf{%
g}\right\vert ^{2}\right) dI^{\prime }dI_{\ast }^{\prime }  \notag \\
&=&\frac{C}{\left\vert \mathbf{g}\right\vert ^{2+4\varpi }}\int_{\mathbb{R}%
_{+}^{2}}\exp \left( -\frac{m_{\beta }}{4}\left( \left\vert \mathbf{g}%
\right\vert -2\left\vert \boldsymbol{\xi }\right\vert \cos \varphi +2\chi
_{\alpha }\right) ^{2}-\frac{m_{\alpha }^{2}}{4m_{\beta }}\left\vert \mathbf{%
g}\right\vert ^{2}\right)  \notag \\
&&\times \frac{e^{-\left( I^{\prime }+I_{\ast }^{\prime }\right) /2}}{\left(
I^{\prime }I_{\ast }^{\prime }\right) ^{1-\delta ^{\left( \beta \right) }/4}}%
dI^{\prime }dI_{\ast }^{\prime }\frac{1}{I^{2\kappa }I_{\ast }^{2-2\kappa
-2\varpi }}\text{, with }\cos \varphi =\mathbf{n}\cdot \dfrac{\boldsymbol{%
\xi }}{\left\vert \boldsymbol{\xi }\right\vert }.  \label{b2}
\end{eqnarray}

By the change of variables $I_{\ast }\rightarrow \Phi =\left\vert \mathbf{g}%
\right\vert -2\left\vert \boldsymbol{\xi }\right\vert \cos \varphi +2\chi
_{\alpha }$, noting that%
\begin{equation*}
dI_{\ast }=\frac{m_{\alpha }}{2}\left\vert \mathbf{g}\right\vert \,d\Phi 
\text{,}
\end{equation*}%
for any positive number $a>0$, the bound 
\begin{eqnarray}
&&\int_{\mathbb{R}_{+}^{3}}\exp \left( -\frac{m_{\beta }}{a}\left(
\left\vert \mathbf{g}\right\vert -2\left\vert \boldsymbol{\xi }\right\vert
\cos \varphi +2\chi _{\alpha }\right) ^{2}\right) \frac{e^{-\left( I^{\prime
}+I_{\ast }^{\prime }\right) /2}}{\left( I^{\prime }I_{\ast }^{\prime
}\right) ^{1-\delta ^{\left( \beta \right) }/4}}dI_{\ast }dI^{\prime
}dI_{\ast }^{\prime }  \notag \\
&\leq &C\left\vert \mathbf{g}\right\vert \int_{\mathbb{-\infty }}^{\infty
}e^{-m_{\beta }\Phi ^{2}/a}\,d\Phi \left( \int_{0}^{\infty }\frac{%
e^{-I^{\prime }/2}}{\left( I^{\prime }\right) ^{1-\delta ^{\left( \beta
\right) }/4}}dI^{\prime }\right) ^{2}=C\left\vert \mathbf{g}\right\vert
\label{b3}
\end{eqnarray}%
may be obtained. Note also that 
\begin{equation}
\int_{\mathbb{R}_{+}}\left( \frac{1}{\left\vert \mathbf{g}\right\vert ^{4}}%
\mathbf{1}_{I\leq 1/2}+\frac{1}{I^{2}}\mathbf{1}_{I\geq 1/2}\right) dI\leq
C\left( 1+\frac{1}{\left\vert \mathbf{g}\right\vert ^{4}}\right) \text{.}
\label{b3a}
\end{equation}%
Then, by the bounds $\left( \ref{b2}\right) $ for $\kappa =1-\varpi =\left\{ 
\begin{array}{l}
0\text{ if }I\leq 1/2 \\ 
1\text{ if }I>1/2%
\end{array}%
\right. $, $\left( \ref{b3}\right) $, and $\left( \ref{b3a}\right) $, $%
k_{\alpha \beta }^{\left( \alpha \right) }(\boldsymbol{\xi },\boldsymbol{\xi 
}_{\ast },I,I_{\ast })\mathbf{1}_{\mathfrak{h}_{N}}\in L^{2}\left( d%
\boldsymbol{\xi \,}\,d\boldsymbol{\xi }_{\ast }dI\,dI_{\ast }\right) $, since%
\begin{eqnarray*}
&&\int_{\mathfrak{h}_{N}}\left( k_{\alpha \beta }^{\left( \alpha \right) }(%
\boldsymbol{\xi },\boldsymbol{\xi }_{\ast },I,I_{\ast })\right) ^{2}\,d%
\boldsymbol{\xi }d\boldsymbol{\xi }_{\ast }dI\mathbf{\,}dI_{\ast } \\
&\leq &C\int_{\left\vert \mathbf{g}\right\vert \geq \frac{1}{N}}\left( \frac{%
1}{\left\vert \mathbf{g}\right\vert }+\frac{1}{\left\vert \mathbf{g}%
\right\vert ^{5}}\right) e^{-m_{\alpha }^{2}\left\vert \mathbf{g}\right\vert
^{2}/\left( 4m_{\beta }\right) }\mathbf{\,}d\mathbf{g}\int_{\left\vert 
\boldsymbol{\xi }\right\vert \leq N}\mathbf{\,}d\boldsymbol{\xi \,} \\
&=&C\int_{1/N}^{\infty }\left( R+\frac{1}{R^{3}}\right) e^{-m_{\alpha
}^{2}R^{2}/\left( 4m_{\beta }\right) }dR\int_{0}^{N}r^{2}\mathbf{\,}dr%
\boldsymbol{\,} \\
&\leq &CN^{3}\int_{0}^{\infty }e^{-m_{\alpha }^{2}R^{2}/\left( 4m_{\beta
}\right) }(1+R^{4})N^{3}dR=CN^{6}\text{.}
\end{eqnarray*}

Note that, by the bound $\left( \ref{b2a}\right) $ for $\kappa =0$ and $%
\varpi \equiv 0$ if $I_{\ast }>1/2$,%
\begin{eqnarray}
&&k_{\alpha \beta }^{\left( \alpha \right) }(\boldsymbol{\xi },\boldsymbol{\xi 
}_{\ast },I,I_{\ast })\notag \\
&\leq &\frac{C}{\left\vert \mathbf{g}\right\vert }%
\int_{\mathbb{R}_{+}^{2}}\exp \left( -\frac{m_{\beta }}{8}\left( \left\vert 
\mathbf{g}\right\vert -2\left\vert \boldsymbol{\xi }\right\vert \cos \varphi
+2\chi _{\alpha }\right) ^{2}-\frac{m_{\alpha }^{2}}{8m_{\beta }}\left\vert 
\mathbf{g}\right\vert ^{2}\right)  \notag \\
&&\times \frac{e^{-\left( I^{\prime }+I_{\ast }^{\prime }\right) /2}}{\left(
I^{\prime }I_{\ast }^{\prime }\right) ^{1-\delta ^{\left( \beta \right) }/4}}%
dI^{\prime }dI_{\ast }^{\prime }\Psi _{\ast \varpi }^{\alpha }\text{ for }%
0\leq \varpi \leq 1\text{,}  \notag \\
\text{with }\Psi _{\ast \varpi }^{\alpha } &=&\left( \frac{1}{\left\vert 
\mathbf{g}\right\vert ^{2\varpi }I_{\ast }^{1-\varpi }}\mathbf{1}_{I_{\ast
}\leq 1/2}+\frac{1}{I_{\ast }}\mathbf{1}_{I_{\ast }\geq 1/2}\right)  \notag
\\
&\leq &\left( 1+\frac{1}{\left\vert \mathbf{g}\right\vert ^{2\varpi }}%
\right) \left( \frac{1}{I_{\ast }^{1-\varpi }}\mathbf{1}_{I_{\ast }\leq 1/2}+%
\frac{1}{I_{\ast }}\mathbf{1}_{I_{\ast }\geq 1/2}\right) .  \label{b7}
\end{eqnarray}

The integral of $k_{\alpha \beta }^{\left( \alpha \right) }(\boldsymbol{\xi }%
,\boldsymbol{\xi }_{\ast },I,I_{\ast })$ with respect to $\left( \boldsymbol{%
\xi },I\right) $ over $\mathbb{R}^{3}\times \mathbb{R}_{+}$ is bounded in $%
\left( \boldsymbol{\xi }_{\ast },I_{\ast }\right) $. Indeed, due to the
symmetry  $\left(\ref{sa1}\right)$ of the kernel $k_{\alpha \beta }^{\left( \alpha \right) }$, follows by the bounds $\left( \ref{b3}\right) $, and $%
\left( \ref{b7}\right) $ (for $\varpi =1/2$) that%
\begin{eqnarray*}
&&\int_{0}^{\infty }\int_{\mathbb{R}^{3}}k_{\alpha \beta }^{\left( \alpha
\right) }(\boldsymbol{\xi },\boldsymbol{\xi }_{\ast },I,I_{\ast })\,d%
\boldsymbol{\xi \,}dI=\int_{0}^{\infty }\int_{\mathbb{R}^{3}}k_{\alpha \beta
}^{\left( \alpha \right) }(\boldsymbol{\xi }_{\ast },\boldsymbol{\xi }%
,I_{\ast },I)\,d\mathbf{g}\boldsymbol{\,}dI\,\, \\
&\leq &\int_{0}^{\infty }\int_{\mathbb{R}^{3}}k_{\alpha \beta }^{\left(
\alpha \right) }(\boldsymbol{\xi }_{\ast },\boldsymbol{\xi },I_{\ast
},I)\left( \mathbf{1}_{I\leq 1/2}+2I\right) \,d\mathbf{g}\boldsymbol{\,}dI \\
&\leq &C\int_{\mathbb{R}^{3}}\left( 1+\frac{1}{\left\vert \mathbf{g}%
\right\vert ^{2}}\int_{0}^{1/2}\frac{1}{I^{1/2}}dI\left( \int_{0}^{\infty }%
\frac{e^{-I^{\prime }/2}}{\left( I^{\prime }\right) ^{1-\delta ^{\left(
\beta \right) }/4}}dI^{\prime }\right) ^{2}\right) \\
&&\times e^{-m_{\alpha }^{2}\left\vert \mathbf{g}\right\vert ^{2}/\left(
8m_{\beta }\right) }d\mathbf{g} \\
&\leq &C\int_{0}^{\infty }\left( 1+R^{2}\right) e^{-m_{\alpha
}^{2}R^{2}/\left( 8m_{\beta }\right) }dR=C\text{.}
\end{eqnarray*}

Aiming to prove that%
\begin{equation}
\int_{\mathbb{R}^{3}}\int_{0}^{\infty }k_{\alpha \beta }^{\left( \alpha
\right) }(\boldsymbol{\xi },\boldsymbol{\xi }_{\ast },I,I_{\ast })\,dI_{\ast
}d\boldsymbol{\xi }_{\ast }\leq C\frac{1+\log \left( 1+\left\vert 
\boldsymbol{\xi }\right\vert \right) }{\left\vert \boldsymbol{\xi }%
\right\vert }\text{,}  \label{b11}
\end{equation}%
for $\left\vert \boldsymbol{\xi }\right\vert \neq 0$, split the domain $%
\mathbb{R}^{3}\times \mathbb{R}_{+}$ into two subdomains%
\begin{equation*}
\mathcal{D}_{+}=\left\{ \mathbb{R}^{3}\times \mathbb{R}_{+}\text{; }I_{\ast
}\geq \left\vert \boldsymbol{\xi }\right\vert \right\} \text{ and }\mathcal{D%
}_{-}=\left\{ \mathbb{R}^{3}\times \mathbb{R}_{+}\text{; }I_{\ast }\leq
\left\vert \boldsymbol{\xi }\right\vert \right\} .
\end{equation*}%
By the bounds $\left( \ref{b3}\right) $ and $\left( \ref{b7}\right) $ (for $%
\varpi =0$), it can be shown that the restriction of bound $\left( \ref{b11}%
\right) $ to the domain $\mathcal{D}_{+}$ is satisfied for $\left\vert 
\boldsymbol{\xi }\right\vert \neq 0$ 
\begin{eqnarray*}
&&\int_{\mathbb{R}^{3}}\int_{\left\vert \boldsymbol{\xi }\right\vert
}^{\infty }k_{\alpha \beta }^{\left( \alpha \right) }(\boldsymbol{\xi },%
\boldsymbol{\xi }_{\ast },I,I_{\ast })\,dI_{\ast }d\boldsymbol{\xi }_{\ast
}\\
&\leq & \frac{C}{\left\vert \boldsymbol{\xi }\right\vert }\int_{\mathbb{R}%
^{3}}\int_{0}^{\infty }k_{\alpha \beta }^{\left( \alpha \right) }(%
\boldsymbol{\xi },\boldsymbol{\xi }_{\ast },I,I_{\ast })I_{\ast }\,dI_{\ast
}d\boldsymbol{\xi }_{\ast } \\
&\leq &\frac{C}{\left\vert \boldsymbol{\xi }\right\vert }\int_{\mathbb{R}%
^{3}}e^{-m_{\alpha }^{2}\left\vert \mathbf{g}\right\vert ^{2}/\left(
8m_{\beta }\right) }d\mathbf{g}\\
&\leq & \frac{C}{\left\vert \boldsymbol{\xi }%
\right\vert }\int_{0}^{\infty }R^{2}e^{-m_{\alpha }^{2}R^{2}/\left(
8m_{\beta }\right) }dR=\frac{C}{\left\vert \boldsymbol{\xi }\right\vert } \\
&\leq &C\frac{1+\log \left( 1+\left\vert \boldsymbol{\xi }\right\vert
\right) }{\left\vert \boldsymbol{\xi }\right\vert }\text{.}
\end{eqnarray*}%
On the other hand, by the bound $\left( \ref{b7}\right) $ (for $\varpi =1/2$
if $\alpha \in \left\{ s_{0}+1,...,s\right\} $), for $\left\vert \boldsymbol{%
\xi }\right\vert \neq 0$, it can be shown that the restriction of bound $%
\left( \ref{b11}\right) $ to the domain $\mathcal{D}_{-}$ is satisfied for $%
\left\vert \boldsymbol{\xi }\right\vert \neq 0$ 
\begin{eqnarray*}
&&\int_{0}^{\left\vert \boldsymbol{\xi }\right\vert }\int_{\mathbb{R}%
^{3}}k_{\alpha \beta }^{\left( \alpha \right) }(\boldsymbol{\xi },%
\boldsymbol{\xi }_{\ast },I,I_{\ast })\,d\boldsymbol{\xi }_{\ast }dI_{\ast }
\\
&\leq &C\int_{\mathbb{R}_{+}^{2}}\int_{0}^{\left\vert \boldsymbol{\xi }%
\right\vert }\int_{\mathbb{R}^{3}}\exp \left( -\frac{m_{\beta }}{8}\left(
\left\vert \mathbf{g}\right\vert -2\left\vert \boldsymbol{\xi }\right\vert
\cos \varphi +\chi _{\alpha }\right) ^{2}-\frac{m_{\alpha }^{2}}{8m_{\beta }}%
\left\vert \mathbf{g}\right\vert ^{2}\right)  \\
&&\times \frac{1+\left\vert \mathbf{g}\right\vert }{\left\vert \mathbf{g}%
\right\vert ^{2}}d\mathbf{g}\left( \frac{1}{I_{\ast }^{1/2}}\mathbf{1}%
_{I_{\ast }\leq 1/2}+\frac{1}{I_{\ast }}\mathbf{1}_{I_{\ast }\geq
1/2}\right) \mathbf{\,}\frac{e^{-\left( I^{\prime }+I_{\ast }^{\prime
}\right) /2}}{\left( I^{\prime }I_{\ast }^{\prime }\right) ^{1-\delta
^{\left( \beta \right) }/4}}dI_{\ast }dI^{\prime }dI_{\ast }^{\prime } \\
&\leq &\frac{C}{\left\vert \boldsymbol{\xi }\right\vert }\left(
\int_{0}^{1/2}\frac{1}{I_{\ast }^{1/2}}dI_{\ast }+\int_{1/2}^{\left\vert 
\boldsymbol{\xi }\right\vert }\frac{1}{I_{\ast }}\mathbf{1}_{\left\vert 
\boldsymbol{\xi }\right\vert \geq 1/2}dI_{\ast }\right) \left(
\int_{0}^{\infty }\frac{e^{-I^{\prime }/2}}{\left( I^{\prime }\right)
^{1-\delta ^{\left( \beta \right) }/4}}dI^{\prime }\right) ^{2} \\
&\leq &\frac{C}{\left\vert \boldsymbol{\xi }\right\vert }\left( 1+\log
\left( \left\vert \boldsymbol{\xi }\right\vert \right) \mathbf{1}%
_{\left\vert \boldsymbol{\xi }\right\vert \geq 1/2}\right) \leq C\frac{%
1+\log \left( 1+\left\vert \boldsymbol{\xi }\right\vert \right) }{\left\vert 
\boldsymbol{\xi }\right\vert }\text{.}
\end{eqnarray*}%
Here the second inequality follows for $\left\vert \boldsymbol{\xi }\right\vert \neq 0$, by, with $\chi _{R}^{\alpha }=\Delta
I_{\ast }/\left( m_{\alpha }R\right) $, 
\begin{eqnarray*}
&&C\int_{\mathbb{R}^{3}}\frac{1+\left\vert \mathbf{g}\right\vert }{%
\left\vert \mathbf{g}\right\vert ^{2}}\exp \left( -\frac{m_{\beta }}{8}%
\left( \left\vert \mathbf{g}\right\vert -2\left\vert \boldsymbol{\xi }%
\right\vert \cos \varphi +2\chi _{\alpha }\right) ^{2}-\frac{m_{\alpha }^{2}%
}{8m_{\beta }}\left\vert \mathbf{g}\right\vert ^{2}\right) \,d\mathbf{g} \\
&=&C\int\limits_{0}^{\infty }\int\limits_{0}^{\pi }(1+R)\exp \left( -\frac{%
m_{\beta }}{8}\left( R-2\left\vert \boldsymbol{\xi }\right\vert \cos \varphi
+2\chi _{R}^{\alpha }\right) ^{2}-\frac{m_{\alpha }^{2}}{8m_{\beta }}%
R^{2}\right) \\
&\times& \sin \varphi \,d\varphi \mathbf{\,}dR \\
&=&\frac{C}{\left\vert \boldsymbol{\xi }\right\vert }\int_{0}^{\infty
}\int_{R+2\chi _{R}^{\alpha }-2\left\vert \boldsymbol{\xi }\right\vert
}^{R+2\chi _{R}^{\alpha }+2\left\vert \boldsymbol{\xi }\right\vert
}e^{-m_{\beta }\Phi ^{2}/8}(1+R)e^{-m_{\alpha }^{2}R^{2}/\left( 8m_{\beta
}\right) }d\Phi dR \\
&\leq &\frac{C}{\left\vert \boldsymbol{\xi }\right\vert }\int_{0}^{\infty
}e^{-m_{\alpha }^{2}R^{2}/\left( 8m_{\beta }\right) }(1+R)\,dR\int_{-\infty
}^{\infty }e^{-m_{\beta }\Phi ^{2}/8}d\Phi =\frac{C}{\left\vert \boldsymbol{%
\xi }\right\vert }\text{,}
\end{eqnarray*}%
obtained by a change to spherical coordinates followed by the change of
variables $\varphi \rightarrow \Phi =R-2\left\vert \boldsymbol{\xi }%
\right\vert \cos \varphi +2\chi _{R}^{\alpha }$, with $d\Phi =\sin \varphi
\,d\varphi $.

Note that $f(x)=\dfrac{1+\log (1+x)}{x}$ is a decreasing function for $x>0$.
Therefore, by the bounds $\left( \ref{b3}\right) $, $\left( \ref{b7}\right) $%
, and $\left( \ref{b11}\right) $, 
\begin{eqnarray*}
&&\sup_{\left( \boldsymbol{\xi },I\right) \in \mathbb{R}^{3}\times \mathbb{R}%
_{+}}\int_{\mathbb{R}^{3}\times \mathbb{R}_{+}}k_{\alpha \beta }^{\left(
\alpha \right) }(\boldsymbol{\xi },\boldsymbol{\xi }_{\ast },I,I_{\ast
})-k_{\alpha \beta }^{\left( \alpha \right) }(\boldsymbol{\xi },\boldsymbol{%
\xi }_{\ast },I,I_{\ast })\mathbf{1}_{\mathfrak{h}_{N}}\,d\boldsymbol{\xi }%
_{\ast }dI_{\ast } \\
&\leq &\sup_{\left( \boldsymbol{\xi },I\right) \in \mathbb{R}^{3}\times 
\mathbb{R}_{+}}\int_{0}^{\infty }\int_{\left\vert \mathbf{g}\right\vert \leq 
\frac{1}{N}}k_{\alpha \beta }^{\left( \alpha \right) }(\boldsymbol{\xi },%
\boldsymbol{\xi }_{\ast },I,I_{\ast })\,d\mathbf{g\,}dI_{\ast } \\
&&+\sup_{\left\vert \boldsymbol{\xi }\right\vert \geq N}\int_{0}^{\infty
}\int_{\mathbb{R}^{3}}k_{\alpha \beta }^{\left( \alpha \right) }(\boldsymbol{%
\xi },\boldsymbol{\xi }_{\ast },I,I_{\ast })\,d\boldsymbol{\xi }_{\ast
}dI_{\ast } \\
&\leq &\sup_{\left( \boldsymbol{\xi },I\right) \in \mathbb{R}^{3}\times 
\mathbb{R}_{+}}\int_{0}^{\infty }\int_{\left\vert \mathbf{g}\right\vert \leq 
\frac{1}{N}}k_{\alpha \beta }^{\left( \alpha \right) }(\boldsymbol{\xi },%
\boldsymbol{\xi }_{\ast },I,I_{\ast })\left( \mathbf{1}_{I_{\ast }\leq
1/2}+2I_{\ast }\right) \,d\mathbf{g}\,dI_{\ast } \\
&&+C\frac{1+\log \left( 1+N\right) }{N} \\
&\leq &\int_{\left\vert \mathbf{g}\right\vert \leq \frac{1}{N}}\frac{C}{%
\left\vert \mathbf{g}\right\vert }\left( \frac{1}{\left\vert \mathbf{g}%
\right\vert }\int_{0}^{1/2}\frac{1}{I_{\ast }^{1/2}}dI_{\ast }+\left\vert 
\mathbf{g}\right\vert \right) \,d\mathbf{g}+C\frac{1+\log \left( 1+N\right) 
}{N} \\
&\leq &C\left( \int_{\left\vert \mathbf{g}\right\vert \leq \frac{1}{N}}1+%
\frac{1}{\left\vert \mathbf{g}\right\vert ^{2}}\,d\mathbf{g}+\frac{1+\log
\left( 1+N\right) }{N}\right) \\
&\leq &C\left( \int_{0}^{\frac{1}{N}}1+R^{2}\,dR+\frac{1+\log \left(
1+N\right) }{N}\right) \\
&=&C\left( \frac{1}{N}+\frac{1}{N^{3}}+\frac{1+\log \left( 1+N\right) }{N}%
\right) \rightarrow 0\text{ as }N\rightarrow \infty \text{.}
\end{eqnarray*}

Hence, by Lemma \ref{LGD}, the operators 
\begin{equation*}
K_{\alpha \beta 3}=\int_{\mathbb{R}^{3}\times \mathbb{R}_{+}}k_{\alpha \beta
}^{(\alpha )}(\boldsymbol{\xi },\boldsymbol{\xi }_{\ast },I,I_{\ast
})\,h_{\ast }\,d\boldsymbol{\xi }_{\ast }dI_{\ast }
\end{equation*}%
are compact on $L^{2}\left( d\mathbf{Z}_{\alpha }\right) $ for $\left\{
\alpha ,\beta \right\} \subseteq \left\{ 1,...,s\right\} $.


\begin{figure}[h]
\centering
\includegraphics[width=0.5\textwidth]{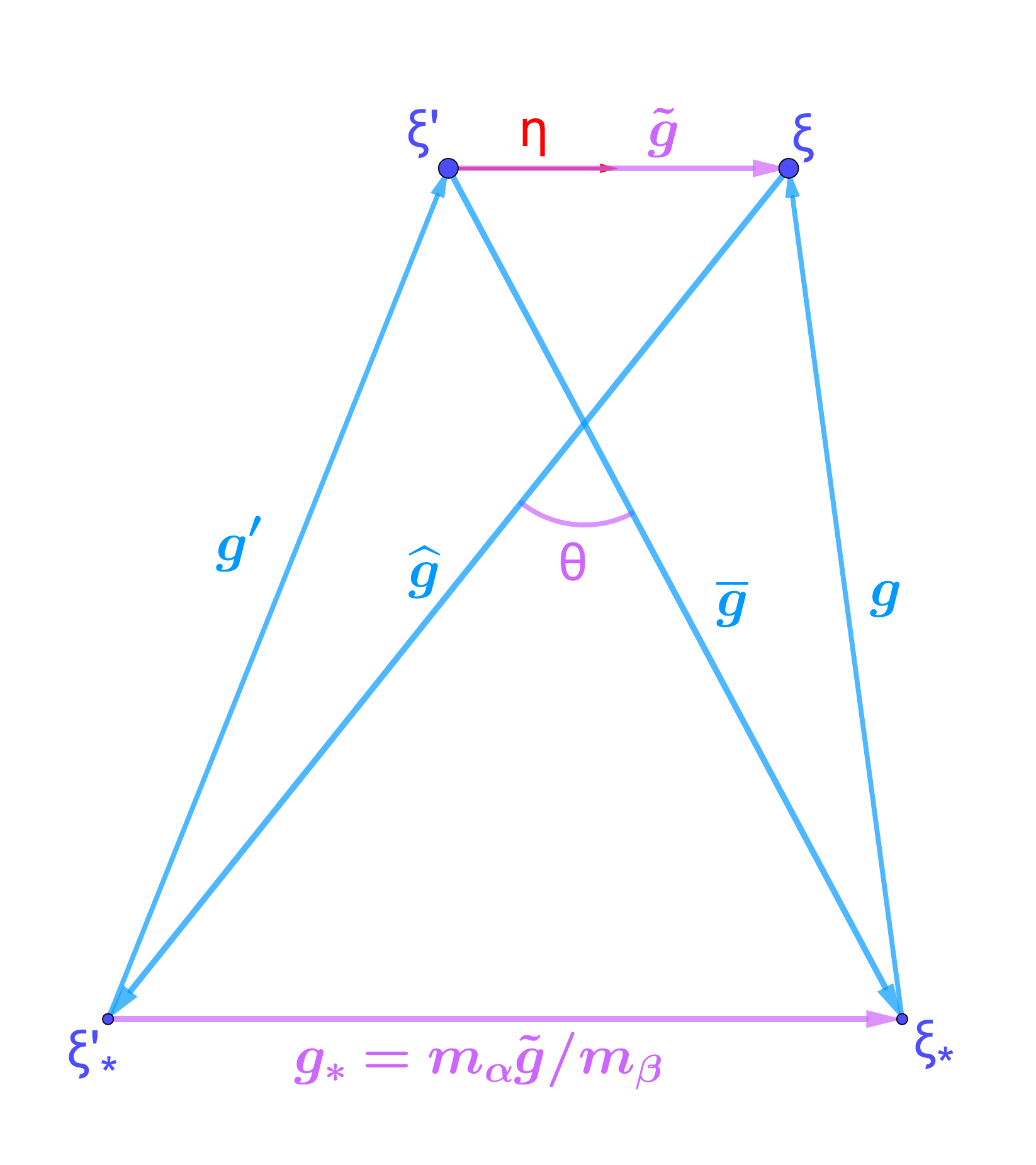}
\caption{Typical collision of $K_{\protect\alpha \protect\beta 2}$.}
\label{fig3}
\end{figure}

\textbf{III. Compactness of }$K_{\alpha \beta 2}=\int_{\mathbb{R}^{3}\times 
\mathbb{R}_{+}}k_{\alpha \beta 2}^{\left( \beta \right) }(\boldsymbol{\xi },%
\boldsymbol{\xi }_{\ast },I,I_{\ast })\,h_{\beta \ast }\,d\boldsymbol{\xi }%
_{\ast }dI_{\ast }$ for $\left\{ \alpha ,\beta \right\} \subseteq \left\{
1,...,s\right\} $.

Assume that $m_{\alpha }\neq m_{\beta }$, denote%
\begin{eqnarray*}
&&\mathbf{g}=\boldsymbol{\xi }-\boldsymbol{\xi }_{\ast }\text{, }\mathbf{g}%
^{\prime }=\boldsymbol{\xi }^{\prime }-\boldsymbol{\xi }_{\ast }^{\prime }%
\text{, }\widehat{\mathbf{g}}=\boldsymbol{\xi }_{\ast }^{\prime }-%
\boldsymbol{\xi }\text{, }\overline{\mathbf{g}}=\boldsymbol{\xi }_{\ast }-%
\boldsymbol{\xi }^{\prime }\text{, }\mathbf{g}_{\alpha \beta }=\dfrac{%
m_{\alpha }\boldsymbol{\xi }-m_{\beta }\boldsymbol{\xi }_{\ast }}{m_{\alpha
}-m_{\beta }}\text{, } \\
&&\mathbf{g}_{\alpha \beta }^{\prime }=\dfrac{m_{\alpha }\boldsymbol{\xi }%
^{\prime }-m_{\beta }\boldsymbol{\xi }_{\ast }^{\prime }}{m_{\alpha
}-m_{\beta }}\text{, }\widehat{\sigma }_{\alpha \beta }=\sigma _{\alpha
\beta }\left( \left\vert \widehat{\mathbf{g}}\right\vert ,\frac{\widehat{%
\mathbf{g}}\cdot \overline{\mathbf{g}}}{\left\vert \widehat{\mathbf{g}}%
\right\vert \left\vert \overline{\mathbf{g}}\right\vert },I,I_{\ast
}^{\prime },I^{\prime },I_{\ast }\right) \text{,} \\
&&\widehat{\Delta }_{\alpha \beta }I_{\#}=\frac{m_{\beta }-m_{\alpha }}{%
m_{\alpha }m_{\beta }}\Delta I_{\#}\text{, and }\Delta I_{\#}=I_{\ast
}+I^{\prime }-I-I_{\ast }^{\prime }\text{.}
\end{eqnarray*}%
Note that, cf. Figure \ref{fig3},%
\begin{eqnarray}
&&W_{\alpha \beta }\left( \boldsymbol{\xi },\boldsymbol{\xi }_{\ast
}^{\prime },I,I_{\ast }^{\prime }\left\vert \boldsymbol{\xi }^{\prime },%
\boldsymbol{\xi }_{\ast },I^{\prime },I_{\ast }\right. \right)  \notag \\
&=&\left( m_{\alpha }+m_{\beta }\right) ^{2}m_{\alpha }m_{\beta }I^{\delta
^{\left( \alpha \right) }/2-1}\left( I^{\prime }\right) ^{\delta ^{\left(
\beta \right) }/2-1}\widehat{\sigma }_{\alpha \beta }\delta _{3}\left(
\left( m_{\alpha }-m_{\beta }\right) \left( \mathbf{g}_{\alpha \beta }-%
\mathbf{g}_{\alpha \beta }^{\prime }\right) \right)  \notag \\
&&\times \frac{\left\vert \widehat{\mathbf{g}}\right\vert }{\left\vert 
\overline{\mathbf{g}}\right\vert }\delta _{1}\left( \frac{m_{\alpha
}m_{\beta }}{2\left( m_{\alpha }-m_{\beta }\right) }\left( \left\vert 
\mathbf{g}^{\prime }\right\vert ^{2}-\left\vert \mathbf{g}\right\vert
^{2}\right) -\Delta I_{\#}\right)  \notag \\
&=&\frac{\left( m_{\alpha }+m_{\beta }\right) ^{2}}{\left( m_{\alpha
}-m_{\beta }\right) ^{2}}I^{\delta ^{\left( \alpha \right) }/2-1}\left(
I^{\prime }\right) ^{\delta ^{\left( \beta \right) }/2-1}\widehat{\sigma }%
_{\alpha \beta }\mathbf{1}_{\left\vert \mathbf{g}\right\vert ^{2}>2\widehat{%
\Delta }_{\alpha \beta }I_{\#}}\frac{\left\vert \widehat{\mathbf{g}}%
\right\vert }{\left\vert \mathbf{g}^{\prime }\right\vert \left\vert 
\overline{\mathbf{g}}\right\vert }\delta _{3}\left( \mathbf{g}_{\alpha \beta
}-\mathbf{g}_{\alpha \beta }^{\prime }\right)  \notag \\
&&\times \delta _{1}\left( \left\vert \mathbf{g}^{\prime }\right\vert -\sqrt{%
\left\vert \mathbf{g}\right\vert ^{2}-2\widehat{\Delta }_{\alpha \beta
}I_{\#}}\right) \text{,}  \label{exp0}
\end{eqnarray}%
Then, by a change of variables $\left\{ \boldsymbol{\xi }^{\prime },%
\boldsymbol{\xi }_{\ast }^{\prime }\right\} \rightarrow \left\{ \!\left\vert 
\mathbf{g}^{\prime }\right\vert ,\boldsymbol{\omega }=\dfrac{\mathbf{g}%
^{\prime }}{\left\vert \mathbf{g}^{\prime }\right\vert },\mathbf{g}_{\alpha
\beta }^{\prime }=\dfrac{m_{\alpha }\boldsymbol{\xi }^{\prime }-m_{\beta }%
\boldsymbol{\xi }_{\ast }^{\prime }}{m_{\alpha }-m_{\beta }}\!\right\} $,
with $\mathbf{g}^{\prime }=\boldsymbol{\xi }^{\prime }-\boldsymbol{\xi }%
_{\ast }^{\prime }$, where%
\begin{equation*}
d\boldsymbol{\xi }^{\prime }d\boldsymbol{\xi }_{\ast }^{\prime }=d\mathbf{g}%
^{\prime }d\mathbf{g}_{\alpha \beta }^{\prime }=\left\vert \mathbf{g}%
^{\prime }\right\vert ^{2}d\left\vert \mathbf{g}^{\prime }\right\vert d%
\mathbf{g}_{\alpha \beta }^{\prime }d\boldsymbol{\omega }\text{,}
\end{equation*}%
the expression $\left( \ref{k1}\right) $ of $k_{\alpha \beta 2}^{\left(
\beta \right) }$ may be transformed to%
\begin{eqnarray}
&&k_{\alpha \beta 2}^{\left( \beta \right) }(\boldsymbol{\xi },\boldsymbol{%
\xi }_{\ast },I,I_{\ast })  \notag \\
&=&\int_{\left( \mathbb{R}^{3}\times \mathbb{R}_{+}\right) ^{2}}\frac{%
w_{\alpha \beta }(\boldsymbol{\xi },\boldsymbol{\xi }_{\ast }^{\prime
},I,I_{\ast }^{\prime }\left\vert \boldsymbol{\xi }^{\prime },\boldsymbol{%
\xi }_{\ast },I^{\prime },I_{\ast }\right. )}{\left( M_{\alpha }M_{\beta
\ast }\right) ^{1/2}}\,d\boldsymbol{\xi }^{\prime }d\boldsymbol{\xi }_{\ast
}^{\prime }dI^{\prime }dI_{\ast }^{\prime }  \notag \\
&=&\int_{\mathbb{R}^{3}\times \mathbb{S}^{2}\times \mathbb{R}%
_{+}^{3}}\left\vert \mathbf{g}^{\prime }\right\vert \left( M_{\alpha
}^{\prime }M_{\beta \ast }^{\prime }\right) ^{1/2}W_{\alpha \beta }(%
\boldsymbol{\xi },\boldsymbol{\xi }_{\ast }^{\prime },I,I_{\ast }^{\prime
}\left\vert \boldsymbol{\xi }^{\prime },\boldsymbol{\xi }_{\ast },I^{\prime
},I_{\ast }\right. )\left\vert \mathbf{g}^{\prime }\right\vert  \notag \\
&&\times \mathbf{1}_{\left\vert \mathbf{g}\right\vert ^{2}>2\widehat{\Delta }%
_{\alpha \beta }I_{\#}}\,dI^{\prime }dI_{\ast }^{\prime }d\left\vert \mathbf{%
g}^{\prime }\right\vert d\mathbf{g}_{\alpha \beta }^{\prime }d\boldsymbol{%
\omega }  \notag \\
&=&\frac{\left( m_{\alpha }+m_{\beta }\right) ^{2}}{\left( m_{\alpha
}-m_{\beta }\right) ^{2}}\int_{\mathbb{S}^{2}\times \mathbb{R}%
_{+}^{2}}\Sigma _{\alpha \beta }\frac{\left\vert \widehat{\mathbf{g}}%
\right\vert \left\vert \mathbf{g}^{\prime }\right\vert }{\left\vert 
\overline{\mathbf{g}}\right\vert }\frac{I^{\delta ^{\left( \alpha \right)
}/4-1/2}\left( I_{\ast }^{\prime }\right) ^{\delta ^{\left( \beta \right)
}/4-1/2}}{\left( I^{\prime }\right) ^{\delta ^{\left( \alpha \right)
}/4-1/2}I_{\ast }^{\delta ^{\left( \beta \right) }/4-1/2}}\widehat{\sigma }%
_{\alpha \beta }  \notag \\
&&\times \mathbf{1}_{\left\vert \mathbf{g}\right\vert ^{2}>2\widehat{\Delta }%
_{\alpha \beta }I_{\#}}\,dI^{\prime }dI_{\ast }^{\prime }d\boldsymbol{\omega 
}\text{.}  \label{exp1}
\end{eqnarray}%
However, also%
\begin{eqnarray}
&&k_{\alpha \beta 2}^{\left( \beta \right) }(\boldsymbol{\xi },\boldsymbol{%
\xi }_{\ast },I,I_{\ast })  \notag \\
&=&\int_{\mathbb{S}^{2}\times \mathbb{R}_{+}^{2}}\left\vert \widehat{\mathbf{%
g}}\right\vert \left( M_{\alpha }^{\prime }M_{\beta \ast }^{\prime }\right)
^{1/2}\frac{I^{\delta ^{\left( \alpha \right) }/4-1/2}\left( I_{\ast
}^{\prime }\right) ^{\delta ^{\left( \beta \right) }/4-1/2}}{\left(
I^{\prime }\right) ^{\delta ^{\left( \alpha \right) }/4-1/2}I_{\ast
}^{\delta ^{\left( \beta \right) }/4-1/2}}\widehat{\sigma }_{\alpha \beta }%
\notag \\ 
&\times& \mathbf{1}_{\mu_{\alpha \beta }\left\vert \widehat{\mathbf{g}}\right\vert ^{2}>2%
\Delta I_{\#}}d\widehat{\boldsymbol{\omega }}dI^{\prime
}dI_{\ast }^{\prime }\text{, with }\widehat{\boldsymbol{\omega }}=\frac{\overline{\mathbf{g}}}{%
\left\vert \overline{\mathbf{g}}\right\vert }\text{,}  \label{exp2}
\end{eqnarray}%
and, furthermore,%
\begin{eqnarray}
&&k_{\alpha \beta 2}^{\left( \beta \right) }(\boldsymbol{\xi },\boldsymbol{%
\xi }_{\ast },I,I_{\ast })  \notag \\
&=&\int_{\mathbb{S}^{2}\times \mathbb{R}_{+}^{2}}\left\vert \overline{%
\mathbf{g}}\right\vert \left( M_{\alpha }^{\prime }M_{\beta \ast }^{\prime
}\right) ^{1/2}\frac{\left( I^{\prime }\right) ^{\delta ^{\left( \alpha
\right) }/4-1/2}I_{\ast }^{\delta ^{\left( \beta \right) }/4-1/2}}{I^{\delta
^{\left( \alpha \right) }/4-1/2}\left( I_{\ast }^{\prime }\right) ^{\delta
^{\left( \beta \right) }/4-1/2}}\widehat{\sigma }_{\alpha \beta }^{\prime }%
\notag \\ 
&\times& \mathbf{1}_{\mu_{\alpha \beta }\left\vert \overline{\mathbf{g}}\right\vert ^{2}>-2%
\Delta I_{\#}}d\overline{\boldsymbol{\omega }}dI^{\prime
}dI_{\ast }^{\prime }\text{, with}  \notag \\
&&\widehat{\sigma }_{\alpha \beta }^{\prime }=\sigma _{\alpha
\beta }\left( \left\vert \overline{\mathbf{g}}\right\vert ,\frac{\widehat{%
\mathbf{g}}\cdot \overline{\mathbf{g}}}{\left\vert \widehat{\mathbf{g}}%
\right\vert \left\vert \overline{\mathbf{g}}\right\vert },I^{\prime
},I_{\ast },I,I_{\ast }^{\prime }\right) \text{ and }\overline{\boldsymbol{%
\omega }}=\frac{\widehat{\mathbf{g}}}{\left\vert \widehat{\mathbf{g}}%
\right\vert }\text{.}  \label{exp3}
\end{eqnarray}

By straightforward calculations, with $\boldsymbol{\omega }=\mathbf{g}%
^{\prime }/\left\vert \mathbf{g}^{\prime }\right\vert $, 
\begin{equation*}
\left\{ 
\begin{array}{l}
\boldsymbol{\xi }^{\prime }=\mathbf{g}_{\alpha \beta }^{\prime }-\dfrac{%
m_{\beta }}{m_{\alpha }-m_{\beta }}\mathbf{g}^{\prime }=\mathbf{g}_{\alpha
\beta }-\dfrac{m_{\beta }\boldsymbol{\omega }}{m_{\alpha }-m_{\beta }}\sqrt{%
\left\vert \mathbf{g}\right\vert ^{2}-2\widehat{\Delta }_{\alpha \beta
}I_{\#}} \\ 
\boldsymbol{\xi }_{\ast }^{\prime }=\mathbf{g}_{\alpha \beta }^{\prime }-%
\dfrac{m_{\alpha }}{m_{\alpha }-m_{\beta }}\mathbf{g}^{\prime }=\mathbf{g}%
_{\alpha \beta }-\dfrac{m_{\alpha }\boldsymbol{\omega }}{m_{\alpha
}-m_{\beta }}\sqrt{\left\vert \mathbf{g}\right\vert ^{2}-2\widehat{\Delta }%
_{\alpha \beta }I_{\#}}%
\end{array}%
\right. \text{.}
\end{equation*}%
It follows, again by straightforward calculations, that%
\begin{eqnarray}
&&m_{\alpha }\left\vert \boldsymbol{\xi }^{\prime }\right\vert ^{2}+m_{\beta
}\left\vert \boldsymbol{\xi }_{\ast }^{\prime }\right\vert ^{2}  \notag \\
&=&\left( m_{\alpha }+m_{\beta }\right) \left( \left\vert \mathbf{g}_{\alpha
\beta }^{\prime }\right\vert ^{2}+\dfrac{m_{\alpha }m_{\beta }}{\left(
m_{\alpha }-m_{\beta }\right) ^{2}}\left\vert \mathbf{g}^{\prime
}\right\vert ^{2}\right) -\dfrac{4m_{\alpha }m_{\beta }}{m_{\alpha
}-m_{\beta }}\mathbf{g}_{\alpha \beta }^{\prime }\cdot \mathbf{g}^{\prime } 
\notag \\
&=&\left( \sqrt{m_{\alpha }}-\sqrt{m_{\beta }}\right) ^{2}\left( \left\vert 
\mathbf{g}_{\alpha \beta }^{\prime }\right\vert ^{2}+\dfrac{m_{\alpha
}m_{\beta }}{\left( m_{\alpha }-m_{\beta }\right) ^{2}}\left\vert \mathbf{g}%
^{\prime }\right\vert ^{2}\right)  \notag \\
&&+2\sqrt{m_{\alpha }m_{\beta }}\left( \mathbf{g}_{\alpha \beta }^{\prime }-%
\dfrac{\sqrt{m_{\alpha }m_{\beta }}}{m_{\alpha }-m_{\beta }}\mathbf{g}%
^{\prime }\right) ^{2}  \notag \\
&\geq &\left( \sqrt{m_{\alpha }}-\sqrt{m_{\beta }}\right) ^{2}\left\vert 
\mathbf{g}_{\alpha \beta }^{\prime }\right\vert ^{2}+\dfrac{m_{\alpha
}m_{\beta }}{\left( \sqrt{m_{\alpha }}+\sqrt{m_{\beta }}\right) ^{2}}%
\left\vert \mathbf{g}^{\prime }\right\vert ^{2}  \notag \\
&=&\left( \sqrt{m_{\alpha }}-\sqrt{m_{\beta }}\right) ^{2}\left\vert \mathbf{%
g}_{\alpha \beta }\right\vert ^{2}+\dfrac{m_{\alpha }m_{\beta }}{\left( 
\sqrt{m_{\alpha }}+\sqrt{m_{\beta }}\right) ^{2}}\left( \left\vert \mathbf{g}%
\right\vert ^{2}-2\widehat{\Delta }_{\alpha \beta }I_{\#}\right) \text{.}
\label{ineq1}
\end{eqnarray}

Denote $\widehat{\mathcal{E}}_{\alpha \beta }=\widehat{\mathcal{E}}_{\beta
\alpha }^{\ast }=1$ if $\alpha \in \left\{ 1,...,s_{0}\right\} $, while if $\alpha \in \left\{ s_{0}+1,...,s\right\} $ and $%
\beta \in \left\{ 1,...,s_{0}\right\} $, then $\widehat{%
\mathcal{E}}_{\beta \alpha }^{\ast }=\mu_{\alpha \beta }\left\vert \widehat{\mathbf{g}}\right\vert
^{2}/2+I_{\ast }^{\prime }=\mu_{\alpha \beta }\left\vert \overline{\mathbf{g}}\right\vert
^{2}/2+I_{\ast }$ and $\widehat{\mathcal{E}}_{\alpha \beta }=\mu_{\alpha \beta }\left\vert \widehat{%
\mathbf{g}}\right\vert ^{2}/2+I=\mu_{\alpha \beta }\left\vert \overline{\mathbf{g}}\right\vert
^{2}/2+I^{\prime }$, and if $\left\{ \alpha ,\beta \right\} \subset \left\{
s_{0}+1,...,s\right\}$, then $\widehat{\mathcal{E}}%
_{\alpha \beta }=\widehat{\mathcal{E}}_{\alpha \beta }^{\ast }=\mu_{\alpha \beta }\left\vert 
\widehat{\mathbf{g}}\right\vert ^{2}/2+I+I_{\ast }^{\prime }=\mu_{\alpha \beta }\left\vert \overline{%
\mathbf{g}}\right\vert ^{2}/2+I^{\prime }+I_{\ast }$. 

Applying the Cauchy-Schwarz inequality, we obtain
that%
\begin{eqnarray}
&&\left( \int_{\mathbb{R}_{+}^{2}}\left( M_{\alpha }^{\prime }M_{\beta \ast
}^{\prime }\right) ^{1/4}\frac{\left( II^{\prime }\right) ^{\delta ^{\left(
\alpha \right) }/4-1/2}\left( I_{\ast }I_{\ast }^{\prime }\right) ^{\delta
^{\left( \beta \right) }/4-1/2}}{\widehat{\mathcal{E}}_{\alpha \beta
}^{\delta ^{\left( \alpha \right) }/2}\left( \widehat{\mathcal{E}}_{\alpha
\beta }^{\ast }\right) ^{\delta ^{\left( \beta \right) }/2}}e^{-\left(
I^{\prime }+I_{\ast }^{\prime }\right) /4}dI^{\prime }dI_{\ast }^{\prime
}\right) ^{2}  \notag \\
&\leq &\int_{\mathbb{R}_{+}^{2}}\left( M_{\alpha }^{\prime }M_{\beta \ast
}^{\prime }\right) ^{1/2}\frac{\left( II^{\prime }\right) ^{\delta ^{\left(
\alpha \right) }/2-1}\left( I_{\ast }I_{\ast }^{\prime }\right) ^{\delta
^{\left( \beta \right) }/2-1}}{\widehat{\mathcal{E}}_{\alpha \beta }^{\delta
^{\left( \alpha \right) }}\left( \widehat{\mathcal{E}}_{\alpha \beta }^{\ast
}\right) ^{\delta ^{\left( \beta \right) }}}dI^{\prime }dI_{\ast }^{\prime
}\int_{\mathbb{R}_{+}^{2}}e^{-\left( I^{\prime }+I_{\ast }^{\prime }\right)
/2}dI^{\prime }dI_{\ast }^{\prime }  \notag \\
&\leq &4\int_{\mathbb{R}_{+}^{2}}\frac{\left( M_{\alpha }^{\prime }M_{\beta
\ast }^{\prime }\right) ^{1/2}}{\widehat{\mathcal{E}}_{\alpha \beta
}^{2}\left( \widehat{\mathcal{E}}_{\alpha \beta }^{\ast }\right) ^{2}}%
dI^{\prime }dI_{\ast }^{\prime }\text{,}  \label{ineq1a}
\end{eqnarray}%
since 
\begin{equation*}
\widehat{\mathcal{E}}_{\alpha \beta }^{\delta ^{\left( \alpha \right)
}-2}\left( \widehat{\mathcal{E}}_{\alpha \beta }^{\ast }\right) ^{\delta
^{\left( \beta \right) }-2}\geq \left( II^{\prime }\right) ^{\delta ^{\left(
\alpha \right) }/2-1}\left( I_{\ast }I_{\ast }^{\prime }\right) ^{\delta
^{\left( \beta \right) }/2-1}\text{.}
\end{equation*}

Note that $\max \left( \left\vert \widehat{\mathbf{g}}\right\vert
,\left\vert \overline{\mathbf{g}}\right\vert \right) \geq \left\vert \mathbf{%
g}\right\vert $, cf. Figure \ref{fig3}.

\textbf{H1.} If $\min \left( \left\vert \widehat{\mathbf{g}}\right\vert
,\left\vert \overline{\mathbf{g}}\right\vert \right) \geq 1$, then $%
\left\vert \widehat{\mathbf{g}}\right\vert \left\vert \overline{\mathbf{g}}%
\right\vert \geq \max \left( \left\vert \widehat{\mathbf{g}}\right\vert
,\left\vert \overline{\mathbf{g}}\right\vert \right) \geq \left\vert \mathbf{%
g}\right\vert $, and%
\begin{eqnarray*}
\frac{\left\vert \widehat{\mathbf{g}}\right\vert \left\vert \overline{%
\mathbf{g}}\right\vert +\left( \left\vert \widehat{\mathbf{g}}\right\vert
\left\vert \overline{\mathbf{g}}\right\vert \right) ^{\gamma /2}}{\left\vert 
\widehat{\mathbf{g}}\right\vert ^{2}} &=&\frac{\left\vert \overline{\mathbf{g%
}}\right\vert }{\left\vert \widehat{\mathbf{g}}\right\vert }\left( 1+\frac{1%
}{\left( \left\vert \widehat{\mathbf{g}}\right\vert \left\vert \overline{%
\mathbf{g}}\right\vert \right) ^{1-\gamma /2}}\right) \\
&\leq &\frac{\left\vert \overline{\mathbf{g}}\right\vert }{\left\vert 
\widehat{\mathbf{g}}\right\vert }\left( 1+\frac{1}{\left\vert \mathbf{g}%
\right\vert ^{1-\gamma /2}}\right) \leq \frac{\left\vert \overline{\mathbf{g}%
}\right\vert }{\left\vert \widehat{\mathbf{g}}\right\vert }\left( 1+\frac{1}{%
\left\vert \mathbf{g}\right\vert }\right) \text{.}
\end{eqnarray*}%
Note that $\left\vert \mathbf{g}^{\prime }\right\vert ^{2}\left( M_{\alpha
}^{\prime }M_{\beta \ast }^{\prime }\right) ^{1/4}\leq Ce^{-\left( I^{\prime
}+I_{\ast }^{\prime }\right) /4}$, since $x^{2}e^{-ax^{2}}\leq 1/a$ for any
positive number $a>0$. Hence, by expressions $\left( \ref{exp0}\right) $-$%
\left( \ref{exp1}\right) $, assumption $\left( \ref{est1}\right) $, and
inequalites $\left( \ref{ineq1}\right) $-$\left( \ref{ineq1a}\right) $, one
may obtain the following bound%
\begin{equation}
\left( k_{\alpha \beta 2}^{\left( \beta \right) }(\boldsymbol{\xi },%
\boldsymbol{\xi }_{\ast },I,I_{\ast })\right) ^{2}\leq C\left( 1+\frac{1}{%
\left\vert \mathbf{g}\right\vert ^{2}}\right) \int_{\mathbb{R}_{+}^{2}}\frac{%
\left( M_{\alpha }^{\prime }M_{\beta \ast }^{\prime }\right) ^{1/2}}{%
\widehat{\mathcal{E}}_{\alpha \beta }^{2}\left( \widehat{\mathcal{E}}%
_{\alpha \beta }^{\ast }\right) ^{2}}dI^{\prime }dI_{\ast }^{\prime }\text{.}
\label{b6}
\end{equation}

\textbf{H2.} If $\min \left( \left\vert \widehat{\mathbf{g}}\right\vert
,\left\vert \overline{\mathbf{g}}\right\vert \right) <1$, then either of the
two cases below apply:

(i) $\left\vert \widehat{\mathbf{g}}%
\right\vert =\max \left( \left\vert \widehat{\mathbf{g}}\right\vert
,\left\vert \overline{\mathbf{g}}\right\vert \right) \geq \left\vert \mathbf{%
g}\right\vert $ and $\left\vert \overline{\mathbf{g}}\right\vert =\min
\left( \left\vert \widehat{\mathbf{g}}\right\vert ,\left\vert \overline{%
\mathbf{g}}\right\vert \right) <1$, and, hence,%
\begin{eqnarray*}
\frac{\left\vert \widehat{\mathbf{g}}\right\vert \left\vert \overline{%
\mathbf{g}}\right\vert +\left( \left\vert \widehat{\mathbf{g}}\right\vert
\left\vert \overline{\mathbf{g}}\right\vert \right) ^{\gamma /2}}{\left\vert 
\widehat{\mathbf{g}}\right\vert ^{2}} &\leq & \frac{1}{\left\vert \widehat{\mathbf{%
g}}\right\vert }\left( 1+\frac{1}{\left\vert \mathbf{g}\right\vert
^{1-\gamma /2}}\right) \\
&\leq &\frac{C}{\left\vert \widehat{\mathbf{g}}\right\vert }\left( 1+\frac{1%
}{\left\vert \mathbf{g}\right\vert }\right) \text{,}
\end{eqnarray*}%
and hence, by expression $\left( \ref{exp2}\right) $, assumption $\left( \ref%
{est1}\right) $, and inequality $\left( \ref{ineq1a}\right) $, the bound $%
\left( \ref{b6}\right) $ is again satisfied.

(ii) $\left\vert \overline{\mathbf{g}}%
\right\vert =\max \left( \left\vert \widehat{\mathbf{g}}\right\vert
,\left\vert \overline{\mathbf{g}}\right\vert \right) \geq \left\vert \mathbf{%
g}\right\vert $ and $\left\vert \widehat{\mathbf{g}}\right\vert =\min \left(
\left\vert \widehat{\mathbf{g}}\right\vert ,\left\vert \overline{\mathbf{g}}%
\right\vert \right) <1$, implying, correspondingly, that%
\begin{equation*}
\frac{\left\vert \widehat{\mathbf{g}}\right\vert \left\vert \overline{%
\mathbf{g}}\right\vert +\left( \left\vert \widehat{\mathbf{g}}\right\vert
\left\vert \overline{\mathbf{g}}\right\vert \right) ^{\gamma /2}}{\left\vert 
\overline{\mathbf{g}}\right\vert ^{2}}\leq \frac{C}{\left\vert \overline{%
\mathbf{g}}\right\vert }\left( 1+\frac{1}{\left\vert \mathbf{g}\right\vert }%
\right) \text{,}
\end{equation*}%
and hence, by expression $\left( \ref{exp3}\right) $, assumption $\left( \ref%
{est1}\right) $, and inequality $\left( \ref{ineq1a}\right) $, one may again
obtain the bound $\left( \ref{b6}\right) $.

Note that%
\begin{eqnarray}
&&\dfrac{m_{\alpha }m_{\beta }}{\left( \sqrt{m_{\alpha }}+\sqrt{m_{\beta }}%
\right) ^{2}}\widehat{\Delta }_{\alpha \beta }I_{\#}-\left( I^{\prime
}+I_{\ast }^{\prime }\right) \notag \\
&=&\dfrac{\sqrt{m_{\alpha }}-\sqrt{m_{\beta }}}{%
\sqrt{m_{\alpha }}+\sqrt{m_{\beta }}}\Delta I_{\#}-\left( I^{\prime
}+I_{\ast }^{\prime }\right) \notag \\
&=&\dfrac{\sqrt{m_{\alpha }}-\sqrt{m_{\beta }}}{%
\sqrt{m_{\alpha }}+\sqrt{m_{\beta }}}\left( I-I_{\ast }\right) -\dfrac{\sqrt{%
m_{\alpha }}I^{\prime }+\sqrt{m_{\beta }}I_{\ast }^{\prime }}{\sqrt{%
m_{\alpha }}+\sqrt{m_{\beta }}}\text{.}  \label{exp4}
\end{eqnarray}%
Moreover, 
\begin{equation}
\widehat{\mathcal{E}}_{\alpha \beta }\left( \widehat{\mathcal{E}}_{\alpha
\beta }^{\ast }\right) \geq I^{1-\zeta _{\alpha }-\pi _{\alpha }}\left(
I^{\prime }\right) ^{\zeta _{\alpha }}I_{\ast }^{1-\zeta _{\beta }-\pi
_{\beta }}\left( I_{\ast }^{\prime }\right) ^{\zeta _{\beta }}\left\vert 
\mathbf{g}\right\vert ^{2\left( \pi _{\alpha }+\pi _{\beta }\right) }\text{,}
\label{b6a}
\end{equation}%
where $0\leq \pi _{\gamma }\leq 1-\zeta _{\gamma }$, $0\leq \zeta _{\gamma
}<1/2$, with $\pi _{\gamma }=0$ if the species $a_{\gamma }$, $\gamma \in
\left\{ 1,...,s_{0}\right\} $, is monatomic, for $\gamma \in \left\{ \alpha
,\beta \right\} $.

Therefore, by inequality $\left( \ref{ineq1}\right) $, expression $\left( %
\ref{exp4}\right) $, and the bounds $\left( \ref{b6}\right) $ and $\left( %
\ref{b6a}\right) $, changing variables of integration $\left\{ \boldsymbol{%
\xi },\boldsymbol{\xi }_{\ast }\right\} \rightarrow \left\{ \mathbf{g},%
\mathbf{g}_{\alpha \beta }\right\} $, with unitary Jacobian,%
\begin{eqnarray}
&&\int_{\mathbb{R}^{3}\times \Omega }\left( k_{\alpha \beta 2}^{\left( \beta
\right) }(\boldsymbol{\xi },\boldsymbol{\xi }_{\ast },I,I_{\ast })\right)
^{2}d\boldsymbol{\xi \,}d\boldsymbol{\xi }_{\ast }  \notag \\
&\leq &C\int_{\Omega }\frac{e^{-A_{2}\left( \left\vert \mathbf{g}\right\vert
^{2}-2\widehat{\Delta }_{\alpha \beta }I_{\#}\right) }\mathbf{1}_{\left\vert 
\mathbf{g}\right\vert ^{2}>2\widehat{\Delta }_{\alpha \beta }I_{\#}}}{%
I^{2-2\zeta _{\alpha }-2\pi _{\alpha }}I_{\ast }^{2-2\zeta _{\beta }-2\pi
_{\beta }}\left\vert \mathbf{g}\right\vert ^{4\left( \pi _{\alpha }+\pi
_{\beta }\right) }}\left( 1+\frac{1}{\left\vert \mathbf{g}\right\vert ^{2}}%
\right) d\mathbf{g}  \notag \\
&&\times \int_{\mathbb{R}^{3}}e^{-\left( \sqrt{m_{\alpha }}-\sqrt{m_{\beta }}%
\right) ^{2}\left\vert \mathbf{g}_{\alpha \beta }\right\vert ^{2}/4}d\mathbf{%
g}_{\alpha \beta }  \notag \\
&&\times \int_{0}^{\infty }\left( I^{\prime }\right) ^{\delta ^{\left(
\alpha \right) }/4-1/2-2\zeta _{\alpha }}e^{-\sqrt{m_{\alpha }}I^{\prime
}/\left( 2\left( \sqrt{m_{\alpha }}+\sqrt{m_{\beta }}\right) \right)
}dI^{\prime }dI_{\ast }^{\prime }  \notag \\
&&\times \int_{0}^{\infty }\left( I_{\ast }^{\prime }\right) ^{\delta
^{\left( \beta \right) }/4-1/2-2\zeta _{\beta }}e^{-\sqrt{m_{\beta }}I_{\ast
}^{\prime }/\left( 2\left( \sqrt{m_{\alpha }}+\sqrt{m_{\beta }}\right)
\right) }dI^{\prime }dI_{\ast }^{\prime }  \notag \\
&\leq &C\int_{\Omega }\frac{e^{A_{1}\left( I-I_{\ast }\right)
-A_{2}\left\vert \mathbf{g}\right\vert ^{2}}}{I^{2-2\zeta _{\alpha }-2\pi
_{\alpha }}I_{\ast }^{2-2\zeta _{\beta }-2\pi _{\beta }}\left\vert \mathbf{g}%
\right\vert ^{4\left( \pi _{\alpha }+\pi _{\beta }\right) }}\left( 1+\frac{1%
}{\left\vert \mathbf{g}\right\vert ^{2}}\right) d\mathbf{g}\text{,}
\label{b9}
\end{eqnarray}%
with%
\begin{equation*}
A_{1}=\frac{\sqrt{m_{\alpha }}-\sqrt{m_{\beta }}}{2\left( \sqrt{m_{\alpha }}+%
\sqrt{m_{\beta }}\right) }\text{ and }A_{2}=\frac{m_{\alpha }m_{\beta }}{%
4\left( \sqrt{m_{\alpha }}+\sqrt{m_{\beta }}\right) ^{2}}>0
\end{equation*}%
for $\Omega \subseteq \mathbb{R}^{3}$, since, by a change to spherical
coordinates,%
\begin{equation}
\int_{\mathbb{R}^{3}}e^{-\left( \sqrt{m_{\alpha }}-\sqrt{m_{\beta }}\right)
^{2}\left\vert \mathbf{g}_{\alpha \beta }\right\vert ^{2}/2}d\mathbf{g}%
_{\alpha \beta }=C\int_{0}^{\infty }R^{2}e^{-R^{2}}dR=C\text{.}  \label{eq2}
\end{equation}%
Without loss of generality we can assume that $m_{\alpha }>m_{\beta }$, and
then $A_{1}>0$. On the other hand, for $\Omega \subseteq \mathbb{R}^{3}$, by
inequality $\left( \ref{ineq1}\right) $, and the bounds $\left( \ref{b6}%
\right) $ and $\left( \ref{b6a}\right) $, also%
\begin{eqnarray}
&&\int_{\mathbb{R}^{3}\times \Omega }\left( k_{\alpha \beta 2}^{\left( \beta
\right) }(\boldsymbol{\xi },\boldsymbol{\xi }_{\ast },I,I_{\ast })\right)
^{2}d\boldsymbol{\xi \,}d\boldsymbol{\xi }_{\ast }  \notag \\
&\leq &C\int_{\Omega \times \mathbb{R}_{+}^{2}}\frac{e^{-A_{2}\left(
\left\vert \mathbf{g}\right\vert ^{2}-2\widehat{\Delta }_{\alpha \beta
}I_{\#}\right) }\mathbf{1}_{\left\vert \mathbf{g}\right\vert ^{2}>2\widehat{%
\Delta }_{\alpha \beta }I_{\#}}}{I^{2-2\zeta _{\alpha }-2\pi _{\alpha
}}I_{\ast }^{2-2\zeta _{\beta }-2\pi _{\beta }}\left\vert \mathbf{g}%
\right\vert ^{4\left( \pi _{\alpha }+\pi _{\beta }\right) }}\left( 1+\frac{1%
}{\left\vert \mathbf{g}\right\vert ^{2}}\right) d\mathbf{g}  \notag \\
&&\times \left( I^{\prime }\right) ^{\delta ^{\left( \alpha \right)
}/4-1/2-2\zeta _{\alpha }}e^{-I^{\prime }/2}\left( I_{\ast }^{\prime
}\right) ^{\delta ^{\left( \beta \right) }/4-1/2-2\zeta _{\beta
}}e^{-I_{\ast }^{\prime }/2}dI^{\prime }dI_{\ast }^{\prime }  \notag \\
&&\times \int_{\mathbb{R}^{3}}e^{-\left( \sqrt{m_{\alpha }}-\sqrt{m_{\beta }}%
\right) ^{2}\left\vert \mathbf{g}_{\alpha \beta }\right\vert ^{2}/4}d\mathbf{%
g}_{\alpha \beta }  \notag \\
&\leq &C\int_{\Omega }\frac{1}{I^{2-2\zeta _{\alpha }-2\pi _{\alpha
}}I_{\ast }^{2-2\zeta _{\beta }-2\pi _{\beta }}\left\vert \mathbf{g}%
\right\vert ^{4\left( \pi _{\alpha }+\pi _{\beta }\right) }}\left( 1+\frac{1%
}{\left\vert \mathbf{g}\right\vert ^{2}}\right) d\mathbf{g}  \notag \\
&&\times \int_{0}^{\infty }\left( I^{\prime }\right) ^{\delta ^{\left(
\alpha \right) }/4-1/2-2\zeta _{\alpha }}e^{-I^{\prime }/2}dI^{\prime
}\int_{0}^{\infty }\left( I_{\ast }^{\prime }\right) ^{\delta ^{\left( \beta
\right) }/4-1/2-2\zeta _{\beta }}e^{-I_{\ast }^{\prime }/2}dI_{\ast
}^{\prime }  \notag \\
&\leq &C\int_{\Omega }\frac{1}{I^{2-2\zeta _{\alpha }-2\pi _{\alpha
}}I_{\ast }^{2-2\zeta _{\beta }-2\pi _{\beta }}\left\vert \mathbf{g}%
\right\vert ^{4\left( \pi _{\alpha }+\pi _{\beta }\right) }}\left( 1+\frac{1%
}{\left\vert \mathbf{g}\right\vert ^{2}}\right) d\mathbf{g}\text{.}
\label{b10}
\end{eqnarray}

By the bound $\left( \ref{b9}\right) $, with $\zeta _{\gamma }=4\pi _{\gamma
}=\left\{ 
\begin{array}{l}
4/9\ \text{if }\gamma \in \left\{ s_{0}+1,...,s\right\} \\ 
\ \ 0\text{ \ if }\gamma \in \left\{ 1,...,s_{0}\right\}%
\end{array}%
\right. $, for $\gamma \in \left\{ \alpha ,\beta \right\} $,%
\begin{eqnarray*}
&&\int_{\left( \mathbb{R}^{3}\times \left[ 0,2\right] \right) ^{2}}\left(
k_{\alpha \beta 2}^{\left( \beta \right) }(\boldsymbol{\xi },\boldsymbol{\xi 
}_{\ast },I,I_{\ast })\right) ^{2}d\boldsymbol{\xi \,}d\boldsymbol{\xi }%
_{\ast }dI\boldsymbol{\,}dI_{\ast } \\
&\leq &C\int_{\mathbb{R}^{3}}\frac{\left( 1+\left\vert \mathbf{g}\right\vert
^{2}\right) e^{-A_{2}\left\vert \mathbf{g}\right\vert ^{2}}}{\left\vert 
\mathbf{g}\right\vert ^{2+4\left( \pi _{\alpha }+\pi _{\beta }\right) }}d%
\mathbf{g}\int_{0}^{2}I^{2\left( \zeta _{\alpha }+\pi _{\alpha }-1\right) }%
\boldsymbol{\,}dI\int_{0}^{2}I_{\ast }^{2\left( \zeta _{\beta }+\pi _{\beta
}-1\right) }\boldsymbol{\,}dI_{\ast } \\
&\leq &C\int_{\mathbb{R}^{3}}\frac{1+R^{2}}{R^{4\left( \pi _{\alpha }+\pi
_{\beta }\right) }}e^{-A_{2}R^{2}}dR=C\text{.}
\end{eqnarray*}%
Now let $I_{\ast }\geq I$, and assume that $I_{\ast }>2$ - otherwise $\left(
I,I_{\ast }\right) \in \left[ 0,2\right] ^{2}$. Then, by the bound $\left( %
\ref{b9}\right) $, with $\zeta _{\alpha }=4\pi _{\alpha }=\left\{ 
\begin{array}{l}
2/5\text{ if }\alpha \in \left\{ s_{0}+1,...,s\right\} \text{ and }I\leq 2
\\ 
\ \ 0\text{ \ if }\alpha \in \left\{ 1,...,s_{0}\right\} \text{, or, }I>2%
\end{array}%
\right. $ and $\zeta _{\beta }=10\pi _{\beta }=10/21$, 
\begin{eqnarray*}
&&\int_{2}^{\infty }\int_{0}^{I_{\ast }}\int_{\left( \mathbb{R}^{3}\right)
^{2}}\left( k_{\alpha \beta 2}^{\left( \beta \right) }(\boldsymbol{\xi },%
\boldsymbol{\xi }_{\ast },I,I_{\ast })\right) ^{2}d\boldsymbol{\xi \,}d%
\boldsymbol{\xi }_{\ast }dI\boldsymbol{\,}dI_{\ast } \\
&\leq &C\int_{0}^{\infty }\frac{1}{I^{2-2\zeta _{\alpha }-2\pi _{\alpha }}}%
dI\int_{\mathbb{R}^{3}}\frac{\left( 1+\left\vert \mathbf{g}\right\vert
^{2}\right) e^{-A_{2}\left\vert \mathbf{g}\right\vert ^{2}}}{\left\vert 
\mathbf{g}\right\vert ^{2+4\left( \pi _{\alpha }+\pi _{\beta }\right) }}d%
\mathbf{g} \\
&&\times \int_{0}^{\infty }\frac{e^{-A_{1}\left( I_{\ast }-I\right) }}{%
\left( I_{\ast }-I\right) ^{2-2\zeta _{\beta }-2\pi _{\beta }}}d\left(
I_{\ast }-I\right) \\
&\leq &C\int_{0}^{\infty }\frac{1}{I^{2-2\zeta _{\alpha }-2\pi _{\alpha }}}%
dI\int_{0}^{\infty }\frac{e^{-A_{2}R^{2}}}{R^{4\left( \pi _{\alpha }+\pi
_{\beta }\right) }}\left( 1+R^{2}\right) dR \\
&\leq &C\int_{0}^{2}\left( 1+I^{-4/5}\right) dI+\int_{2}^{\infty }I^{-2}dI=C%
\text{.}
\end{eqnarray*}

On the other hand, if $I\geq I_{\ast }$, then assume that $I>2$ - otherwise $%
\left( I,I_{\ast }\right) \in \left[ 0,2\right] ^{2}$. Then, by the bound $%
\left( \ref{b10}\right) $, with $\pi _{\alpha }=\zeta _{\alpha }=0$ and $%
\zeta _{\beta }=2\pi _{\beta }=\left\{ 
\begin{array}{l}
2/5\ \text{if }\beta \in \left\{ s_{0}+1,...,s\right\} \\ 
\ \ 0\text{ \ if }\beta \in \left\{ 1,...,s_{0}\right\}%
\end{array}%
\right. $, 
\begin{eqnarray*}
&&\int_{2}^{\infty }\int_{0}^{I}\int_{\left\vert \mathbf{g}\right\vert \leq
1}\left( k_{\alpha \beta 2}^{\left( \beta \right) }(\boldsymbol{\xi },%
\boldsymbol{\xi }_{\ast },I,I_{\ast })\right) ^{2}d\boldsymbol{\xi \,}d%
\boldsymbol{\xi }_{\ast }dI_{\ast }\boldsymbol{\,}dI \\
&\leq &C\int_{2}^{\infty }I^{-2}\int_{0}^{I}\frac{1}{I_{\ast }^{2-2\zeta
_{\beta }-2\pi _{\beta }}}\mu _{\beta }\left( I_{\ast }\right) dI_{\ast
}dI\int_{\left\vert \mathbf{g}\right\vert \leq 1}\frac{1}{\left\vert \mathbf{%
g}\right\vert ^{2+4\pi _{\beta }}}d\mathbf{g} \\
&\leq &C\int_{2}^{\infty }I^{-2}I^{1/5}dI\int_{\left\vert \mathbf{g}%
\right\vert \leq 1}\left\vert \mathbf{g}\right\vert ^{-14/5}d\mathbf{g} \\
&=&C\int_{2}^{\infty }I^{-9/5}dI\int_{0}^{1}R^{-4/5}dR=C\text{.}
\end{eqnarray*}%
Moreover, assuming that the species $a_{\beta }$, $\beta \in \left\{
s_{0}+1,...,s\right\} $, is polyatomic, then, by the bound $\left( \ref{b10}%
\right) $, with $\zeta _{\alpha }=\zeta _{\beta }=0$ and $1-2\pi _{\alpha
}=\pi _{\beta }=3/5$,%
\begin{eqnarray*}
&&\int_{2}^{\infty }\int_{0}^{I}\int_{\left\vert \mathbf{g}\right\vert \geq
1}\left( k_{\alpha \beta 2}^{\left( \beta \right) }(\boldsymbol{\xi },%
\boldsymbol{\xi }_{\ast },I,I_{\ast })\right) ^{2}d\boldsymbol{\xi \,}d%
\boldsymbol{\xi }_{\ast }dI_{\ast }\boldsymbol{\,}dI \\
&\leq &C\int_{2}^{\infty }\frac{1}{I^{8/5}}\int_{0}^{I}\frac{1}{I_{\ast
}^{4/5}}dI_{\ast }dI\int_{\left\vert \mathbf{g}\right\vert \geq 1}\frac{1}{%
\left\vert \mathbf{g}\right\vert ^{4\left( \pi _{\alpha }+\pi _{\beta
}\right) }}d\mathbf{g} \\
&\leq &C\int_{2}^{\infty }\frac{1}{I^{7/5}}dI\int_{\left\vert \mathbf{g}%
\right\vert \geq 1}\frac{1}{\left\vert \mathbf{g}\right\vert ^{16/5}}d%
\mathbf{g}=C\int_{1}^{\infty }R^{-6/5}dR=C\text{.}
\end{eqnarray*}

Assume now that the species $a_{\beta }$, $\beta \in \left\{
1,...,s_{0}\right\} $, is monatomic (i.e. with $I_{\ast }=I_{\ast }^{\prime
}=1$). If $A_{1}\left( I-1\right) \leq A_{2}\left\vert \mathbf{g}\right\vert
^{2}$, by the bound $\left( \ref{b9}\right) $, with $\zeta _{\alpha }=\zeta
_{\beta }=\pi _{\beta }=1-\pi _{\alpha }=0$,%
\begin{eqnarray*}
&&\int_{2}^{\infty }\int_{0}^{I}\int_{\substack{ \left\vert \mathbf{g}%
\right\vert \geq 1  \\ A_{1}\left( I-1\right) \leq A_{2}\left\vert \mathbf{g}%
\right\vert ^{2}}}\left( k_{\alpha \beta 2}^{\left( \beta \right) }(%
\boldsymbol{\xi },\boldsymbol{\xi }_{\ast },I,I_{\ast })\right) ^{2}d%
\boldsymbol{\xi \,}d\boldsymbol{\xi }_{\ast }dI_{\ast }\boldsymbol{\,}dI \\
&\leq &\int_{1}^{1+A_{2}\left\vert \mathbf{g}\right\vert
^{2}/A_{1}}\int_{\left\vert \mathbf{g}\right\vert \geq 1}\int_{\left\vert 
\mathbf{g}\right\vert \geq 1}\left( k_{\alpha \beta 2}^{\left( \beta \right)
}(\boldsymbol{\xi },\boldsymbol{\xi }_{\ast },I,1)\right) ^{2}d\boldsymbol{%
\xi \,}d\boldsymbol{\xi }_{\ast }\boldsymbol{\,}dI \\
&\leq &C\int_{\left\vert \mathbf{g}\right\vert \geq 1}\left\vert \mathbf{g}%
\right\vert ^{-4}\int_{0}^{A_{2}\left\vert \mathbf{g}\right\vert
^{2}/A_{1}}e^{A_{1}\left( I-1\right) -A_{2}\left\vert \mathbf{g}\right\vert
^{2}}d\left( I-1\right) d\mathbf{g} \\
&\leq &C\int_{1}^{\infty }R^{-2}dR=C\text{.}
\end{eqnarray*}%
On the other hand, if $A_{1}\left( I-1\right) \geq A_{2}\left\vert \mathbf{g}%
\right\vert ^{2}\geq A_{2}$, then, by inequality $\left( \ref{ineq1}\right) $
and the bounds $\left( \ref{b6}\right) $ and $\left( \ref{b6a}\right) $,
with $\zeta _{\alpha }=\zeta _{\beta }=\pi _{\beta }=1-\pi _{\alpha }=0$, 
\begin{eqnarray*}
& & \left( k_{\alpha \beta 2}^{\left( \beta \right) }(\boldsymbol{\xi },%
\boldsymbol{\xi }_{\ast },I,I_{\ast })\right) ^{2}  \\
&\leq & C\left( 1+\frac{1}{%
\left\vert \mathbf{g}\right\vert ^{2}}\right) \int_{\mathbb{R}_{+}^{2}}\frac{%
\left( M_{\alpha }^{\prime }M_{\beta \ast }^{\prime }\right) ^{1/2}}{%
\left\vert \mathbf{g}\right\vert ^{4}}dI^{\prime }dI_{\ast }^{\prime } \\
&\leq &C\frac{e^{-\left( \sqrt{m_{\alpha }}-\sqrt{m_{\beta }}\right)
^{2}\left\vert \mathbf{g}_{\alpha \beta }\right\vert ^{2}/4}}{\left\vert 
\mathbf{g}\right\vert ^{4}}\int_{0}^{\infty }\left( I^{\prime }\right)
^{\delta ^{\left( \alpha \right) }/2-1}e^{-I^{\prime }/2}dI^{\prime }
\end{eqnarray*}%
Noting that%
\begin{equation*}
I^{\prime }=I_{\ast }^{\prime }+I-I_{\ast }+\dfrac{A_{2}}{A_{1}}\left(
\left\vert \mathbf{g}^{\prime }\right\vert ^{2}-\left\vert \mathbf{g}%
\right\vert ^{2}\right) \geq I-\dfrac{A_{2}}{A_{1}}\left\vert \mathbf{g}%
\right\vert ^{2}\geq 1\text{,}
\end{equation*}%
and changing variables of integration $\left\{ \boldsymbol{\xi },\boldsymbol{%
\xi }_{\ast }\right\} \rightarrow \left\{ \mathbf{g},\mathbf{g}_{\alpha
\beta }\right\} $, with unitary Jacobian, we obtain the bound%
\begin{eqnarray*}
&&\int_{2}^{\infty }\int_{0}^{I}\int_{\substack{ \left\vert \mathbf{g}%
\right\vert \geq 1  \\ A_{1}\left( I-1\right) \geq A_{2}\left\vert \mathbf{g}%
\right\vert ^{2}}}\left( k_{\alpha \beta 2}^{\left( \beta \right) }(%
\boldsymbol{\xi },\boldsymbol{\xi }_{\ast },I,I_{\ast })\right) ^{2}d%
\boldsymbol{\xi \,}d\boldsymbol{\xi }_{\ast }dI_{\ast }\boldsymbol{\,}dI \\
&=&\int_{2}^{\infty }\int_{\substack{ \left\vert \mathbf{g}\right\vert \geq
1  \\ A_{1}\left( I-1\right) \geq A_{2}\left\vert \mathbf{g}\right\vert ^{2} 
}}\left( k_{\alpha \beta 2}^{\left( \beta \right) }(\boldsymbol{\xi },%
\boldsymbol{\xi }_{\ast },I,1)\right) ^{2}d\boldsymbol{\xi \,}d\boldsymbol{%
\xi }_{\ast }\boldsymbol{\,}dI \\
&\leq &C\int_{1}^{\infty }\left\vert \mathbf{g}\right\vert
^{-4}\int_{A_{2}\left\vert \mathbf{g}\right\vert ^{2}/A_{1}}^{\infty
}e^{-\left( I-A_{2}\left\vert \mathbf{g}\right\vert ^{2}/A_{1}\right) /4}dId%
\mathbf{g} \\
&&\times \int_{\mathbb{R}^{3}}e^{-\left( \sqrt{m_{\alpha }}-\sqrt{m_{\beta }}%
\right) ^{2}\left\vert \mathbf{g}_{\alpha \beta }\right\vert ^{2}/4}d\mathbf{%
g}_{\alpha \beta }\int_{0}^{\infty }\left( I^{\prime }\right) ^{\delta
^{\left( \alpha \right) }/2-1}e^{-I^{\prime }/4}dI^{\prime } \\
&\leq &C\int_{1}^{\infty }R^{-2}dR=C\text{.}
\end{eqnarray*}

Concluding, $\left( k_{\alpha \beta 2}^{\left( \beta \right) }(\boldsymbol{%
\xi },\boldsymbol{\xi }_{\ast },I,I_{\ast })\right) ^{2}\in L^{2}\left( d%
\boldsymbol{\xi \,}\,d\boldsymbol{\xi }_{\ast }dI\,dI_{\ast }\right) $,
implying that%
\begin{equation*}
K_{\alpha \beta 2}=\int_{\mathbb{R}^{3}}k_{\alpha \beta 2}^{\left( \beta
\right) }(\boldsymbol{\xi },\boldsymbol{\xi }_{\ast },I,I_{\ast })h_{\beta
\ast }\,d\boldsymbol{\xi }_{\ast }dI_{\ast }
\end{equation*}%
are Hilbert-Schmidt integral operators, and as such continuous and compact
on $L^{2}\left( d\mathbf{Z}_{\alpha }\right) $ \cite[Theorem 7.83]%
{RenardyRogers}, for $\left\{ \alpha ,\beta \right\} \subseteq \left\{
1,...,s\right\} $.

On the other hand, if $m_{\alpha }=m_{\beta }$, then 
\begin{eqnarray*}
&&k_{\alpha \beta 2}^{\left( \beta \right) }(\boldsymbol{\xi },\boldsymbol{%
\xi }_{\ast },I,I_{\ast })\\
&=&\int_{\left( \mathbb{R}^{3}\right) ^{\perp _{\mathbf{n}}}}4\frac{%
\left\vert \widehat{\mathbf{g}}\right\vert \left( M_{\alpha }^{\prime
}M_{\beta \ast }^{\prime }\right) ^{1/2}}{\left\vert \overline{\mathbf{g}}%
\right\vert \left\vert \mathbf{g}\right\vert } \frac{I^{\delta ^{\left(
\alpha \right) }/4-1/2}\left( I_{\ast }^{\prime }\right) ^{\delta ^{\left(
\beta \right) }/4-1/2}}{\left( I^{\prime }\right) ^{\delta ^{\left( \alpha
\right) }/4-1/2}I_{\ast }^{\delta ^{\left( \beta \right) }/4-1/2}}
\\
&\times&\mathbf{1}_{m_{\alpha }\left\vert \widehat{\mathbf{g}}\right\vert ^{2}>4I_{\#}}%
\widehat{\sigma }_{\alpha \beta }\,d\mathbf{w}\text{, with }\widehat{\mathbf{g}}=\boldsymbol{\xi }-\boldsymbol{\xi }_{\ast
}^{\prime }\text{ and }\overline{\mathbf{g}}=\boldsymbol{\xi }_{\ast }-%
\boldsymbol{\xi }.
\end{eqnarray*}%
Here, with $\mathbf{g}=\boldsymbol{\xi }-\boldsymbol{\xi }_{\ast }$ and $%
\mathbf{n}=\mathbf{g/}\left\vert \mathbf{g}\right\vert $,%
\begin{equation*}
\left\{ 
\begin{array}{l}
\boldsymbol{\xi }^{\prime }=\boldsymbol{\xi }+\mathbf{w}-\chi _{\alpha }%
\mathbf{n} \\ 
\boldsymbol{\xi }_{\ast }^{\prime }=\boldsymbol{\xi }_{\ast }+\mathbf{w}%
-\chi _{\alpha }\mathbf{n}%
\end{array}%
\right. \text{, where }\mathbf{w}\perp \mathbf{g}\text{ and }\chi _{\alpha }=%
\frac{\Delta I}{m_{\alpha }\left\vert \mathbf{g}\right\vert }\text{.}
\end{equation*}%
Then similar arguments to the ones for $k_{\alpha \beta }^{\left( \alpha
\right) }(\boldsymbol{\xi },\boldsymbol{\xi }_{\ast },I,I_{\ast })$ (with $%
m_{\alpha }=m_{\beta }$) above, can be applied.

Concluding, the operator 
\begin{equation*}
K=(K_{1},...,K_{s})=\sum\limits_{\beta =1}^{s}(K_{1\beta 2}+K_{1\beta
3}-K_{1\beta 1},...,K_{s\beta 2}+K_{s\beta 3}-K_{s\beta 1})\text{,}
\end{equation*}%
is a compact self-adjoint operator on $\mathcal{\mathfrak{h}}$.
Self-adjointness is due to the symmetry relations $\left( \ref{sa1}\right)
,\left( \ref{sa2}\right) $, cf. \cite[p.198]{Yoshida-65}.
\end{proof}

\section{Bounds on the collision frequency \label{PT2}}

This section concerns the proof of Theorem \ref{Thm2}. Note that throughout
the proof, $C$ and $C_{\alpha \beta }$ for $\left\{ \alpha ,\beta \right\}
\subseteq \left\{ 1,...,s\right\} $, will denote generic positive constants.

\begin{proof}
Noting that $\left( \ref{df1}\right) $, under assumption $\left( \ref{e1}%
\right) $, the collision frequencies $\nu _{\alpha }$ equal 
\begin{eqnarray*}
\nu _{\alpha } &=&\sum\limits_{\beta =1}^{s}\int_{\left( \mathbb{R}%
^{3}\times \mathbb{R}_{+}\right) ^{3}}\frac{M_{\beta \ast }}{I^{\delta
^{\left( \alpha \right) }/2-1}I_{\ast }^{\delta ^{\left( \beta \right) }/2-1}%
}W_{\alpha \beta }d\boldsymbol{\xi }_{\ast }d\boldsymbol{\xi }^{\prime }d%
\boldsymbol{\xi }_{\ast }^{\prime }dI_{\ast }dI^{\prime }dI_{\ast }^{\prime }
\\
&=&\sum\limits_{\beta =1}^{s}\int_{\left( \mathbb{R}^{3}\right) ^{2}\times 
\mathbb{R}_{+}^{4}x\mathbb{S}^{2}}M_{\beta \ast }\sigma _{\alpha \beta }%
\left\vert \mathbf{g}\right\vert \delta _{1}\left( \sqrt{%
\left\vert \mathbf{g}\right\vert ^{2}-\dfrac{2\Delta I}{\mu_{\alpha \beta }}}-\left\vert 
\mathbf{g}^{\prime }\right\vert \right) \\
&&\mathbf{1}_{\mu_{\alpha \beta }\left\vert \mathbf{g}\right\vert ^{2}>2\Delta I}\times \delta _{3}\left( \mathbf{G}_{\alpha \beta }-\mathbf{G}_{\alpha
\beta }^{\prime }\right) d\boldsymbol{\xi }_{\ast }d\mathbf{G}_{\alpha \beta
}^{\prime }d\left\vert \mathbf{g}^{\prime }\right\vert d\boldsymbol{\omega }%
dI_{\ast }dI^{\prime }dI_{\ast }^{\prime } \\
&=&\sum\limits_{\beta =1}^{s}C_{\alpha \beta }\int_{\mathbb{S}^{2}}d%
\boldsymbol{\omega }\int_{\mathbb{R}^{3}\times \left( \mathbb{R}_{+}\right)
^{3}}I_{\ast }^{\delta ^{\left( \beta \right) }/2-1}e^{-I_{\ast
}}e^{-m_{\beta }\left\vert \boldsymbol{\xi }_{\ast }\right\vert
^{2}/2}\sigma _{\alpha \beta } \\
&&\times \mathbf{1}_{\mu_{\alpha \beta }\left\vert \mathbf{g}\right\vert ^{2}>2\Delta I}\left\vert \mathbf{g}\right\vert d\boldsymbol{\xi }%
_{\ast }dI_{\ast }dI^{\prime }dI_{\ast }^{\prime } \\
&=&\sum\limits_{\beta =1}^{s}C_{\alpha \beta }\int_{\mathbb{R}^{3}\times
\left( \mathbb{R}_{+}\right) ^{3}}e^{-m_{\beta }\left\vert \boldsymbol{\xi }%
_{\ast }\right\vert ^{2}/2}e^{-I_{\ast }}\frac{\left( I^{\prime }\right)
^{\delta ^{\left( \alpha \right) }/2-1}\left( I_{\ast }I_{\ast }^{\prime
}\right) ^{\delta ^{\left( \beta \right) }/2-1}}{\mathcal{E}_{\alpha \beta
}^{\delta ^{\left( \alpha \right) }/2}\left( \mathcal{E}_{\alpha \beta
}^{\ast }\right) ^{\delta ^{\left( \beta \right) }/2}E_{\alpha \beta }^{\eta
/2}} \\
&&\times \mathbf{1}_{\mu_{\alpha \beta }\left\vert \mathbf{g}\right\vert ^{2}>2\Delta I}\sqrt{\left\vert \mathbf{g}\right\vert ^{2}-2%
\widetilde{\Delta }_{\alpha \beta }I}d\boldsymbol{\xi }_{\ast }dI_{\ast
}dI^{\prime }dI_{\ast }^{\prime }
\end{eqnarray*}%
for $\alpha \in \left\{ 1,...,s\right\} $.

Clearly, 
\begin{eqnarray*}
\nu _{\alpha } &\geq &C_{\alpha \alpha }\int_{\mathbb{R}^{3}\times \left( 
\mathbb{R}_{+}\right) ^{3}}\frac{\left( I_{\ast }I^{\prime }I_{\ast
}^{\prime }\right) ^{\delta ^{\left( \alpha \right) }/2-1}e^{-I_{\ast }}}{%
\mathcal{E}_{\alpha \alpha }^{\delta ^{\left( \alpha \right) }}E_{\alpha
\alpha }^{\eta /2}}\sqrt{\left\vert \mathbf{g}\right\vert ^{2}-\frac{4}{%
m_{\alpha }}\Delta I} \\
&&\times e^{-m_{\alpha }\left\vert \boldsymbol{\xi }_{\ast }\right\vert
^{2}/2}\mathbf{1}_{m_{\alpha }\left\vert \mathbf{g}\right\vert ^{2}>4\Delta
I}d\boldsymbol{\xi }_{\ast }dI_{\ast }dI^{\prime }dI_{\ast }^{\prime }\text{.%
}
\end{eqnarray*}%
Then, if the species $a_{\alpha }$, $\alpha \in \left\{ 1,...,s_{0}\right\} $%
, is monatomic%
\begin{eqnarray*}
\nu _{\alpha } &\geq &C\int_{\mathbb{R}^{3}\times \left( \mathbb{R}%
_{+}\right) ^{3}}e^{-m_{\alpha }\left\vert \boldsymbol{\xi }_{\ast
}\right\vert ^{2}/2}\left\vert \mathbf{g}\right\vert ^{1-\eta }d\boldsymbol{%
\xi }_{\ast } \\
&\geq &C\int_{\mathbb{R}^{3}\times \left( \mathbb{R}_{+}\right)
^{3}}e^{-m_{\alpha }\left\vert \boldsymbol{\xi }_{\ast }\right\vert
^{2}/2}\left( \left\vert \left\vert \boldsymbol{\xi }\right\vert -\left\vert 
\boldsymbol{\xi }_{\ast }\right\vert \right\vert ^{2}\right) ^{\left( 1-\eta
\right) /2}d\boldsymbol{\xi }_{\ast }\text{.}
\end{eqnarray*}%
while if $a_{\alpha }$, $\alpha \in \left\{ s_{0}+1,...,s\right\} $, is
polyatomic%
\begin{eqnarray*}
\nu _{\alpha } &\geq &C\int_{\mathbb{R}^{3}\times \mathbb{R}%
_{+}}\int_{I^{\prime }+I_{\ast }^{\prime }\leq E_{\alpha \alpha
}/2}e^{-m_{\alpha }\left\vert \boldsymbol{\xi }_{\ast }\right\vert ^{2}/2}%
\frac{\left( I_{\ast }I^{\prime }I_{\ast }^{\prime }\right) ^{\delta
^{\left( \alpha \right) }/2-1}e^{-I_{\ast }}}{\mathcal{E}_{\alpha \alpha
}^{\delta ^{\left( \alpha \right) }}E_{\alpha \alpha }^{\eta /2}} \\
&&\times \sqrt{E_{\alpha \alpha }-\left( I^{\prime }+I_{\ast }^{\prime
}\right) }d\boldsymbol{\xi }_{\ast }dI_{\ast }dI^{\prime }dI_{\ast }^{\prime
} \\
&\geq &C\int_{\mathbb{R}^{3}\times \mathbb{R}_{+}}\dfrac{E_{\alpha \alpha
}^{1/2}e^{-m_{\alpha }\left\vert \boldsymbol{\xi }_{\ast }\right\vert ^{2}/2}%
}{\mathcal{E}_{\alpha \alpha }^{\delta ^{\left( \alpha \right) }}E_{\alpha
\alpha }^{\eta /2}}I_{\ast }^{\delta ^{\left( \alpha \right)
}/2-1}e^{-I_{\ast }} \\
&&\times \left( \int_{0}^{E_{\alpha \alpha }/4}\left( I^{\prime }\right)
^{\delta ^{\left( \alpha \right) }/2-1}dI^{\prime }\right) ^{2}\,d%
\boldsymbol{\xi }_{\ast }dI_{\ast } \\
&=&C\int_{\mathbb{R}^{3}\times \mathbb{R}_{+}}E_{\alpha \alpha }^{\left(
1-\eta \right) /2}e^{-m_{\alpha }\left\vert \boldsymbol{\xi }_{\ast
}\right\vert ^{2}/2}I_{\ast }^{\delta /2-1}e^{-I_{\ast }}\,d\boldsymbol{\xi }%
_{\ast }dI_{\ast } \\
&\geq &C\int_{\mathbb{R}^{3}}\left( \left\vert \mathbf{g}\right\vert
^{2}+I\right) ^{\left( 1-\eta \right) /2}e^{-m_{\alpha }\left\vert 
\boldsymbol{\xi }_{\ast }\right\vert ^{2}/2}\,d\boldsymbol{\xi }_{\ast
}\int_{0}^{\infty }I_{\ast }^{\delta /2-1}e^{-I_{\ast }}\,dI_{\ast } \\
&\geq &C\int_{\mathbb{R}^{3}}\left( \left\vert \left\vert \boldsymbol{\xi }%
\right\vert -\left\vert \boldsymbol{\xi }_{\ast }\right\vert \right\vert
^{2}+I\right) ^{\left( 1-\eta \right) /2}e^{-m_{\alpha }\left\vert 
\boldsymbol{\xi }_{\ast }\right\vert ^{2}/2}\,d\boldsymbol{\xi }_{\ast }%
\text{.}
\end{eqnarray*}%
Now it follows that 
\begin{eqnarray*}
\nu _{\alpha } &\geq &C\int_{\mathbb{R}^{3}}\left( \left( \left\vert 
\boldsymbol{\xi }\right\vert -\left\vert \boldsymbol{\xi }_{\ast
}\right\vert \right) ^{2}+I\right) ^{\left( 1-\eta \right) /2}e^{-m_{\alpha
}\left\vert \boldsymbol{\xi }_{\ast }\right\vert ^{2}/2} \\
&&\times \left( \mathbf{1}_{\left\vert \boldsymbol{\xi }_{\ast }\right\vert
\leq 1/2}\mathbf{1}_{\left\vert \boldsymbol{\xi }\right\vert \geq 1}+\mathbf{%
1}_{\left\vert \boldsymbol{\xi }_{\ast }\right\vert \geq 2}\mathbf{1}%
_{\left\vert \boldsymbol{\xi }\right\vert \leq 1}\right) \,d\boldsymbol{\xi }%
_{\ast } \\
&\geq &C\left( \left( \left\vert \boldsymbol{\xi }\right\vert ^{2}+I\right)
^{\left( 1-\eta \right) /2}\mathbf{1}_{\left\vert \boldsymbol{\xi }%
\right\vert \geq 1}\int_{\left\vert \boldsymbol{\xi }_{\ast }\right\vert
\leq 1/2}e^{-m_{\alpha }\left\vert \boldsymbol{\xi }_{\ast }\right\vert
^{2}/2}\,d\boldsymbol{\xi }_{\ast }\right. \\
&&\left. +\left( 1+I\right) ^{\left( 1-\eta \right) /2}\mathbf{1}%
_{\left\vert \boldsymbol{\xi }\right\vert \leq 1}\int_{\left\vert 
\boldsymbol{\xi }_{\ast }\right\vert \geq 2}e^{-m_{\alpha }\left\vert 
\boldsymbol{\xi }_{\ast }\right\vert ^{2}/2}\,d\boldsymbol{\xi }_{\ast
}\right) \\
&\geq &C\left( \left( \left\vert \boldsymbol{\xi }\right\vert ^{2}+I\right)
^{\left( 1-\eta \right) /2}\mathbf{1}_{\left\vert \boldsymbol{\xi }%
\right\vert \geq 1}+(1+I)^{\left( 1-\eta \right) /2}\mathbf{1}_{\left\vert 
\boldsymbol{\xi }\right\vert \leq 1}\right) \\
&\geq &C\left( 1+\left\vert \boldsymbol{\xi }\right\vert ^{2}+I\right)
^{\left( 1-\eta \right) /2} \\
&\geq &C\left( 1+\left\vert \boldsymbol{\xi }\right\vert +\sqrt{I}\right)
^{1-\eta }\text{.}
\end{eqnarray*}%
Hence, there is a positive constant $\upsilon _{-}>0$, such that 
\begin{equation*}
\nu _{\alpha }\geq \nu _{-}\left( 1+\left\vert \boldsymbol{\xi }\right\vert +%
\sqrt{I}\right) ^{1-\eta }
\end{equation*}%
for all $\alpha \in \left\{ 1,...,s\right\} $ and $\boldsymbol{\xi }\in 
\mathbb{R}^{3}$.

On the other hand,%
\begin{eqnarray*}
\nu _{\alpha } &\leq &C\sum\limits_{\beta =1}^{s}\int_{\mathbb{R}^{3}\times
\left( \mathbb{R}_{+}\right) ^{3}}e^{-m_{\beta }\left\vert \boldsymbol{\xi }%
_{\ast }\right\vert ^{2}/2}e^{-I_{\ast }}E_{\alpha \beta }^{\left( 1-\eta
\right) /2} \\
&&\times \frac{\left( I^{\prime }\right) ^{\delta ^{\left( \alpha \right)
}/2-1}\left( I_{\ast }I_{\ast }^{\prime }\right) ^{\delta ^{\left( \beta
\right) }/2-1}}{\mathcal{E}_{\alpha \beta }^{\delta ^{\left( \alpha \right)
}/2}\left( \mathcal{E}_{\alpha \beta }^{\ast }\right) ^{\delta ^{\left(
\beta \right) }/2}}\mathbf{1}_{I^{\prime }\leq \mathcal{E}_{\alpha \beta }}%
\mathbf{1}_{I_{\ast }^{\prime }\leq \mathcal{E}_{\alpha \beta }^{\ast }}d%
\boldsymbol{\xi }_{\ast }dI_{\ast }dI^{\prime }dI_{\ast }^{\prime } \\
&=&C\sum\limits_{\beta =1}^{s}\int_{\mathbb{R}^{3}\times \mathbb{R}%
_{+}}e^{-m_{\beta }\left\vert \boldsymbol{\xi }_{\ast }\right\vert
^{2}/2}e^{-I_{\ast }}E_{\alpha \beta }^{\left( 1-\eta \right) /2}\int_{0}^{%
\mathcal{E}_{\alpha \beta }}\frac{\left( I^{\prime }\right) ^{\delta
^{\left( \alpha \right) }/2-1}}{\mathcal{E}_{\alpha \beta }^{\delta ^{\left(
\alpha \right) }/2}}dI^{\prime } \\
&&\times \int_{0}^{\mathcal{E}_{\alpha \beta }^{\ast }}\frac{\left( I_{\ast
}^{\prime }\right) ^{\delta ^{\left( \beta \right) }/2-1}}{\left( \mathcal{E}%
_{\alpha \beta }^{\ast }\right) ^{\delta ^{\left( \beta \right) }/2}}%
dI_{\ast }^{\prime }dI_{\ast }d\boldsymbol{\xi }_{\ast } \\
&=&C\sum\limits_{\beta =1}^{s}\int_{\mathbb{R}^{3}\times \mathbb{R}%
_{+}}e^{-m_{\beta }\left\vert \boldsymbol{\xi }_{\ast }\right\vert
^{2}/2}e^{-I_{\ast }}E_{\alpha \beta }^{\left( 1-\eta \right) /2}dI_{\ast }d%
\boldsymbol{\xi }_{\ast }
\end{eqnarray*}%
implying, since, clearly, $E_{\alpha \beta }\leq \left( 1+\mu_{\alpha \beta }\left\vert \mathbf{g}%
\right\vert ^{2}/2+I\right) \left( 1+I_{\ast }\right) $, that 
\begin{eqnarray*}
\nu _{\alpha } &\leq &C\sum\limits_{\beta =1}^{s}\int_{\mathbb{R}^{3}}\left(
1+\frac{\mu_{\alpha \beta }}{2}%
\left\vert \mathbf{g}\right\vert ^{2}+I\right) ^{\left( 1-\eta \right)
/2}e^{-m_{\beta }\left\vert \boldsymbol{\xi }_{\ast }\right\vert ^{2}/2}d%
\boldsymbol{\xi }_{\ast } \\
&&\times \int_{0}^{\infty }\left( 1+I_{\ast }\right) ^{\left( 1-\eta \right)
/2}I_{\ast }^{\delta ^{\left( \beta \right) }/2-1}e^{-I_{\ast }}dI_{\ast } \\
&\leq &C\left( 1+\left\vert \boldsymbol{\xi }\right\vert ^{2}+I\right)
^{\left( 1-\eta \right) /2}\int_{\mathbb{R}^{3}}\left( 1+\left\vert 
\boldsymbol{\xi }_{\ast }\right\vert ^{2}\right) ^{\left( 1-\eta \right)
/2}e^{-m_{\beta }\left\vert \boldsymbol{\xi }_{\ast }\right\vert ^{2}/2}d%
\boldsymbol{\xi }_{\ast } \\
&\leq &C\left( 1+\left\vert \boldsymbol{\xi }\right\vert +\sqrt{I}\right)
^{1-\eta }\int_{0}^{\infty }\left( 1+r^{2}\right) ^{\left( 1-\eta
+\varepsilon \right) /2}r^{2}e^{-m_{\beta }r^{2}/2}dr \\
&=&C\left( 1+\left\vert \boldsymbol{\xi }\right\vert +\sqrt{I}\right)
^{1-\eta }\text{.}
\end{eqnarray*}%
Hence, there is a positive constant $\nu _{+}>0$, such that 
\begin{equation*}
\nu _{\alpha }\leq \nu _{+}\left( 1+\left\vert \boldsymbol{\xi }\right\vert +%
\sqrt{I}\right) ^{1-\eta }
\end{equation*}%
for all $\alpha \in \left\{ 1,...,s\right\} $ and $\boldsymbol{\xi }\in 
\mathbb{R}^{3}$.
\end{proof}

\end{document}